\documentclass[10pt]{article}

\pdfoutput=1

\usepackage{fullpage}

\usepackage[latin1]{inputenc}
\usepackage{mystyle}

\graphicspath{{./figures/}}

\newcommand{\myauthor}[3]{#1\protect\footnote{#2, \protect\url{#3}}}
\title{Multi-dimensional Sparse Super-resolution}

\author{\myauthor{Clarice Poon}{DAMTP, University of Cambridge}{C.M.H.S.Poon@maths.cam.ac.uk}, \quad
		\myauthor{Gabriel Peyr{\'e}}{CNRS \& DMA, \'Ecole Normale Sup\'erieure}{gabriel.peyre@ens.fr}}

\date{\today}

\begin{document}

\maketitle


\begin{abstract} 
This paper studies sparse super-resolution in arbitrary dimensions. More precisely, it develops a theoretical analysis of support recovery for the so-called BLASSO method, which is an off-the-grid generalisation of $\ell^1$ regularization (also known as the LASSO). 
While super-resolution is of paramount importance in overcoming the limitations of many imaging devices, its theoretical analysis is still lacking beyond the 1-dimensional (1-D) case. The reason is that in the 2-dimensional (2-D) case and beyond, the relative position of the spikes enters the picture, and different geometrical configurations lead to different stability properties.  
Our first main contribution is a connection, in the limit where the spikes cluster at a given point, between solutions of the dual of the BLASSO problem and Hermite polynomial interpolation ideals. Polynomial bases for these ideals, introduced by De Boor, can be computed by Gaussian elimination, and lead to an algorithmic description of limiting solutions to the dual problem.
With this construction at hand, our second main contribution is a detailed analysis of the support stability and super-resolution effect in the case of a pair of spikes. This includes in particular a sharp analysis of how the signal-to-noise ratio should scale with respect to the separation distance between the spikes.
Lastly, numerical simulations on different classes of kernels show the applicability of this theory and highlight the richness of super-resolution in 2-D. 
\end{abstract}



\section{Introduction}

Sparse super-resolution is a fundamental problem of imaging sciences, where one seeks to recover the positions and amplitudes of pointwise sources (so-called ``spikes'') from linear measurements against \textit{smooth} functions. A typical example is deconvolution, where the measurements correspond to a low-pass filtering (or equivalently the observation of low-frequency Fourier coefficients), and is at the heart of fluorescent microscopy techniques such as PALM and STORM~\cite{betzig2006imaging,rust2006sub}. More generally, the measurements need not to be translation-invariant, and this is for instance the case in MEG/EEG~\cite{baillet2001electromagnetic} where the location of pointwise activity sources is crucial~\cite{gramfort2013time}.

The fundamental question in all these fields is to understand the super-resolution limit (often called the ``Rayleigh limit'') of some  computational method. This corresponds, for a given signal-to-noise ratio, to the minimum allowable separation limit between spikes so that their locations can be estimated. Certifying whether (or not) this limit goes to zero as the noise level drops, and at which speed, is a difficult problem, which until now, has mostly been addressed in 1-D. These questions are even more involved in higher-dimensions, because the geometry of the spikes configurations is much richer, and in contrary to the 1-D setting, these geometric configurations (e.g. whether 3 spikes are aligned or not) are expected to impact the super-resolution ability. It is the purpose of this paper to shed some light on higher-dimensional super-resolution phenomena.

\subsection{BLASSO and Super-resolution}
\label{sec-blasso}

In this work, we focus on inverse problems for arbitrary dimension $d$. To this end, we consider the underlying domain to be $\Xx=\RR^d$ or $\Xx=\RR^d/\ZZ^d$ the $d$-dimensional torus. 
The sparse recovery methods we consider are framed as optimization problems on the Banach space $\Mm(\Xx)$ of bounded Radon measures on $X$.
The unknown data to recover is thus a measure $m_0 \in \Mm(\Xx)$, and one has access to linear measurements $y$ in some Hilbert space $\Hh$, with
\eql{\label{eq-fwd-model}
	y = \Phi m_0 + w \in \Hh
}
where $w \in \Hh$ models the acquisition noise, and $\Phi : \Mm(\Xx) \rightarrow \Hh$ is an operator of the form
\eq{
	\Phi: m  \in \Mm(\Xx) \mapsto \int_\Xx \varphi(x) \d m(x),
}	
 defined through its kernel $\phi: x \in \Xx \mapsto \phi(x) \in \Hh$. This kernel is assumed to be smooth, and it is indeed this smoothness that makes the operator ``coherent'' and ill-posed, thus requiring some sort of regularization technique to invert the forward imaging model~\eqref{eq-fwd-model}.

A typical example is deconvolution, which corresponds to the translation invariant problem, where $\phi(x)=\psi(x-\cdot)$ for some $\psi \in L^2(\Xx)$ and $\Hh=L^2(\Xx)$. On $\Xx=\RR^d/\ZZ^d$, when $\psi$ has a finitely supported Fourier spectrum $(\hat \psi(\om))_{\om \in \Om}$, up to a rescaling of the measurement, this is equivalent to consider directly a sampling of the Fourier frequencies, i.e. $\phi(x)=( e^{2\imath\pi\dotp{x}{\om}} )_{\om \in \Om}$ and $\Hh=\CC^{|\Om|}$. 
Non-translation invariant operators also ubiquitous in imaging, for instance non-stationary blurs or indirect observation on the boundary of the domain as for instance in EEG/MEG. Denoting $S \subset \Xx$ the domain of interest (for instance the boundary of the skull for brain imaging), these techniques can be modelled as a sub-sampled convolution, i.e. $\phi(x)=( \psi(x-z) )_{z \in \partial S}$, where typically $\psi$ is some kernel associated to a stationary electric or magnetic field. Another example of non-convolution problems routinely encountered in imaging is the Laplace transform $\phi(x)=e^{-\dotp{x}{\cdot}}$. Lastly, let us mention the problem of mixture estimation, which can also be framed as a multi-dimensional sparse recovery problem~\cite{gribonval2017compressive}.

This work focusses on the recovery of a superposition of point sources, which are measures of the form $m= \sum_{i=1}^N a_i \delta_{z_i}$, where $a \delta_z$ is a Diracs (or a ``spike'') at position $z \in X$ and with amplitude $a \in \RR$. 
Following several recent works (reviewed in Section~\ref{sec-pw} below), we regularize using the total variation norm $\abs{m}(\Xx)$ of $m \in \Mm(\Xx)$, which is defined as
\eq{
	\abs{m}(\Xx) \eqdef \sup\enscond{\int_\Xx \eta(x) \mathrm{d}m(x)}{ \eta\in C(\Xx), \; \norm{\eta}_{L^\infty} \leq 1}.
}
This corresponds to the generalization of the $\ell^1$ and $L^1$ norms of vectors and functions, since one has $|\sum_i a_i \delta_{z_i}|(\Xx) = \norm{a}_{\ell^1}$ and if $m$ has a density $f=\frac{\d m}{\d x}$ with respect to the Lebesgue measure $\d x$, then $|m|(\Xx)=\norm{f}_{L^1(\d x)}$. 

We thus consider a least-squares total-variation regularization 
\eql{\label{eq-blasso}\tag{$\Pp_\la(y)$}
	\uargmin{m\in \Mm(\Xx)}  \abs{m}(\Xx) + \frac{1}{2\la} \norm{\Phi m - y}_\Hh^2, 
}
where $\la>0$ is the regularization parameter, which should adapted to the noise level $\norm{w}$ (theoretical results such as ours advocating a linear scaling between $\la$ and $\norm{w}_\Hh$).
In the case of noiseless measurements $y$ (i.e. $w=0$), the limit of~\eqref{eq-blasso} as $\la \rightarrow 0$ is the constrained problem
\eql{\label{eq-blasso-noiseless}\tag{$\Pp_0(y)$}
	\uargmin{m \in \Mm(\Xx)} \abs{m}(\Xx) \text{ subject to } \Phi m = y
}
The purpose of this paper is to study the structure (and in particular the support) of the solutions to~\eqref{eq-blasso} when $\norm{w}$ and $\la$ are small, and the positions of the spikes $(z_i)_i$ of the sought after measure $m_0$ cluster around a fixed point. 


\subsection{Previous Works}
\label{sec-pw}

\paragraph{On-the-grid LASSO}

There is a long history on the use of convex methods, and in particular the $\ell^1$ norm, to recover sparse signals on a grid. This was pioneered by geophysicists~\cite{claerbout-robust1973,levy-reconstruction1981,santosa-linear1986} and then became mainstream in signal processing~\cite{chen-atomic1998} and for model selection in statistics~\cite{tibshirani1996regression}. 
General theoretical approaches (only constraining the sparsity of the signal), such as those used to study compressed sensing~\cite{donoho2006compressed,candes2006robust}, are however ineffective for super-resolution, where the linear measurement operator is highly coherent. The theory of super-resolution requires to impose constraints on the minimum separation distance between the spikes in place of the total number of those spikes. 

\paragraph{Off-the-grid BLASSO, minimum-separation condition.}

In order to avoid introducing a-priori a fixed grid, which might be detrimental both theoretically and computationally, it makes sense to consider the ``off-the-grid'' setting introduced in the previous section~\ref{sec-blasso}, as exposed by several authors, including~\cite{deCastro-exact2012,bredies-inverse2013,candes-towards2013}.
The ground breaking work of Cand\`es and Fernandez-Granda~\cite{candes-towards2013}, for the low-pass filter, proved that under a $\Oo(1/f_c)$ minimum separation condition between the spikes, exact recovery is achieved. This initial result has been extended to include noise stability~\cite{candes-superresolution2013,azais-spike2014}. 
Further refinements have been proposed, for instance statistical bounds~\cite{bhaskar2013atomic} and exact support recovery condition~\cite{duval2015exact} for more general (not necessarily translation invariant) measurements. 
This line of theoretical study works in arbitrary dimensions, but this does not corresponds to a super-resolution regime ($O(1/f_c)$ being often considered as the natural Rayleigh). For spikes with arbitary signs, it is easy to construct counter example below the Rayleigh limit where the BLASSO does not work even when there is no noise (this is to be contrasted with Prony-type approaches, see below).

\paragraph{BLASSO, positive spikes in 1-D.}

Going below the Rayleigh limit requires an analysis of the signal-to-noise scaling, i.e. at which speed the signal-to-noise ratio SNR should drop to zero as a function of the spikes separation distance. This scaling is typically polynomial, and the exponent depends on the number of spikes that cluster around a given location. 
The study of this scaling was initiated by \cite{donoho1992superresolution}, for the combinatorial search method (thus not numerically tractable). With such non-convex methods, the optimal scaling is obtained with no constraint on the sign pattern of the spikes, see also~\cite{demanet-recoverability2014} for a refined analysis.

As noticed above, BLASSO~\eqref{eq-blasso} in general cannot go below the Rayleigh limit for measures with arbitrary signs.
For positive spikes, BLASSO achieves super-resolution for some classes of measurement operators $\Phi$. This was initially studied for the LASSO problem (on a grid)~\cite{donoho1992maximum,fuchs2005sparsity}, and for the BLASSO, this holds for low-pass Fourier measurements and polynomial moments~\cite{deCastro-exact2012}. This exact recovery of positive spikes can also cope with sub-sampling~\cite{schiebinger2015superresolution}, but this is restricted to the 1-D setting.

Stability to noise for the LASSO (on discrete grid) is studied in~\cite{morgenshtern2016super}, with a signal-to-noise scaling $\norm{w} = O(t^{2N})$ where $N$ is the number of spikes clustered in a radius of $t$. This result holds in 1-D and 2-D, see also~\cite{bendory2017robust}. 
In the 1-D BLASSO case, a signal-to-noise scaling $\norm{w} = O(t^{2N-1})$ actually leads to exact support recovery under a non-degeneracy condition (which is true for the Gaussian convolution kernel)~\cite{2017-denoyelle-jafa}. Our work proposes extensions of this last result in arbitrary dimensions.

\paragraph{BLASSO, positive spikes in multiple-dimension.}

Only few theoretical works have studied sparse super-resolution in more than one dimension. Let us single out the work of~\cite{morgenshtern2016super}, which proves that 2-D super resolution on a discrete grid has a similar signal-to-noise scaling with spikes separation distance $t$ as in 1-D. 
This however does not provide an understanding of how super-resolution and support recovery operates in multiple dimensions, which in turn requires to work off-the-grid, using the BLASSO, as we do in this paper. 
Let us also note the work of \cite{shahram2006statistical} which studies scaling of statistical decision-theoretic bounds for pairs of spikes. This is inline with our study in Section~\ref{sec-two-spikes}, which shows that the BLASSO achieves the same signal-to-noise scaling.

\paragraph{Prony's type methods.}

While this paper is dedicated to convex $\ell^1$-type methods, there is a large body of methods and analysis that use non-convex or non-variational approaches. These methods are very often generalizations of the initial idea of Prony~\cite{prony1795essai} which encodes the spikes positions as the zeros of some polynomial, whose coefficients are derived from the measurements, see~\cite{stoica2005spectral} and the review paper~\cite{krim-two1996}. Let us for instance cite MUSIC~\cite{schmidt-multiple1986}, matrix pencil~\cite{hua1990matrix}, ESPRIT~\cite{roy1989esprit}, finite rate of innovation~\cite{vetterli-sparse2008}, Cadzow's denoising~\cite{cadzow1988signal,condat2015cadzow}. 

It is not the purpose of this paper to advocate for or against the use of convex methods, and there are pros-and-cons both in term of both practical performances and theoretical understanding. An important advantage of Prony based approaches is that, in the noiseless setting $w=0$, they achieve exact recovery without any condition on the sign of the spikes (whereas BLASSO requires either a minimum separation distance or positivity). 
The theoretical analysis of these approaches in the presence of noise is however more intricate, and only partial results exist. Cramer-Rao statistical bounds can be derived~\cite{clergeot1989performance} and non-asymptotic bounds have been proposed under minimum-separation conditions~\cite{liao-music2014,moitra2014threshold}.
A second difficulty with Prony based approaches is that they are non-trivial to extend to higher dimensions, and there is no general agreement on a canonical formulation even in 2-D. 
We refer for instance to~\cite{peterreconstruction,kunis2016multivariate,sauer2017prony,sacchini1993two,clark1994two,andersson2017espirit} and the references therein for several such extensions.
Furthermore, in the noiseless setting, in contrary to the 1-D setting, the number of recovered spikes does not scale linearly with the number of measurements~\cite{jiang2001almost}. 

\paragraph{Numerical Methods for the BLASSO.}

While it is not the purpose of this paper to develop numerical solvers, let us sketch some pointers to solvers for the BLASSO problem.
The BLASSO is computationally challenging because it is an infinite dimensional optimization problem. 
The most straightforward approach is to approximate the problem on a grid, which then becomes a finite dimensional LASSO (which itself is a linear program). This however leads to quantization artifacts, typically doubling the number of observed spikes in 1-D~\cite{2017-Duval-IP-lasso}. 

In the case of a finite number of Fourier frequencies, the dual optimization problem is finite dimensional. In 1-D, it can be solved exactly by lifting to a semi-definite program in $O(f_c^2)$ variables~\cite{candes-towards2013}. In 2-D and on more general semi-algebraic domains, this lifting is more involved numerically, and requires the use of a Lasserre hierarchy~\cite{de2017exact}.
For general measurements, \cite{bredies-inverse2013} proposed to use the Frank-Wolfe algorithm, which operates by adding in a greedy-manner new spikes, and is a convex counterpart of the celebrated matching pursuit algorithm~\cite{mallat1993matching} over a continuous dictionary~\cite{jacques2008geometrical}. It has a slow convergence rate, which is improved by interleaving non-convex optimization steps, also used in~\cite{boyd2017alternating}. This algorithm is in practice very efficient, and leads to state of the art result in 2-D and 3-D resolution, for instance with application to single-molecule imaging~\cite{boyd2017alternating}.

\subsection{Contributions}

We begin Section~\ref{sec-certif} with a recap on the link between support stability and the dual solutions of \eqref{eq-blasso}.  As already noted for instance in~\cite{duval2015exact}, support stability of~\eqref{eq-blasso} when $(w,\la)$ are small, is governed by a minimal norm dual solution associated with \eqref{eq-blasso}, and as shown in \cite{2017-denoyelle-jafa} for the 1-D case, analysis of the limit (as the minimum separation distance $t$ tends to 0) of these minimum norm dual solutions lead to a clear understanding of the behaviour of support stability. As a contribution, we also provide a closed form expression for this limiting certificate in the case of the ideal low pass filter with $N=f_c$, thus complementing some of the numerical observations from \cite{2017-denoyelle-jafa}.

 Our first main contribution is detailed in Section~\ref{sec-limiting}, where we provide a characterisation of the limiting certificate in the multi-dimensional setting. Furthermore, we provide a closed form expression for this limiting certificate in the case of the Gaussian filter. Due to the necessity of this certificate, our analysis thus sheds light on the behaviour of support stability in arbitrary dimensions.

Our second main contribution, detailed in Section~\ref{sec-two-spikes}, is a detailed analysis of the structure of the solution to~\eqref{eq-blasso} in the case of $N=2$ spikes. Under a non-degeneracy condition on the limiting dual certificate, we show that that the solution to the BLASSO is composed of two spikes, and we precisely characterise how the signal-to-noise should scale with the separation between the spikes for this recovery to hold.

Lastly, Section~\ref{sec-numerics} showcases numerical illustrations of these theoretical advances. The code to reproduce the results of this paper is available online\footnote{\url{https://github.com/gpeyre/2017-MSL-super-resolution}}.

\newcommand{\bdotp}[2]{\langle #1,\,#2\rangle_{\Bb}}

\section{Asymptotics of Dual Certificates}
\label{sec-certif}

\subsection{Vanishing Pre-certificate $\eta_{V,\Z}$}

A positive discrete measure $m_0 = \sum_i a_i \de_{z_i} \in \Mm(\Xx)$ is solution to~\eqref{eq-blasso-noiseless} if and only if the set of Lagrange multipliers of the constraint, often refered to in the literature as ``dual certificates''
\eq{
	\Dd(\Z) = \enscond{ \eta \in \Im(\Phi^*) }{ \norm{\eta}_\infty \leq 1, \forall i, \eta(z_i)=1 } 
}
is non-empty, where we denote by $\Z=(z_i)_{i=1}^N \in \Xx^N$ the spikes locations. Proving that $\Dd(\Z) \neq \emptyset$ is thus a Lagrange interpolation problem using continuous interpolating functions in $\Im(\Phi^*)$, and with the additional constraint that $\eta$ should be bounded by 1.
In the following, we say that a smooth function $\eta$ is non-degenerate for the positions $\Z$ if
\eql{\label{eq-cond-nd}\tag{$\text{ND}(\Z)$}
	\foralls x \notin \Z, \quad \eta(x) < 1, 
	\qandq
	\foralls i=1,\ldots,N, \quad  \nabla^2 \eta(z_i) \prec 0.
}
Here, we denote by $\nabla^2 \eta(x) \in \RR^{d \times d}$ the Hessian matrix of $\eta$ at $x \in \Xx$, and $A \prec 0$ means that the matrix $A \in \RR^{d \times d}$ is negative definite.
%
Throughout this paper, we shall use $\nabla^j$ to denote the full gradient of order $j$, for $\alpha \in \NN_0^d$, $\partial^\alpha \eqdef P(\partial)$ where $P(X) = X^\alpha$, and given $x, z\in \Xx$ and $f\in \Cder{j}(\Xx)$, let $\partial^j_z f(x) \eqdef \frac{\mathrm{d}^j}{\mathrm{d}t^j}\vert_{t=0} f(x+tz)$.
The condition \eqref{eq-cond-nd} is a strengthening of the condition of being a dual certificate, and is reminiscent of the non-degeneracy condition on the relative interior of the sub-differential which is standard in sensitivity analysis in finite dimension~\cite{bonnans2013perturbation} (recall that we are here dealing with infinite-dimensional optimization problems).

As initially shown in~\cite{duval2015exact}, the support stability of the solution to~\eqref{eq-blasso} with small $(w,\la)$ is governed by a specific dual certificate with minimum norm
\eq{
	\eta_{0,\Z} \eqdef \Phi^* p_{0,\Z}
	\qwhereq
	p_{0,\Z} \eqdef \uargmin{p \in \Hh} \enscond{ \norm{p}_\Hh }{ \Phi^* p \in \Dd(\Z) }. 
}
More precisely, if $\eta_{0,\Z}$ is non degenerated (i.e. satisfies~\eqref{eq-cond-nd}), then for $(\norm{w}/\la, \la)=O(1)$, then~\cite{duval2015exact} shows that the solution of~\eqref{eq-blasso} is unique and composed of $N$ spikes, whose positions and amplitudes converge smoothly toward $(a,\Z)$ as $\la \rightarrow 0$.   

Direct analysis of this $\eta_{0,\Z}$, and in particular proving that it satisfies condition~\eqref{eq-cond-nd}, is difficult, mainly because of the non-linear constraint $\norm{\eta_{0,\Z}}_\infty \leq 1$. Fortunately,~\cite{duval2015exact} introduces a simpler proxy, which is defined by replacing this constraint by the linear one of having vanishing derivatives at the spikes positions
\eql{\label{eq-defn-etaV}
	\eta_{V,\Z} \eqdef \Phi^* p_{V,\Z}
	\qwhereq
	p_{V,\Z} \eqdef \uargmin{p \in \Hh} \enscond{ \norm{p}_\Hh }{ \forall i, (\Phi^* p)(z_i)=1, \nabla(\Phi^* p)(z_i)=\zeros_d }
}
where $\nabla \eta(x) \in \RR^d$ is the gradient vector of $\eta$ at $x \in \Xx$.
The interest of this vanishing pre-certificate $\eta_{V,\Z}$ stems from the fact that 
\eq{
	\eta_{V,\Z} \text{ satisfies }~\eqref{eq-cond-nd}
	\quad\Longrightarrow\quad
	\eta_{0,\Z} \text{ satisfies }~\eqref{eq-cond-nd} \qandq 
	\eta_{0,\Z} = \eta_{V,\Z}.
}
(and the converse is also true), which means that one can simply check the non-degeneracy of $\eta_{V,\Z}$ to guarantee support stability in the small noise regime. 

\begin{rem}[Computation of $\eta_{V,\Z}$]
It is important to note that $p_{V,\Z}$ can be computed easily by solving a linear system of size $Q \times Q$ where $Q \eqdef (d+1)N$. Indeed, introducing the correlation kernel
\eql{\label{eq-etaV-corr}
	\foralls (x,x') \in \Xx^2, \quad \Corr(x,x') \eqdef \dotp{\phi(x)}{\phi(x')}_\Hh \in \RR, 
} 
one has that 
\eq{\label{eq-etaV-defn}
	\eta_{V,\Z}(x) = \sum_{i=1}^N \sum_{k=0}^d \al_{i,k} \partial_{1,k} \Corr(z_i,x)
	\qwhereq
	\al =  M^{-1} u_{d,N}
	\qandq
	M = ( \partial_{1,k}\partial_{2,\ell} \Corr(z_i,z_j) )_{i,j=1,\ldots,N}^{k,\ell=1,\ldots,d},
}
where we denoted $\partial_{1,k}\partial_{2,\ell} \Corr(x,x') \in \RR$ the derivative at $(x,x') \in \Xx^2$ with respect to the $k^{\text{th}}$ coordinate of $x$ and the $\ell^{\text{th}}$ coordinate of $x'$. The vector $u_{d,N} \in \RR^{Q}$ is defined by 
\eq{
	\foralls (i,k) \in \{1,\ldots,N\} \times \{0,\ldots,d\}, \quad
	(u_{d,N})_{i,k} \eqdef
	\choice{ 
		1 \qifq k=0, \\
		0 \quad\text{otherwise.}
	}
}
The matrix $M \in \RR^{Q \times Q}$ operates as $M\al = ( \sum_{j,\ell} M_{(i,k),(j,\ell)} \al_{j,\ell} )_{(i,k)}$.
\end{rem}

This simple expression~\eqref{eq-etaV-defn} for $\eta_{V,\Z}$, which only requires the inversion of a $Q \times Q$ linear system, has been used in 1-D to show that $\eta_{V,\Z}$ is non-degenerate, even for measure with arbitrary sign, under a minimum separation condition and for a class of convolution operator with fast decay (including the Cauchy kernel $\phi(x)=(1+(x-\cdot))^{-2} \in L^2(\RR)$), see~\cite{tang2013atomic}.

The main goal of this paper is to study the non-degeneracy of $\eta_{V,\Z}$ in multiple dimensions. 
Of particular interest for us is the case where the spike locations cluster around a point (which we set without loss of generality to be 0). So, we consider spike locations of the form $t\Z=(tz_i)_{i=1}^N$ where the scaling parameter $t>0$ controls the minimum separation distance between the spikes. An important question is to understand whether $\eta_{V,t\Z}$ converges as $t \rightarrow 0$ and to derive an explicit formula for this limit (which we denote by $\eta_{W,\Z}$).

\subsection{Asymptotic Pre-certificate in 1-D}

We now summarise the results of~\cite{2017-denoyelle-jafa}, which hold in the 1-D setting, $d=1$. First one has the convergence of $\eta_{V,t\Z}$, as $t \rightarrow 0$, towards
\eql{\label{eq-defn-etaW-1d}
	\eta_{W} \eqdef \Phi^* p_{W}
	\qwhereq
	p_{W} \eqdef \uargmin{p \in \Hh} \enscond{ \norm{p}_\Hh }{ (\Phi^* p)(0)=1, \forall s=1,\ldots,2N-1,  (\Phi^* p)^{(s)}(0) = 0  },  
}
where $\eta^{(k)}$ denotes the $k^{\text{th}}$ derivative of $\eta$.
Note in particular that the limit $\eta_W$ is \textit{independent} of $\Z$, which should be contrasted with the higher dimensional case, see Section~\ref{sec-asymp-highdim} below. 

The main result of~\cite{2017-denoyelle-jafa} is that, if $\eta_{W}$ is non-degenerate, in the sense that 
\eq{
	\foralls x \neq 0, \quad \eta_{W}(x) < 1, 
	\qandq
	\eta_{W}^{(2N)}(0) < 0
}
then for $t$ small enough, $\eta_{V,t\Z}$ is also non-degenerate, and one can compute a sharp estimate of the support stability constant involved as a function of $t$. More precisely, one should have $\norm{w}_\Hh/\la=O(1)$ and $\la=O(t^{2N-1})$ in order for~\eqref{eq-blasso} to recover the correct number $N$ of spikes and to have smoothly converging positions and amplitudes as $\la \rightarrow 0$.

\begin{rem}[Computation of $\eta_{W}$]
Similarly to~\eqref{eq-etaV-corr}, $\eta_{W}$ is conveniently computed by solving a finite dimensional  linear system of size $2N \times 2N$
\eql{\label{eq-etaw-1d-corr}
	\eta_{W}(x) = \sum_{s=0}^{2N-1} \be_{r} \partial_1^{(r)} \Corr(0,x)
	\qwhereq
	\be =  R^{-1} \de_{2N}
	\qandq
	R = ( \partial_1^{(r)} \partial_2^{(s)} \Corr(z_i,z_j) )_{r,s=0,\ldots,2N-1} \in \RR^{2N \times 2N},
}
where $\partial_1^{(r)}$ and $\partial_2^{(s)}$ the $r^{\text{th}}$ and $s^{\text{th}}$ order derivatives with respect to the first and second variables (with the convention $\partial_1^{(0)} \Corr= \Corr$), and $\de_{2N} \eqdef (1,0,\ldots,0)^* \in \RR^{2N}$. 
\end{rem}

The simple expression~\eqref{eq-etaw-1d-corr} allows easily to study $\eta_W$, and numerical computations shows that it is indeed non-degenerate for many low pass filters~\cite{2017-denoyelle-jafa}. It can even be computed in closed form in the case of the Gaussian filter, see also Section~\ref{sec-gaussian-closedform} below.

The following proposition, which is new and thus a contribution of our paper, shows that one can also compute $\eta_W$ in closed form in the case of an ideal low-pass filter for the special case $N=f_c$, and that both $\eta_{V,\Z}$ and $\eta_W$ are non-degenerate. The general case of $f_c \neq N$ is still an open problem.

\begin{thm}\label{thm-etaw-lowwpass}
Let  $\Xx = \ZZ/\RR$.
Let $\tilde \varphi_D(x) = \sum_{\abs{k}\leq f_c} e^{2\pi i k x}$ where $f_c\in\NN$ is the cutoff frequency. Let $\varphi(x) = \tilde \varphi_D(\cdot -x)$. If $f_c = N$, then $\eta_{V,\Z}(x)<1$ for all $x \notin z$, and one has 
\eq{
	\eta_W(x) = 1- C \sin^{2N}(\pi x) , 
}
for some $C>0$ (whose explicit form can be found in \eqref{eq:g}).
So in particular, $\eta_W(x)<1$ for all $x \in \Xx \setminus \{0\}$.
\end{thm}

The proof of this theorem can be found in Appendix~\ref{sec-proof-thm-etaw-lowwpass}.

\subsection{Asymptotic Pre-certificate in Arbitrary Dimension}
\label{sec-asymp-highdim}

We now aim at generalizing the expression~\eqref{eq-etaw-1d-corr} to the case $d>1$. 


\subsubsection{General Definition of Vanishing Pre-Certicates}

The first difficulty is that the limit of $\eta_{V,t\Z}$ as $t \rightarrow 0$, provided that it exists, depends on the direction of convergence $\Z \in \Xx^N$, as illustrated by Figure~\ref{fig-etav-conv}, and we will thus denote it $\eta_{W,\Z}$. Intuitively, the difficulty is to identify which derivatives of $\eta_{W,\Z}$ should vanish in the limit. They should span a space of dimension $Q \eqdef N(d+1)$, since this matches the number of constraints appearing in~\eqref{eq-defn-etaV}.

We denote $\Pi^d$ the space of polynomials in $d$ variables $(X_1,\ldots,X_d)$ and for a $d$-tuple $k=(k_1,\ldots,k_d)$, the associated monomial $X^k \eqdef X_1^{k_1}\ldots X_d^{k_d}$. 
To ease the description of these constraints, for a polynomial $P = \sum_{k} a_{k} X^{k} \in \Pi^d$, we denote $P(\partial)$ the differential operator 
\eq{
	P(\partial) \eqdef \sum_{k} a_{k} \partial_1^{k_1} \ldots \partial_d^{k_d} 
}
where $\partial_s$ is the derivative with respect to the $s^{\text{th}}$ variable.

\newcommand{\bSs}{\bar\Ss}

The construction of the limiting certificate requires to identify a linear subspace $\Ss_{z} \subset \Pi^d$ which encodes the vanishing derivative constraints (which should have dimension $Q$) and solve 
\eql{\label{eq-defn-etaW-multidim}
	\eta_{W,\Z} \eqdef \Phi^* p_{W,\Z}
	\qwhereq
	p_{W,\Z} \eqdef \uargmin{p \in \Hh} \enscond{ \norm{p}_\Hh }{ (\Phi^* p)(0)=1, \forall P \in \bSs_{z}, 
		( P(\partial)[ \Phi^* p ] )(0) = 0 },
}
where $\bSs \subset \Ss$ is the linear subspace of polynomials $P$ such that $P(0)=0$. 
We show in Section~\ref{sec-least-interpolant} below that indeed such a space $\Ss_{z}$ exists, and that it can be computed using a simple Gaussian elimination algorithm.

\begin{rem}[Computation of $\eta_{W}$]\label{rem-comp-etaW-multid}
Once again, $\eta_{W,\Z}$ defined by~\eqref{eq-defn-etaW-multidim} can be computed by solving a $Q \times Q$ linear system.
Indeed, the linear space $\Ss_{z}$ is described using a basis of polynomials 
\eq{
	\Ss_{z} = \Span \enscond{ P_r }{ r=0,\ldots, Q-1  }
}
to which we impose for notation convenience $P_0=1$ (so that $P_0(\partial)[\eta]=\eta$). 
Then, denoting the various derivatives of the covariance as
\eq{
	\foralls (r,s) \in \{0,\ldots,Q-1\}^2, \quad
	\Corr_{r,s} \eqdef P_r^{[1]}(\partial) P_s^{[2]}(\partial)[\Corr]
}
where here we have use the notations $P_r^{[1]}(\partial)$ and $P_s^{[2]}(\partial)$ to indicate whether the polynomial should be used to differentiate on the first variable $x$ or the second variable $x'$ of $\Corr(x,x')$ (in particular $\Corr_{0,0}=\Corr$), one has
\eql{\label{eq-etaw-highdim-corr}
	\eta_{W,\Z}(x) = \sum_{r=0}^{Q-1} \be_{r} \Corr_{r,0}(0,x)
	\qwhereq
	\be =  R^{-1} \de_{Q}
}
\eq{
	\qandq
	R = \Big( \Corr_{r,s}(0,0) \Big)_{r,s=0,\ldots,Q-1} \in \RR^{Q \times Q}.
}
\end{rem}

Note that in dimension $d=1$, the expressions~\eqref{eq-defn-etaW-multidim} and~\eqref{eq-etaw-highdim-corr} are equivalent to those already given in~\eqref{eq-defn-etaW-1d} and~\eqref{eq-etaw-1d-corr} when using the monomial basis $\Ss_{z} = \Span\{X_1^r\}_{r=0}^{2N-1}$, with $Q=2N$.

\newcommand{\MyFigEtaV}[1]{\includegraphics[width=.24\linewidth]{etav-conv/#1}}
\begin{figure}
\centering
\begin{tabular}{@{}c@{\hspace{1mm}}c@{\hspace{1mm}}c@{\hspace{1mm}}c@{}}
\MyFigEtaV{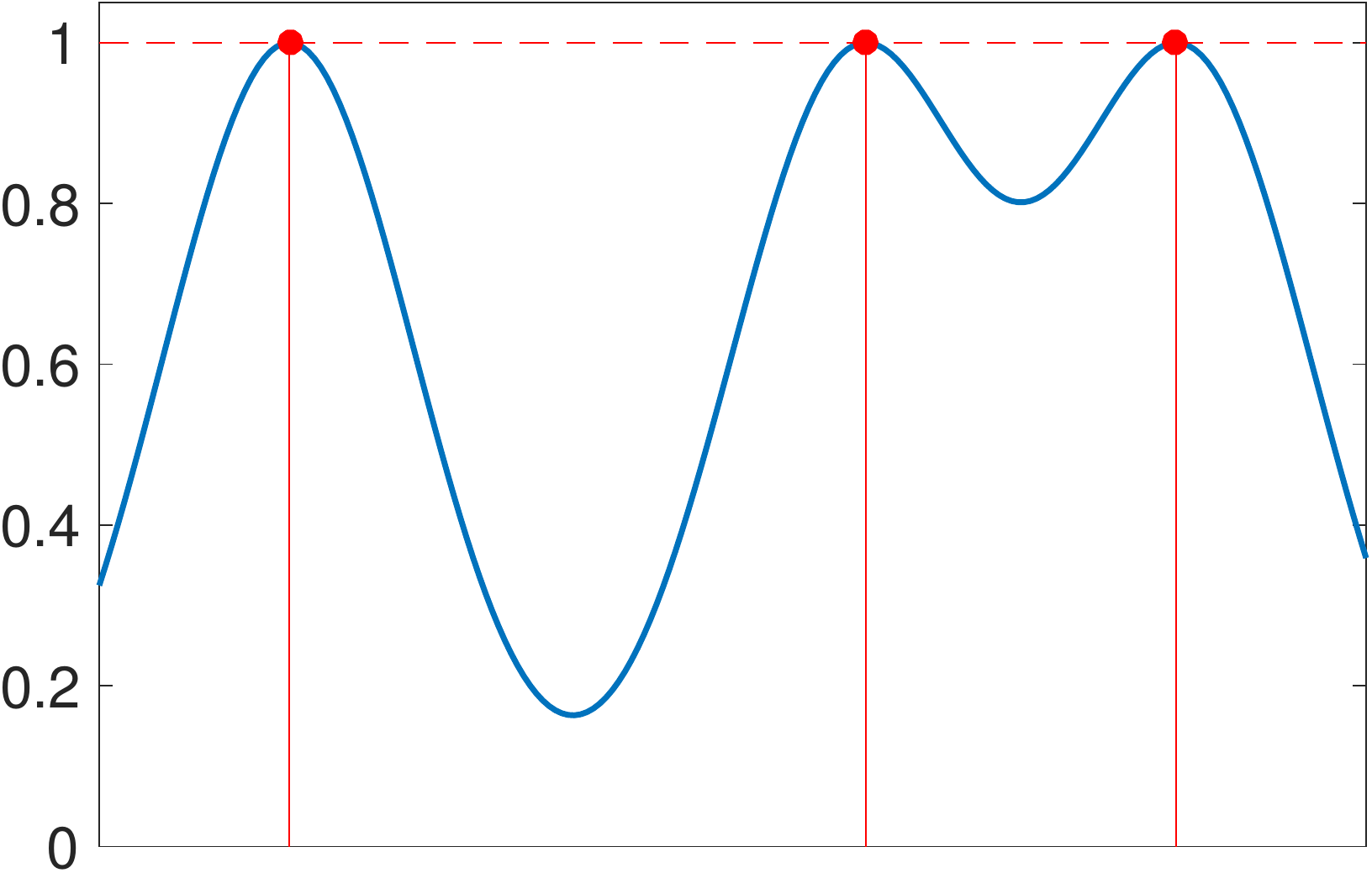}&
\MyFigEtaV{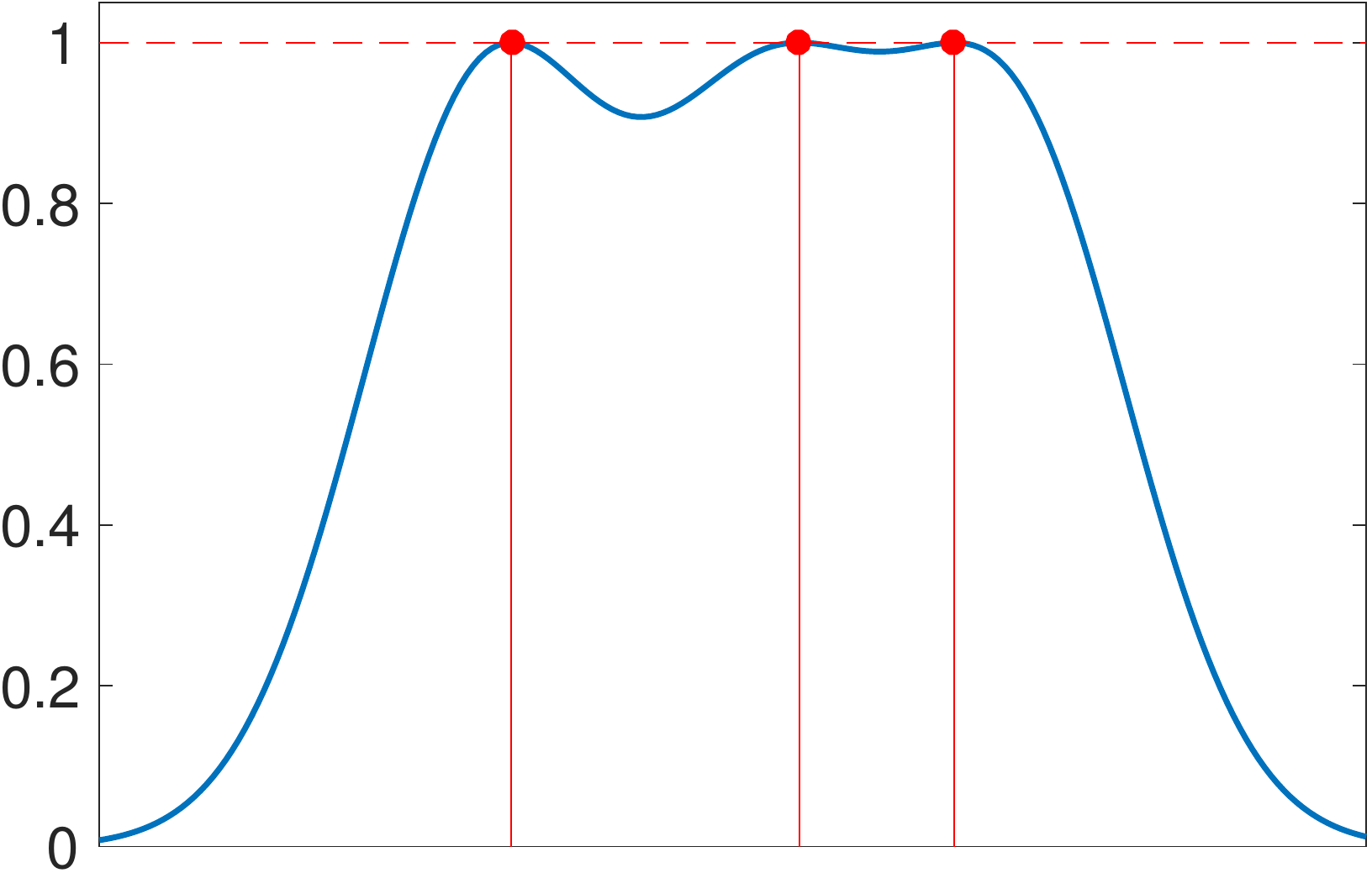}&
\MyFigEtaV{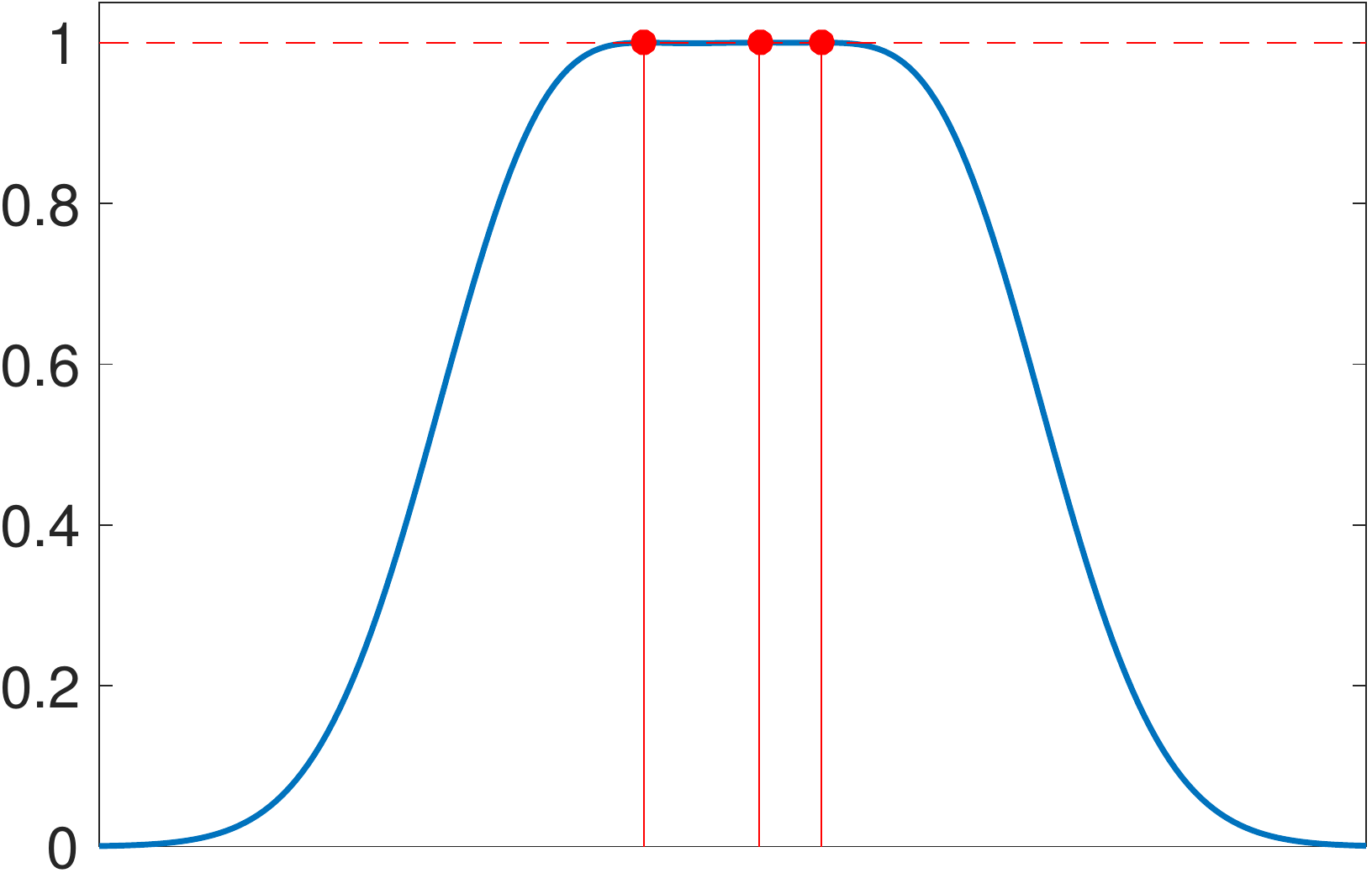}&
\MyFigEtaV{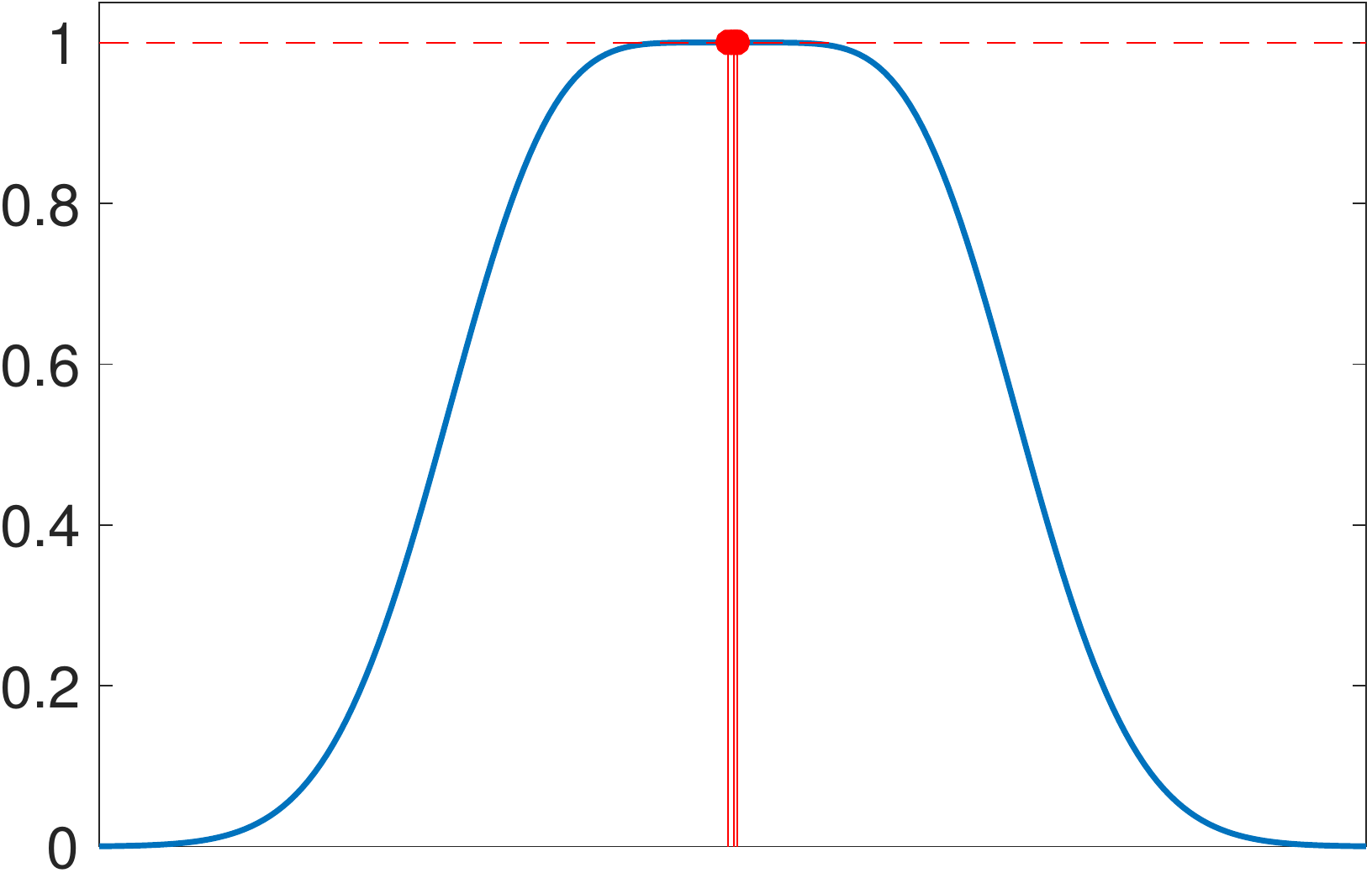}\\
\MyFigEtaV{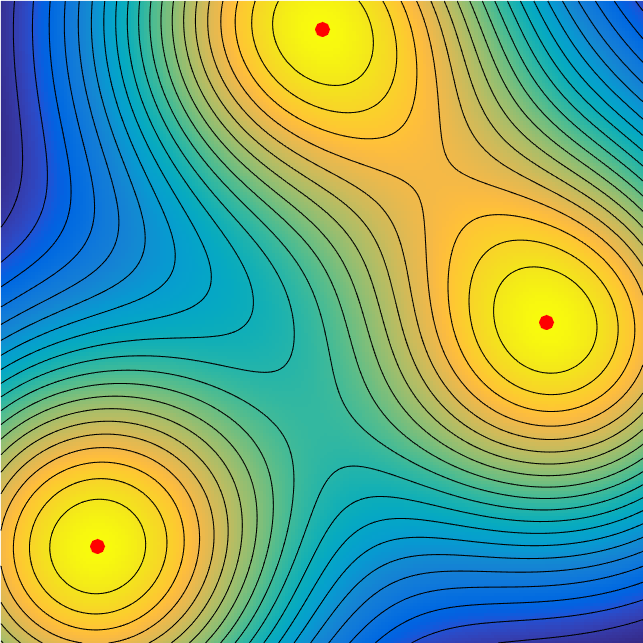}&
\MyFigEtaV{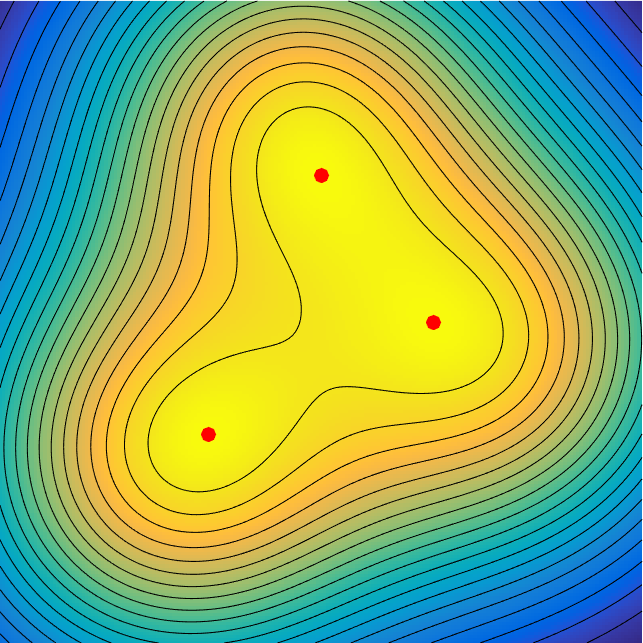}&
\MyFigEtaV{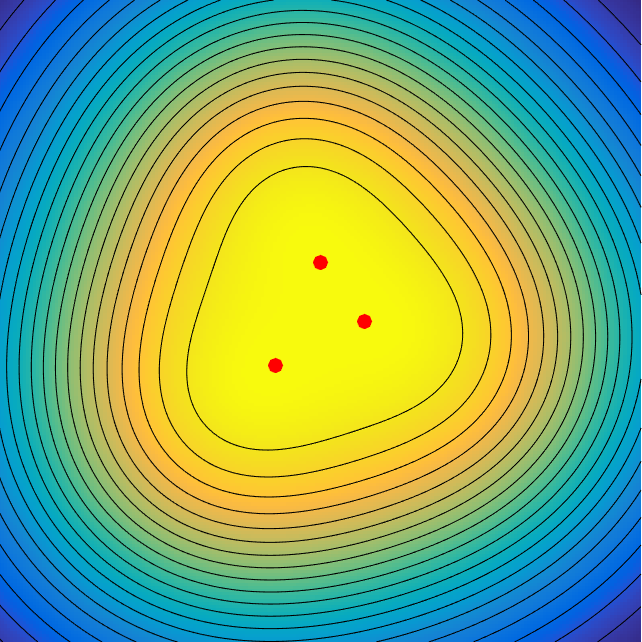}&
\MyFigEtaV{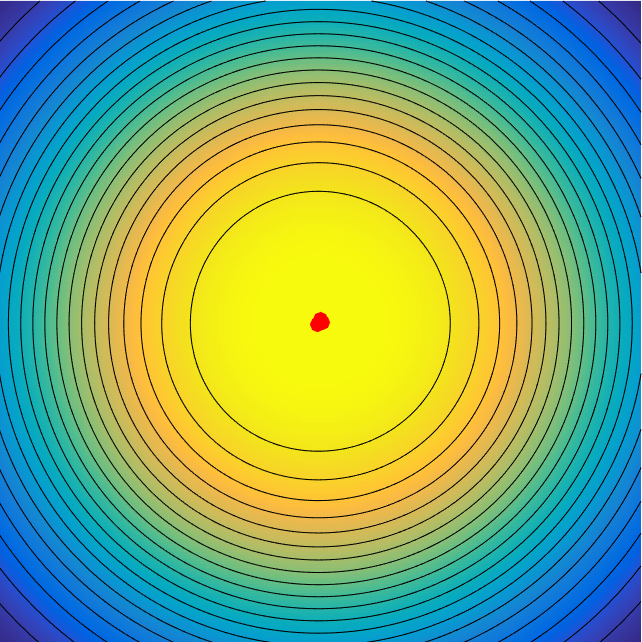}\\
\MyFigEtaV{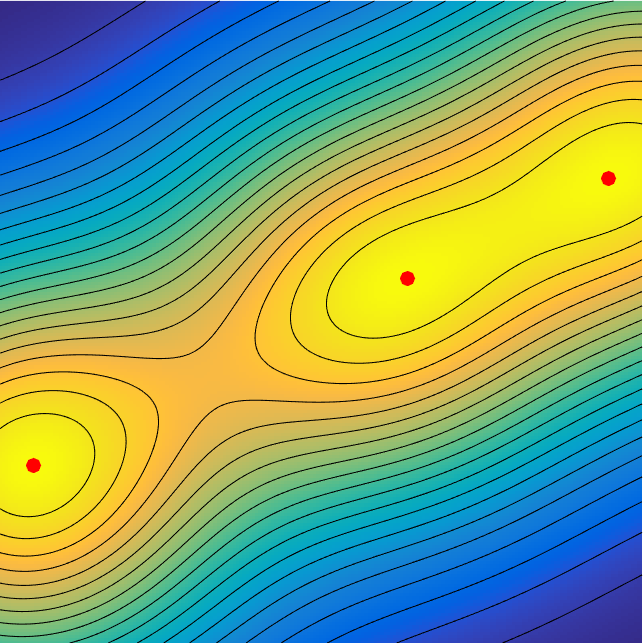}&
\MyFigEtaV{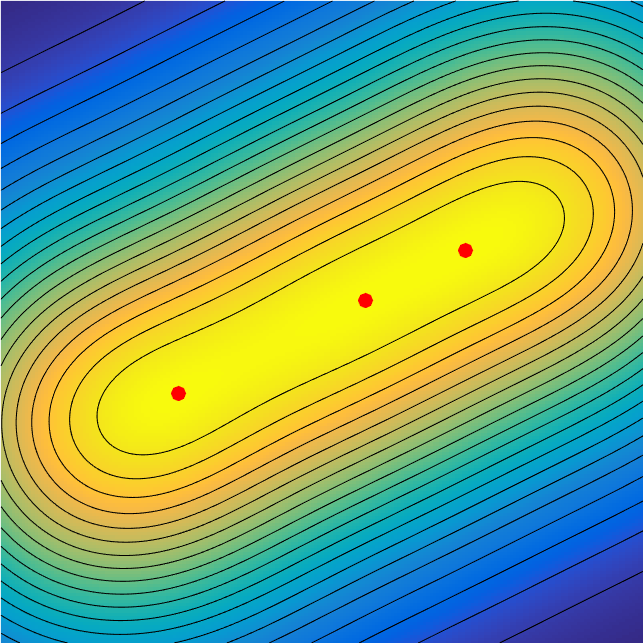}&
\MyFigEtaV{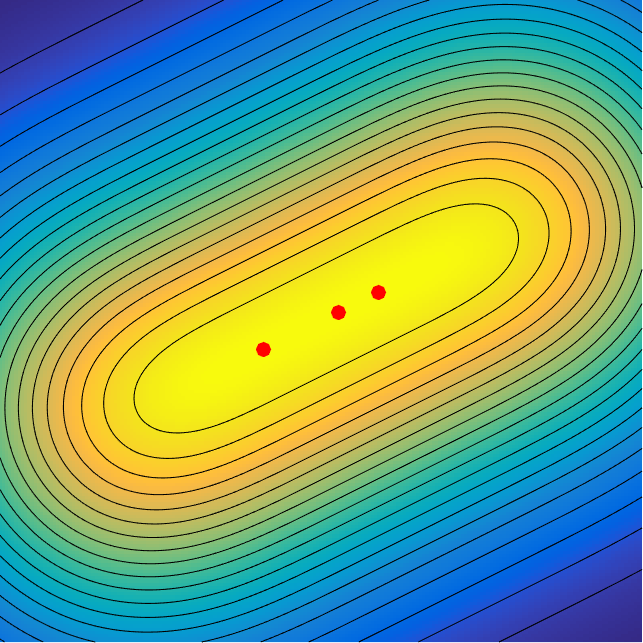}&
\MyFigEtaV{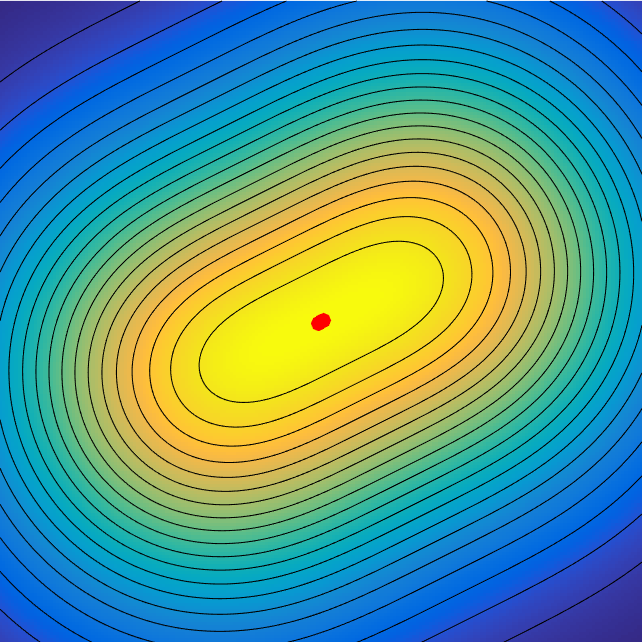}\\
$t=1$ & $t=.5$ & $t=0.2$ & $t=0.01$
\end{tabular}
\caption{\label{fig-etav-conv}
Display of the evolution of $\eta_{V,t\Z}$ for $t \rightarrow 0$. Top: 1-D Gaussian convolution.
Middle: 2-D Gaussian convolution, $\Z$ in generic positions.
Bottom: same but with aligned positions $\Z$.
}
\end{figure}

\subsubsection{The Least Interpolant Space $\Ll_\Ff$}
\label{sec-least-interpolant}

The definition~\eqref{eq-defn-etaV} of $\eta_{V,\Z}$ involves a Hermite interpolation problem at nodes $\Z=(z_i)_{i=1}^N$. Considering the asymptotic $t\Z$ with $t \rightarrow 0$ of the associated interpolation problem naturally leads to the analysis (through Taylor expansion) of the behavior of polynomial interpolation. 
Polynomial interpolation in arbitrary dimension is notoriously difficult, and we refer to the monograph~\cite{lorentz2000multivariate} for a detailed account on this topic. This is due in large part to the fact that finding suitable polynomial spaces so that the interpolation problem is \textit{regular} (has a unique solution) is non trivial, and that, in contrary to the 1-D case, such a space space depends on the interpolating positions $\Z$. As we now explain, solving this issue is at the heart of the description of $\eta_{W,\Z}$, and can be achieved in a canonical way using a construction of de Boor. Although we only use it for a specific interpolation problem (Hermite interpolation with first order derivatives only), we describe here in more generality. 

An interpolation problem over a space $\Ss \subset \Pi^d$ looks for a polynomial $P \in \Ss$ solution of a system of equations 
\begin{equation}\label{eq-poly-interp-prob}
F_r(P)=c_r \quad \text{for} \quad r=1,\ldots,Q
\end{equation}
 for some $c \in \RR^Q$, and where $\Ff = (F_r)_{r=1}^Q$ are linear forms. Of interest for us are differential forms evaluated at the positions $(z_i)_{i=1}^N$
\eql{\label{eq-diff-forms}
	\foralls i=1,\ldots,N, \quad
	\foralls j=1,\ldots,k_i,  
	\quad F_{(i,j)}(P) = \Big( P_{i,j}(\partial)[P] \Big)(z_i)
}
where we denoted $r=(i,j)$ the index, $Q=\sum_i k_i$, and $(P_{i,j})_{i,j}$ are given polynomials. 

As an example, the Hermite interpolation problem in dimension $d$ uses the monomials $\{ P_{i,j} \}_j = \{ X^{\al} \}_{|\al| \leq k_i}$ up to a fixed degree $k_i$, where $|\al|=\sum_{s=1}^d \al_s$ is the degree of the monomial. 
Lagrange interpolation is the special case where $n_i=0$, and to account for the constraints appearing in~\eqref{eq-defn-etaV}, we need to set $k_i=1$, so that $Q=(d+1)N$.  

An important question is how one should choose the subspace $\Ss$ for this problem to be regular, i.e. have an unique solution for any choice of right hand side $c$ in \eqref{eq-poly-interp-prob}. In the univariate case $d=1$, one can always choose $\Ss =\Pi^1_n$ where the degree is $n=Q$. However, the situation is much more complicated in the multivariate case because one cannot always choose $\Ss = \Pi^d_n$ for some $n\in \NN_0$.
To understand the issues here, first note that since $\dim \Pi_n^d = \binom{d+n}{d}$ and the number of partial derivatives to interpolate at $z_i$ is $\binom{d+k_i}{d}$, we would need $n$ to satisfy
\begin{equation}\label{eq:necc_interp}
	\binom{d+n}{d} = \sum_{i=1}^N \binom{d+k_i}{d}.
\end{equation}
For instance, for the Lagrange interpolation at $\Z = \{z_1,z_2\} \subset \RR^2$ (so $d=2$), there does not exist an integer $n$ such that~\eqref{eq:necc_interp} holds since the number of interpolation conditions is 2 while $\abs{\Pi_0^2}=1$ and $\abs{\Pi_1^2}=3$. Furthermore, in the case of Hermite interpolation at $\Z$ with $k_i=1$, although choosing $\Ss = \Pi_2^2$ would satisfy~\eqref{eq:necc_interp}, interpolation with this space is in fact singular for all choices of $z_1$ and $z_2$ \cite{lorentz2000multivariate}.




In \cite{de1992least,de1992computational}, given a finite set of linear functionals $\Ff = \{F_q\}_q$, de Boor and Ron established a general technique for finding an appropriate polynomial space $\Ss=\Ss_\Ff$, so that the interpolation problem is regular, i.e. such that for any $c \in \CC^Q$, there exists a unique element  $P \in \Ss_\Ff$  such that \eqref{eq-poly-interp-prob} holds. This space is defined through the use of the least term in formal expansion of exponential forms.

\begin{defn}[Least term]
	Let $g$ be a real-analytic function on $\RR^d$ (or at least analytic at $x=0$, so that $g(x) = \sum_{\abs{\al}=0}^\infty a_\al x^\al$). Let $\al_0$ be the smallest integer $\abs{\al}$ such that $a_\al\neq 0$. Then the least term $g_\downarrow$ of $g$ is $g_\downarrow = \sum_{|\al|=\al_0} a_\al x^\al$.
\end{defn}


\begin{defn}[Exponential space]
For a linear functional $F\in (\Pi^d)'$, we define the formal power series 
\eq{
	g_F(x) \eqdef F(e^{\dotp{\cdot}{x}}) \eqdef \sum_{\abs{\al}=0}^\infty \frac{F(p_\alpha)}{j!} x^\al.
}
where $p_\alpha \eqdef x\mapsto x^\alpha$.
Given functionals $\Ff = \{F_q\}_{q}$ of the form~\eqref{eq-diff-forms}, we define the space
\begin{align*}
	\exp_\Ff &\eqdef \Span\enscond{g_F}{F\in \Ff}, \\
	\Ll_\Ff &\eqdef \Span\enscond{g_\downarrow}{g\in \exp_\Ff}.
\end{align*}
The polynomial space $\Ll_\Ff$ is called the least interpolant space.
\end{defn}

The main theorem of~\cite{de1992least} asserts that this space $\Ll_\Ff$ defines regular interpolation problems. Note that of course $\Ll_\Ff$ depends on the positions $\Z$. We recall below this theorem, stated as in \cite{lorentz2000multivariate}.

\begin{thm}[\cite{lorentz2000multivariate}]
Let $\Ff = \{F_q\}_{q=1}^Q$ be functionals of the form~\eqref{eq-diff-forms}. 
Then, for any $c \in \RR^Q$, there exists a unique $P \in \Ll_\Ff$ with 
\eq{
	\foralls r=1,\ldots,Q, \quad
	F_r(P) = c_r.
}
\end{thm}

\begin{exmp}\label{2spikes_init_examp}
Let us consider an example presented in \cite{lorentz2000multivariate}, which is also serves as an explanatory example throughout this article.

Consider the problem of Hermite interpolation at $\Z = \{ z_1\eqdef (0,0), z_2\eqdef (0,1)\}$.  The 6 linear functionals are
$$
\Ff = \{ F_{0,i}: p\mapsto p(z_i), \; F_{1,i}: p\mapsto \partial_x p(z_i), \; F_{2,i}: p \mapsto \partial_y p(z_i) \}_{i=1,2}.
$$
Then, the corresponding exponential functions are
\begin{align*}
& g_{F_{0,1}}(x,y) = 1, \quad  g_{F_{1,1}}(x,y) =x, \quad g_{F_{2,1}}(x,y) =y,\\
&g_{F_{0,2}}(x,y) = e^y,\quad g_{F_{1,2}}(x,y) = x e^y, \quad g_{F_{2,2}}(x,y)=ye^y,
\end{align*}
and
$
\Ll_\Ff = \Span\{1,x,y,xy, y^2, y^3\}.
$
\end{exmp}

The polynomial space constructed by de Boor and Ron preserves many properties of univariate interpolation, but notably,  $\Ll_\Ff$ is of least degree in the sense that for any other space $\Ss$ leading to a regular interpolation problem, 
$$
	\dim(\Ss \cap \Pi_n^d)\leq \dim(\Ll_\Ff \cap \Pi^d_n), \qquad \forall n \in \NN_0.
$$
Furthermore, if we write $\Ss_{\Z} \eqdef \Ll_\Ff$ to be the least interpolant spaced associated to $\Ff$ defined in~\eqref{eq-diff-forms} (to highlight the dependency on $\Z$), then we have that
for all invertible matrices  $A \in \RR^{d\times d}$ and $x \in \RR^d$, 
\eq{
	\Ss_{A\Z+x} = \Ss_{\Z} \circ A^\top = \enscond{ P \circ A^\top }{ P \in \Ss_{\Z} }, 
} 
where we denoted $A\Z+x=(A z_i+x)_i$, $A^\top$ the transpose matrix.

\subsubsection{The de Boor Basis $\Bb_\Ff$ of $\Ll_\Ff$}
\label{sec-de-boor}

As highlighted already in Remark~\ref{rem-comp-etaW-multid}, to be useful from a computational point of view, it is  important to describe an interpolation space $\Ss$ using a basis of polynomial $(P_r)_{r=0,\ldots, Q-1}$.
In the case of the least interpolant space $\Ss_\Z=\Ll_\Ff$ for Hermite functionals $\Ff$ defined in~\eqref{eq-diff-forms}, algorithms for finding such a basis are presented  in~\cite{de1992computational}.
For simplicity, we describe one of the proposed algorithms in the special case of 2-D Hermite interpolation with the first order derivative, i.e. the case $k_i=1$, as this is the one of interest for us. However, it extends to verbatim in the general settingwith interpolation conditions defined via general differential forms \eqref{eq-diff-forms}, see ~\cite{de1992computational}  for further details.
%
%
%

%
%
%

The basic procedure of computing a basis of $\Ss_{\Z}$ for $\Z\in \Xx^N$ can be summarised as follows.
\begin{alg}\label{thm:deboor} \cite[Thm. 2.7]{de1992computational}
\begin{enumerate}
\item  By identifying each polynomial with its coefficients, define the Hermite intepolation operator $V_{\Z}: \Pi_n \to \RR^{3N}$ as an infinite dimensional matrix (with infinitely many columns indexed by $\NN_0^2$ and $Q$ columns indexed by $\{1,\ldots, Q\}$) 
$$
	V_{\Z}: a \mapsto \left( (P_a(z_i))_{i=1}^N, (\partial_{x_1} P_a(z_i))_{i=1}^N, (\partial_{x_2} P_a(z_i))_{i=1}^N \right) \in \RR^{Q},
$$
where $P_a = \sum_{\alpha\in \NN_0^2} a_\al X^\al$.
\item Perform Gaussian elimination with partial pivoting \cite{trefethen1997numerical} to obtain the decomposition $V = LW$, where $L\in \RR^{Q \times Q}$ is an invertible matrix and $W \in \RR^{3N\times \NN_0^2}$ is in row reduced echelon form. 
\item For each row $j$ of $W$, let $\beta_j$ be the first index of $W_{j,\cdot}$ such that $W_{j,\alpha}\neq 0$. Define 
\eq{
	\foralls j=0,\ldots,Q-1, \quad 
	P_j(X) \eqdef \sum_{\abs{\al} = \abs{\beta_j}} \frac{1}{\al!} W_{j,\al} X^\al.
} 
\end{enumerate}
Then, $\{P_j\}_{j=1}^{Q-1}$ defines a basis of $\Ss_{\Z}$.

\end{alg}

\begin{rem}\label{rem:V_z_restr}
It is in fact sufficient to restrict $V_{\Z}$ to the polynomial space $\Pi_{2N-1}$ because Hermite interpolation on $N$ nodes is always regular on  $\Pi_{2N-1}$  \cite[Theorem 19]{lorentz2000multivariate}. Therefore, since $\abs{\Pi_{2N-1}} = 2N^2+N$, a basis of $\Ss_{\Z}$ can be computed in $O(Q^2 N^2)$ operations.
In particular, we can replace $V_{\Z}$ by $\tilde V_\Z \in \RR^{Q\times (2N^2+N)}$ where
$$
\tilde V_{\Z} \eqdef \begin{pmatrix}
\left( z_l^\alpha \right)_{l\in[N],  \abs{\alpha}\leq 2N-1}\\
\left( z_l^{\alpha-(1,0)} \right)_{l\in[N],  \abs{\alpha}\leq 2N-1}\\
\left( z_l^{\alpha-(0,1)} \right)_{l\in[N],  \abs{\alpha}\leq 2N-1}
\end{pmatrix}.
$$
Note that as a result, in Step 2, $W$ is a $Q\times (2N^2+N)$ matrix.

\end{rem}

%
%

\begin{rem} \label{rem:deboor_basis}
The main result of the paper~\cite{de1992computational} also presents a more sophisticated construction of a basis $\Ss_{\Z}$, based on Gaussian elimination after appropriately grouping together columns of $V_{\Z}$. That approach has the advantage that the resultant basis is numerically more stable and is orthogonal with respect to the product  $\dotp{P}{Q}_B \eqdef (P(\partial)[Q])(0)$. This basis will be referred to as \textbf{the de Boor basis}.
In the following section, we shall establish a precise link between the least interpolant space and $\eta_W$ using Theorem \ref{thm:deboor}. We emphasize, however, that although an explicit basis is useful for computational purposes, it is rather the space $\Ss_{\Z}$ which determines $\eta_W$.
\end{rem}

\begin{exmp}
Returning to Example \ref{2spikes_init_examp} where  $\Z = \{(0,0), (0,1)\}$, we have that
$
V_\Z = LW,
$
where
$$L\eqdef \begin{pmatrix}
1 & 0 & 0 & 0 & 0 & 0\\
1 & 1 & 0 & 1/2 & 0 & 1/6\\
0 & 0 & 1 & 0 & 0 & 0\\
0 & 0 & 1 & 0 & 1 & 0\\
0 & 1 & 0 & 0 & 0 & 0\\
0 & 1 & 0 & 1 & 0 & 1/2
\end{pmatrix}
\quad \text{and} \quad
W: a\mapsto \begin{pmatrix}
a_{0,0}\\
a_{0,1}\\
a_{1,0}\\
a_{0,2}+ \sum_{j\geq 4} \frac{(6-2j)}{j!} a_{j,0} \\
a_{1,1} + \sum_{j\geq 2} \frac{1}{j!} a_{1,j}\\
a_{0,3} + \sum_{j\geq 4} \frac{(6j-12)}{j!} a_{0,j} 
\end{pmatrix}.
$$
By Theorem \ref{thm:deboor}, a basis of $\Ll_\Ff$ is therefore 
 $\Bb_\Ff = \{1,y,x,y^2,xy, y^3\}$. 
\end{exmp}

\begin{exmp}[Further examples]\label{exp:further-etaW}
In the following, we write $x^j y^k$ for the polynomial $(x,y)\mapsto x^jy^k$.
\begin{itemize}
\item When $\Z = \enscond{ (a_j,0) }{j=1,\ldots, N}$,
$$
	\Bb_\Ff = \{ 1,x,x^2,\ldots, x^{2N-1}, y, xy,\ldots, x^{N-1}y\}
$$
is a basis for $\Ss_{\Z}$.
\item When $\Z= \{(0,1), (0,0), (1,0)\}$, 
$$
	\Bb_\Ff = \{1,x,y,x^2,y^2,xy, x^3, y^3, x^2y - y^2x\}.
$$
\item When $\Z = \{(1,1),(1,-1),(-1,1),(-1,-1)\}$,
$$
	\Bb_\Ff = \{1,x,y,x^2,y^2,xy, x^3,y^3,x^2 y,y^2x, x^3y, y^3,x\}.
$$
\end{itemize}
\end{exmp}

\subsubsection{The Limiting Certificate}
\label{sec-limiting}

With the construction of the least interpolant space at hand, we are now ready to explicitly define the limit $\eta_{W,\Z}$ of $(\eta_{V,t\Z})_{t>0}$ as $t\to 0$. We use for this the following interpolation space for spikes $\Z=(z_i)_{i=1}^N$
\eql{\label{eq-interp-space-etaw}
	\Ss_{\Z} \eqdef \Ll_\Ff
	\quad\text{where $\Ff$ is defined in~\eqref{eq-diff-forms} with } \foralls i, n_i=1.
}

To begin with, let us define an operator which will be useful for establishing the technical results of this paper.
Let
\begin{equation}\label{eq:Gamma}
\Gamma_{t\Z}:\RR^{3N} \to \Hh, \qquad \Gamma_{t\Z} \begin{pmatrix}
a\\b\\c\\
\end{pmatrix} \eqdef
\sum_{j=1}^{N} a_j \varphi(tz_j) + \sum_{j=1}^N  b_j \partial_{x_1}\varphi(tz_j) + \sum_{j=1}^N  c_j \partial_{x_2}\varphi(tz_j)
\end{equation}
Note that the precertificates can be written as $p_{V,t\Z} = \Gamma_{t\Z}^{*,\dagger}\binom{1_N}{0_{2N}}$.
Moreover,  observe that given $p\in \Hh$, by identifying $p$ with the  coefficients of the Taylor expansion of $(\Phi^* p)(t\cdot)$ around 0, that is $P =  \left(\dotp{p}{\frac{t^{\abs{\alpha}}}{\alpha!}\partial^\alpha\varphi(0)}\right)_{\alpha\in \NN_0^2}$, we can associate $\Gamma_{t\Z}$ with the Hermite interpolation matrix $V_\Z$ (as defined in Procedure \ref{thm:deboor}) via
\begin{equation}\label{Gammaz_Vz}
\Gamma_{t\Z}^* p =\diag((1_N, t^{-1} 1_{2N}) V_{\Z} P. 
\end{equation}


\begin{thm} \label{thm-conv-etaw}
Let $\Ss_{\Z}$ be the least interpolant space defined in ~\eqref{eq-interp-space-etaw} and suppose that 
$\enscond{h(\partial)\varphi(0)}{h\in \Ss_{\Z}}$ is of dimension $3N$. Then, for $\Z \in \Xx^N$, one has
\eq{
	\norm{p_{V,t\Z}  - p_{W,\Z}}_{\Hh} = \Oo(t)
	\qandq
	\norm{\eta_{V,t\Z} - \eta_{W,\Z}}_{L^\infty(\Xx)} = \Oo(t)
}
where $p_{W,\Z}$ and $\eta_{W,\Z}$ are defined in~\eqref{eq-defn-etaW-multidim} using  $\Ss_{\Z}$.
\end{thm}

\begin{rem}\label{rem:other-diff-forms}
In this article, we are interested in the limit of $\eta_{V,t\Z}$ which are defined using Hermite interpolation conditions at $t\Z$. Note however that  this result  holds also for the limit of certificates  defined via other differential forms. In particular, given any linear subspace of polynomials $\overline \Ss$ such that $P\in \overline\Ss$ implies that $P(0) =0$, if
$$
\tilde p_{V,t\Z} \eqdef \enscond{\norm{p}}{(\Phi^*p)(t z_i) = 1, \; P(\partial)(\Phi^* p)(t z_i)= 0,\; \forall x\in \Z, \; P\in \overline\Ss},
$$
then $\norm{\tilde p_{V,t\Z} - \tilde p_{W,\Z}}_\Hh  = \Oo(t)$, where $\tilde p_{W,\Z}$ is defined through \eqref{eq-defn-etaW-multidim} using the least interpolation space associated with $\overline \Ss \cup \{x\mapsto 1\}$.
\end{rem}

\begin{proof}
Suppose that we decompose $V_{\Z}$ via Gaussian elimination so that $V_{\Z} = LW$, where $L$ is  invertible  and $W$ is in row-reduced echelon form. Let $(\beta_j)_{j=0}^{Q-1}$ be as in Step 3 of Procedure \ref{thm:deboor}. Then, using the representation of $\Gamma_{t\Z}$ from \eqref{Gammaz_Vz}, we have that
$$
\Gamma_{t\Z}^* p  = \binom{1_N}{0_{2N}} \iff  V_{\Z} P =  \binom{1_N}{0_{2N}} \iff W P =  L^{-1} \binom{1_N}{0_{2N}}.
$$
Note that since the first $N$ entries of the first column of $V_{\Z}$ are all 1's, we have that $L^{-1} \binom{1_N}{0_{2N}}  = \delta_{3N}$.
By definition of the $\beta_i$'s,  we have that
\begin{equation}\label{eta_w_proof_eq1}
W P = \left( t^{ \abs{\beta_i}} \sum_{\abs{\al} = \abs{\beta_i}}  W(z_i, \al)\dotp{\partial^\al \phi(0)}{p} /\al! + \Oo(t^{ \abs{\beta_i}+1}) \right)_{i=1}^{3N}.
\end{equation}
Let $\Psi^*: \Hh\to \RR^{3N}$ be defined by
$$
\Psi^* p =\left(( h_i(\partial)[\Phi^* p])(0)\right)_{i=1}^{3N},
$$
where 
$$
\{h_i\}_{i=1}^{3N} \eqdef \enscond{X \mapsto \sum_{\abs{\al} = \abs{\beta_i}}  \frac{W_{z_i,\alpha}}{\alpha!} X^\alpha }{i=1,\ldots,3N} 
$$
is known to be a basis of $\Ss_{\Z}$ from Theorem \ref{thm:deboor}. 
Therefore, from \eqref{eta_w_proof_eq1}, there exists an operator  $\tilde \Psi_t^*: \Hh \to \RR^{3N}$  with $\tilde \Psi_t = \Oo(t)$ such that
$$
W P = \diag((t^{\abs{\beta_i}})_{i=1}^{3N})\left( \Psi^* p + \tilde \Psi_t^* p \right).
$$
Therefore, 
$$
\Gamma_{t\Z}^* p \binom{1_N}{0_{2N}}\iff \Psi^* p + \tilde \Psi_t^* p = \delta_{3N}
$$
and $p_{V,t\Z} = (\Psi^* + \tilde \Psi_t^*)^\dagger \delta_{3N} = \Psi^{*,\dagger} \delta_{3N} + \Oo(t)$ whenever $\Psi$ is full rank.
Finally, since the first row of $V_{\Z}$ coincides with the first row of $W$ (thanks to the fact that the top left entry of $V_{\Z}$ is 1), $h_1\equiv 1$. Therefore, $p_{W,\Z}$ is as defined in~\eqref{eq-defn-etaW-multidim} using  $\Ss_{\Z}$.  The final claim of this theorem is due to the following inequality:
$$
\norm{\eta_{W,\Z}-\eta_{V,t\Z}}_{L^\infty(\Xx)} = \sup_{x\in\Xx} \abs{\dotp{\varphi(x)}{p_{W,\Z}- p_{V,t\Z}}} \leq \sup_{x\in\Xx}\norm{\varphi(x)}_\Hh \norm{p_{W,\Z}- p_{V,t\Z}}_\Hh.
$$
\end{proof}


\begin{rem}[Computation of $\eta_{W,\Z}$]
	With this theorem~\ref{thm-conv-etaw}, it is now simple to compute $\eta_{W,\Z}$ using the scheme detailed in Remark~\ref{rem-comp-etaW-multid} and the de Boor basis $\Bb_\Ff = \{P_r\}_{r=0}^{Q-1}$ in formula~\eqref{eq-etaw-highdim-corr}.
\end{rem}

A key assumption of Theorem \ref{thm-conv-etaw} is on the dimension of the space $\enscond{h(\partial)\varphi(0)}{h\in \Ss_{\Z}}$. The following proposition shows that this is satisfied for convolution kernels of sufficiently large bandwidth. 
\begin{prop}\label{prop:lin_indep}
Let $\Phi$ be a convolution operator with $\varphi(x) =  \psi(\cdot -y)$. Let $L$ be the smallest integer such that $\Pi_L^2 \supseteq \Ss_{\Z}$. Suppose that $\hat \psi(\alpha)\neq 0$  for all $\al\in\NN_0^2$ such that $\abs{\al}\leq L$. Then, given any $z\in \Xx^N$, 
$\enscond{h(\partial)\varphi(0)}{h\in \Ss_{\Z}}$ is of dimension $3N$. Furthermore, $L\leq 2N-1$.

\end{prop}

The proof of this proposition can be found in Appendix~\ref{proof:prop:lin_indep}.

\subsection{Necessity of  $\eta_W$}
\label{sec-necessity-nondegen}

The following result shows that if there is support stability for $m_{a,t\Z}$ for all $t$ sufficiently small and $\Z$ sufficiently close to $\Z_0$, then $\eta_{W,\Z}$ must be a valid certificate. In particular, if $\eta_{W,\Z}(x)>1$ for some $x\in\Xx$, then under arbitrarily small noise and regularization parameter $\lambda$,  \ref{eq-blasso} will produce solutions with additional small spikes (see Section \ref{sec-numerics}).

\begin{thm}
Suppose that $\enscond{h(\partial)\varphi(0)}{h\in\Ss_{\Z_0}}$ is of dimension $3N$ and that $\varphi \in \Cder{2}(\Xx)$.
Suppose that there exists $t_n\to 0$ and $(a_n,\Z_n)\in\RR^N\times \Xx^N$ with $\Z_n\to \Z_0$ such that $m_{a_n,t_n \Z_n}$ is support stable: i.e. For each $n$, there exists a neighbourhood $V_n\subset \RR\times \Hh$ of $0$  a continuous path $g_n: (\lambda,w)\in V_n \mapsto (a,\Z)\in\RR^N\times \Xx^N$ such that $m_{a,\Z}$ solves $\Pp_\la(y_n + w)$ with $y_n =\Phi m_{a_n,t_n\Z_n}$.
Then, $\norm{\eta_{W,\Z_0}}_{L^\infty} = 1$.
\end{thm}
\begin{proof}
First note that since $\enscond{h(\partial)\varphi(0)}{h\in\Ss_{\Z_0}}$ is of dimension $3N$, we have that $\Gamma_{t_n \Z_n}$ is full rank for all $n$ sufficiently large.

We first show that the path $g_n$ coincides with a $\Cder{1}$ function in a neighbourhood of $0$: 
For each $n$, define for $u=(a,\Z)$ and $v=(\la,w)$,
$$
f_{n}(u,v) \eqdef \Gamma_{t_n\Z}^*(\Phi_{t_n\Z} a-\Phi_{t_n \Z_n} a_n - w) -\la \binom{1_N}{0_{2N}}.
$$
For all $(\la,w)\in V_n$, optimality of $m_{a,\Z}$ for $(a,\Z) = g_{n}(\la,w)$ implies that $f_n(g(\la,w),(\la,w)) = 0$.
Since $f_n$ is $\Cder{1}$, $f_n((a_n,\Z_n),0) = 0$, $\partial_u f_n((a_n,\Z_n),0) $ is invertible, we can apply the implicit function theorem to deduce that there exists $g_*\in \Cder{1}$, a neighbourhood $V\subset \RR\times \Hh$ of $0$ and a neighbourhood $U\subset \RR^N\times \Xx^N$ of $(a_n,\Z_n)$ such that $g_*:V\to U$ and $g_*(\la,w) = (a,\Z)$ if and only if $f_n((a,\Z),(\la,w)) = 0$. Now, by continuity of $g_n$, there exists $\tilde V\subset V_n\cap V$ such that $g_{n}(\tilde V)\subset U$. Moreover, since $f_n(g(\la,w),(\la,w)) = 0$ for all $(\la,w)\in \tilde V$, we have that $g_n(\la,w) = g_*(\la,w)$ for all $(\la,w)\in\tilde V$. Therefore, $m_{a_n,t_n \Z_n}$ is support stable with a $\Cder{1}$ function. We may now apply \cite[Proposition 8]{duval2015exact} to conclude that $\norm{\eta_{V,t_n\Z_n}}_{L^\infty}\leq 1$. It remains to show that
\begin{equation}\label{eq:remains}
\eta_{V,t_n\Z_n} = \Phi^* \Gamma_{t_n \Z_n}^{*,\dagger}\binom{1_N}{0_{2N}} \to \eta_{W,\Z_0}
\end{equation}
 as $n\to \infty$.
 To show this, note that if $\lim_{n\to }\Z_n = \Z_0$, then for $n$ sufficiently large, there exist smooth bijections $T_n: \Xx \to \Xx$ such that $T_n \Z_0 = \Z_n$ and $\lim_{n\to \infty}\norm{T_n - \Id} = 0$.  Let $\tilde T_n = t_n T_n \circ (t_n^{-1}\Id)$. Then, $\tilde T_n (t_n \Z_0) = t_n \Z_n$. So, by Corollary \ref{cor:transinv},
$$\eta_{V,t_n\Z_n}= \eta_{V,t_n\Z_0} + \Oo(\norm{\Id-\tilde T_n}) = \eta_{V,t_n\Z_0} + \Oo(\norm{\Id- T_n}),
$$
thus yielding \eqref{eq:remains} by letting $n\to \infty$.

\end{proof}

\subsection{Special cases}

\subsubsection{Explicit formula of $\eta_{W,\Z}$ for Gaussian convolution}
\label{sec-gaussian-closedform}

We consider the Gaussian convolution measurement operator $\phi(x) = \psi(x-\cdot) \in \Hh = L^2(\RR^d)$ on $\Xx=\RR^d$ where
\eq{
	\psi(x_1,x_2) \eqdef (y_1,y_2)\mapsto e^{-(\abs{x_1-y_1}^2+ \abs{x_2-y_2}^2)}
} 
i.e. $\Phi m =  m\star \psi$ is the convolution against kernel $\psi$. 

\begin{prop}\label{prop:gaussian}
Suppose that  the de Boor basis $\Bb$ for $\Ss_{\Z}$ is of the form
$$
	\{P_\al\}_{\al\in J} \eqdef \{X\mapsto X^\al; \; \abs{\al}\leq L\}\cup \{P_\al\}_{j\in J_{L+1}}
$$ 
	where $P_\al$, $j\in J_{L+1}$ are homogeneous polynomials of degree $L+1$. Define the inner product $\dotp{P}{Q}_B \eqdef (P(\partial)[Q])(0)$. Then,
$$
	\eta_{W,\Z} = e^{-(x_1^2+x_2^2)/2}\left(  \sum_{\al \in J} 
	\frac{ \dotp{P_\al}{\tilde \psi}_B }{ \dotp{P_\al}{P_\al}_B  } P_\al(x) 
	\right)
	\qwhereq
	\tilde \psi(x) \eqdef \exp(\norm{x}^2/2).
$$
In particular, if $\Bb = \enscond{X\mapsto X^\al}{\abs{\al}\leq L}$, then
$$
\eta_{W,\Z}(x) = \exp(-\norm{x}^2/2) \sum_{0\leq 2\al_1+2\al_2\leq L} \frac{1}{\al!} x_1^{2\al} x_2^{2\al}.
$$
In the case where $\Z$ consists of $N$ points, all aligned along the first axis,
$$
\eta_{W,\Z}(x) = \exp(-\norm{x}^2/2) \sum_{0\leq j\leq 2N-1} \frac{x_1^{2j}}{j!}.
$$

\end{prop}

\begin{proof}
 First note that $\eta_{W,\Z}$ is of the unique function of the form
$$
\eta_{W,\Z}(x) = \sum_{\abs{\al}\leq L} a_\al \partial^\al[\psi\star\psi](x) + \sum_{\al\in J_{L+1}} a_\al (P_\al(\partial)[\psi\star\psi])(x),
$$
with $\eta_{W,\Z}(0)=1$, and $P_\alpha(\partial)\eta_{W,\Z}(0) = 0$ for $\al\in J \setminus \{(0,0)\}$.
Note that $$[\psi\star\psi] (x_1,x_2) = \frac{\pi}{2} \exp(\frac{-(\abs{x_1-y_1}^2+ \abs{x_2-y_2}^2)}{2}).$$
Moreover, $$\frac{\mathrm{d}^n}{\mathrm{d}x^n} \exp(-x^2/2) = (-1)^N \exp(-x^2/2) H_n(x)$$ where $H_n$ is a monic polynomial of degree $n$ (called the Hermite polynomial of degree $n$).

Since $P_\al$ is a homogeneous polynomial of degree $\abs{\al}$, it follows that
$$
(P_\al(\partial)[\psi\star\psi])(x) = (-1)^{\abs{\al}} \exp(-\norm{x}^2/2)(P_\al(x) + f_\al(x))
$$
where $f_\al$ is a polynomial of degree at most $\abs{\al}-1$.
Therefore,
$$
F(x) = \exp(\norm{x}^2/2)\cdot \eta_{W,\Z}(x) = \sum_{\abs{\al}\leq L} a_\al x^\al + \sum_{\abs{\al}=L+1} a_{\al}P_\al(x).
$$
is a polynomial of degree $L+1$ and
and it remains to determine the coefficients $a$.
Since the de Boor basis is orthogonal w.r.t. the product $\dotp{\cdot}{\cdot}_B$ (see Remark \ref{rem:deboor_basis}), we have that
$$
P_\al(\partial)F(0) = a_\al P_\al(\partial)P_\al(0) = [P_\al(\partial)\tilde G](0).
$$
Therefore,
$$
	\eta_{W,\Z} = e^{-(x_1^2+x_2^2)}\left(  \sum_{\al \in J} 
		\frac{ \dotp{P_\al}{\tilde \psi}_B }{ \dotp{P_\al}{P_\al}_B  }
	 P_\al(x) \right).
$$

The case where $\Z$ consists of aligned points can be dealt with in a similar manner.

\end{proof}

\subsubsection{Convolution Operators and Vanishing Odd Derivatives}
\label{sec-convolution-vanish}

The following proposition shows that convolution operators enjoy the property that the odd derivatives of $\eta_{W,\Z}$ vanish. This typically leads to better behaved (e.g. non-degenerate) certificates, as illustrated in Section~\ref{sec-numerics}.
More generally, this proposition shows that, if  the correlation kernel of $\Phi$ satisfies $C(x,x') = C(-x,-x')$, then the vanishing of consecutive derivatives up to some $N\in\NN$ will imply the vanishing of all odd derivatives.

\begin{prop}\label{lem:odd_vanish}
Let $C:\RR^d\times \RR^d \to \RR$ be such that $C(x,x') = C(-x,-x')$. Let $C_{\al,\beta} \eqdef \partial^{\al_1}_{x_1}\cdots \partial^{\al_d}_{x_d} \partial_{x'_1}^{\beta_1} \cdots \partial_{x'_d}^{\beta_d}  C$.
For $N\in \NN$ we define $\eta(x) \eqdef \sum_{\abs{\al}\leq N} b_\al C_{\al,0}(0,x)$, where 
\eq{
	b \eqdef R^{-1}(1,0,\cdots,0)^T, \quad
	R \eqdef (C_{\al,\beta}(0,0))_{\substack{\al,\beta\in\NN_0^2, \:
	\abs{\al},\abs{\beta}\leq N}} \in \RR^{K\times K}
	\qandq
	K \eqdef \binom{N+d}{N}.
}
Then, $\nabla^j \eta(0) = 0$ for all odd integers $j\in \NN$.
\end{prop}
\begin{proof}
Since $C(x,x') = C(-x,-x')$, $C_{\al,\beta}(0,0) = 0$ whenever $\abs{\al+\beta}$ is odd. Therefore, we may rearrange the row and columns of $R$ so that
$$
R = \begin{pmatrix}
 R_1 & 0\\
 0 & R_2 
\end{pmatrix}, \quad R_1 = \left(C_{\al,\beta}(0,0)\right)_{\substack{\al,\beta\in\NN_0^2\\
\abs{\al},\abs{\beta}\text{ both even}}}, \quad  R_2 = (C_{\al,\beta}(0,0))_{\substack{\al,\beta\in\NN_0^2\\
\abs{\al},\abs{\beta}\text{ both odd}}}.
$$
So,
$$
R^{-1}(1,0,\cdots,0)^T = \begin{pmatrix}
 R_1^{-1} & 0\\
 0 & R_2^{-1}
\end{pmatrix}(1,0,\cdots,0)^T = R^{-1}_1 (1,0,\cdots,0)^T.
$$
Therefore,
$$
\eta = \sum_{\substack{\abs{\al}\leq N\\ \abs{\al} \text{ even}}} b_\al C_{\al,0}(0,x)
$$
is an even function, and therefore, all odd derivatives of $\eta$ must vanish.
\end{proof}

\begin{rem}
As a consequence of this lemma, 
consider
$$
	p_{*} \eqdef \argmin \enscond{\norm{p}_{L^2}}{ (\Phi^* p)(0) = 1, \quad \nabla^k (\Phi^* p)(0) = 0, \quad k=1,\ldots, 2k, \quad \partial^\al(\Phi^* p)(0) = 0, \al\in J },
$$
where $J$ consists of multi-indices such that for $j\in J$, $\abs{j} = 2k+1$.
For $N\in\NN$, let
$$
	p_{N} \eqdef \argmin \enscond{\norm{p}_{L^2}}{ \Phi^* p(0) = 1, \quad \nabla^k (\Phi^* p)(0) = 0, \quad k=1,\ldots, N }.
$$ and consider $p_{2k}$ and $p_{2k+1}$ for some $k\in \NN$. 
Then, $\norm{p_{2k}} \leq \norm{p_*} \leq \norm{p_{2k+1}}$. However, by the above lemma, $p_{2k}=p_{2k+1}$.  Therefore, $p_* = p_{2k+1}$ and in particular, $\nabla^{2n+1} (\Phi^* p_*)(0) = 0$ for all $n\in \NN$.
\end{rem}


\section{Pair of Spikes}
\label{sec-two-spikes}

In this section, we consider the problem of recovering a superposition of two spikes at positions $t\Z_0$, $m_0 = \sum_{j=1}^2 a_{0,j} \delta_{tz_{0,j}}$, via  BLASSO minimization \eqref{eq-blasso} with $y = \Phi m_0+w$ for small noise $w\in\Hh$. We will show that for small $t$, provided that the limiting certificate $\eta_{W,\Z_0}$ is non-degenerate (see Definition~\ref{def:eta_W-nondegen}), the solution of \eqref{eq-blasso}  is \textit{support stable} with respect to $m_0$. By support stable, we mean that \textit{the solution is unique, has exactly $2$ spikes and that the positions and  amplitude of the recovered measure converge to $a_0$ and $\Z_0$ whenever  $(\lambda,w)$ converge  to 0 sufficiently fast.} The main result of this section will not only establish support stability, but also give precise bounds on how fast $(\lambda,w)$ should converge to 0.

In the 1-D case, the limiting certificate
is said to be non-degenerate if $\eta_W(x)<1$ for all $x\neq 0$, and its first derivative which has not been imposed to vanish at zero is negative at zero. In the case where $\eta_W$ is defined on $N$ spikes, this is the derivative of order $2N$. In 2-D, the behaviour of $\eta_{W,\Z}$ is in general non-isotropic, and in general, full derivatives are not imposed to vanish completely. When $\Z=(z_1,z_2)\in \Xx^2$, recalling that $\eta_{W,\Z}(0)=1$,
$$
\partial^j_{d_{\Z}} \eta_{W,\Z}(0) = 0, \; j=1,2,3,\quad \partial_{d_{\Z}^\perp}\eta_{W,\Z}(0)=0,\quad  \partial_{d_{\Z}} \partial_{d_{\Z}^\perp}\eta_{W,\Z}(0)=0,
$$
where $d_{\Z} = z_1-z_2$,
the analogous notion of non-degeneracy for $\eta_{W,\Z_0}$ is as follows.
\begin{defn}\label{def:eta_W-nondegen}
Let $\Z \eqdef \{z_1,z_2\}\in \Xx^2$ and let $d_{\Z} = z_2-z_1$, we say that $\eta_{W,\Z}$ is non-degenerate if   $\eta_{W,\Z}(x)<1$ for all $x\neq 0$ and
\begin{equation*}
\begin{pmatrix}
\partial^2_{d_{\Z}^\perp} \eta_{W,\Z}(0) &  \frac{1}{2}\partial_{d_{\Z}^\perp} \partial^{2}_{d_{\Z}}  \eta_{W,\Z}(0)\\
\frac{1}{2}\partial_{d_{\Z}^\perp} \partial^{2}_{d_{\Z}}  \eta_{W,\Z_0}(0) & \frac{1}{12}
\partial^{4}_{d_{\Z}}  \eta_{W,\Z}(0)
\end{pmatrix} \prec 0.
\end{equation*}
\end{defn}

The main result of this section is as follows.

\begin{thm}\label{thm-twospikes}
Let $\Z_0\in \Xx^2$.
Suppose that $\Psi_{\Z_0}$ is full rank and that $\eta_{W,\Z_0}$ is non-degenerate. Then, there exists constants $t_0$, $c_1$, $c_2$, $M$, such that for all $t\in (0,t_0)$, all $(\lambda,w)\in B(0, c_1 t^4)$ and $\norm{w/\lambda}\leq c_2$,
\begin{itemize}
\item $P_\la(y_t+w)$ has a unique solution.
\item the solution has exactly $N$ spikes and is of the form $m_{a,t\Z}$ where $(a,\Z) = g_t^*(\lambda,w)$, a continuously differentiable function defined on $B(0,c_1t^4)$.
\item  The following inequality holds:
\begin{equation}\label{eq:error-main}
\abs{(a,\Z)-(a_0,\Z_0)}_\infty \leq M \left( \frac{\abs{\la} + \norm{w}}{t^3} \right).
\end{equation}
\end{itemize}
\end{thm}

\paragraph{Reading guide.}

We begin with some preliminary bounds in Section \ref{sec-proof-asymptotic-exp}.
In Section~\ref{sec:nondegen-transfer}, we show that this stability is a direct consequence of the non-degeneracy transfer of $\eta_{W,\Z_0}$ to $\eta_{V,t\Z_0}$. Observe that support stability is then  a direct consequence of this non-degeneracy transfer,  Theorem~\ref{thm-conv-etaw} which shows that $\eta_{V,t\Z_0}$ converges to $\eta_{W,\Z_0}$ and  the main result of \cite{duval2015exact} which shows that non-degeneracy of $\eta_{V,t\Z_0}$ implies support stability. The remainder of this section is then devoted  to establishing precisely \textit{how fast} $(\lambda,\norm{w}_\Hh)$ need to converge to 0 to ensure support stability. In Section~\ref{sec-proof-asymptotic-exp}, we derive more precise bounds on the convergence of $\eta_{V,t\Z}$ to $\eta_{W,\Z}$ in the case where $\Z\in \Xx^2$. In Section~\ref{sec-proof-implicit-func-thm}, we construct a $\Cder{1}$ mapping $g:(\lambda,w)\to (a,\Z)$, and show that the associated measure $m_{a,\Z}$ is indeed a solution to \eqref{eq-blasso}. Similarly to the approach of~\cite{2017-denoyelle-jafa}, this is achieved via the Implicit Function Theorem. Section~\ref{sec-proof-thm-boundvt} is then devoted to analysis of the size of the region for which this function $g$ is defined, establishing bounds on the differential of $g$ (which eventually leads to the convergence bounds of Theorem~\ref{thm-twospikes}) and Section~\ref{sec-proof-thm-use-nondegen} proves that the measure $m_{a,\Z}$ is indeed a solution of \eqref{eq-blasso}. 

\subsection{Preliminaries}
\label{sec-proof-asymptotic-exp}

We have already seen from Theorem~\ref{thm-twospikes} that $\eta_{V,t\Z}$ converges to $\eta_{W,\Z}$ as $t\to 0$.  For the purpose of deriving precise estimates on the speed of convergence of $(\lambda,w)$ for support stability, we write explicitly in this section the relationship between  $\eta_{V,t\Z}$ and $\eta_{W,\Z}$.

\begin{lem}\label{prop:cvgence}
Let $\Z = \{(u_1,u_2), (v_1,v_2)\}$.
Then,
$$
\Gamma_{t\Z} = \Psi_{t\Z} H_{t\Z}
$$
where $\Psi_{t\Z}\eqdef \Psi_{\Z} + \Lambda_{t\Z}$,
$$
\Psi_{\Z}^* p \eqdef
\begin{pmatrix}
a_{0,0}\\
a_{0,1}\\
a_{1,0}\\
a_{0,2}-a_{2,0}(u_1-v_1)^2/(u_2-v_2)^2\\
a_{1,1} + (u_1-v_1)a_{2,0}/(u_2-v_2)\\
a_{0,3} + a_{3,0} \frac{(u_1-v_1)^3}{(u_2-v_2)^3} + 3 a_{1,2} \frac{u_1-v_1}{u_2-v_2} + 3a_{2,1} \frac{(u_1-v_1)^2}{(u_2-v_2)^2}
\end{pmatrix}, \qquad a_{j,k} = \dotp{\partial_x^j\partial_y^k \varphi(0)}{p},
$$
$$
H_{t\Z}^{*} \eqdef \begin{pmatrix}
1 & tu_2 & tu_1 & t^2 u_2^2/2 &t^2 u_1 u_2 & t^3 u_2^3/6 \\
1 & tv_2 & tv_1 & t^2 v_2^2/2 & t^2 v_1 v_2 & t^3 v_2^3/6\\
0 & 0 & 1 & 0 & t u_2 & 0\\
0 & 0 & 1 & 0 & t v_2 & 0\\
0 & 1 & 0 & tu_2 & t u_1 & t^2 u_2^2/2\\
0 & 1 & 0 & t v_2 & t v_1 & t^2 v_2^2/2
\end{pmatrix},
$$
and 
$
\Lambda_{t\Z}^*: \Hh \to \RR^6
$
satisfies the following properties:
\begin{itemize}
\item $\diag(t^{-2},t^{-1},t^{-1},t^{-1},t^{-1},t^{-1}) \Lambda_{t\Z}^* = \Oo(1)$, in particular, $\Lambda_{t\Z} = \Oo(t)$,
\item For any $\Z_0\in \Xx^2$, $\diag(t^{-2},t^{-1},t^{-1},t^{-1},t^{-1},t^{-1}) (\Lambda_{t\Z}^* - \Lambda_{t\Z}^*) = \Oo(\norm{\Z-\Z_0})$.
\end{itemize}
Furthermore, given $\Z_0\in \Xx^2$, we have that
\begin{itemize}
\item $\Psi_\Z = \Psi_{\Z_0} + \Oo(\norm{\Z-\Z_0})$
\item If  $\Psi_{\Z_0}$ is full rank, then $\Gamma_{t\Z}^{*,\dagger} \binom{1_2}{0_4} = \Psi_{\Z_0}^{*,\dagger}\delta_6 + \Oo(\norm{\Z-\Z_0})+ \Oo(t)$.
\end{itemize}
\end{lem}

\begin{proof}
We first recall that for $p\in \Hh$, we can define a vector $a= (a_{j,k})_{(j,k)\in\NN_0^2}\eqdef (\dotp{\partial^\alpha \varphi(0)}{p})_{\alpha\in \NN_0^2}$ and write $$\Gamma^*_{t\Z} p= V_{t\Z} a,$$ where $V_{t\Z}$ is the Hermite interpolation matrix at $t\Z$.
Then, by performing Gaussian elimination on the matrix $V_{t\Z}$, we obtain the following decomposition:
$$
\Gamma_{t\Z}^* p = H_{t\Z}^* \left( \underbrace{ \begin{pmatrix}
a_{0,0}\\
a_{0,1}\\
a_{1,0}\\
a_{0,2}-a_{2,0}(u_1-v_1)^2/(u_2-v_2)^2\\
a_{1,1} + (u_1-v_1)a_{2,0}/(u_2-v_2)\\
a_{0,3} + a_{3,0} \frac{(u_1-v_1)^3}{(u_2-v_2)^3} + 3 a_{1,2} \frac{u_1-v_1}{u_2-v_2} + 3a_{2,1} \frac{(u_1-v_1)^2}{(u_2-v_2)^2}
\end{pmatrix}}_{\Psi_{\Z}^* p} +\underbrace{ \begin{pmatrix}
h_1\\h_2\\h_3\\h_4\\h_5\\h_6
\end{pmatrix}}_{\Lambda_{t\Z}^* p}
\right).
$$
Note that 
$$  H_{t\Z}^{*,-1} =\diag(1,1/t,1/t,1/t^2,1/t^2,1/t^3) H_{\Z}^{*,-1} \diag(1,1,t,t,t,t),$$ and $(u_2-v_2)^3 H_{\Z}^{*,-1}$ is
{\tiny
$$
 \begin{pmatrix}
-v_2^3 + 3 u_2 v_2^2 & u_2^3 -3v_2u_2^2 & v_1v_2 u_2^2 - 2u_1u_2v_2^2 + u_1 v_2^3 & u_1u_2v_2^2 -v_1 u_2^3-2v_1v_2u_2^2  & - u_2v_2^2(u_2-v_2) & - u_2^2 v_2(u_2-v_2)\\ 
-6u_2v_2 & 6u_2v_2 & u_2^2v_1 -u_1v_2^2-4u_1u_2v_2+2u_2v_1v_2
& u_1v_2^2 + u_2^2 v_1 +2u_1u_2v_2 - 4 u_2 v_1v_2
&
(v_2^2+2u_2v_2)(u_2-v_2) & (u_2^2+2v_2u_2v_2)(u_2-v_2)\\
0& 0&-v_2(u_2-v_2)^2 & u_2(u_2-v_2)^2& 0 & 0\\

6(u_2+v_2) &-6(u_2+v_2) & -2(u_1-v_1)(2u_2+v_2) & -2(u_1-v_1)(u_2+2v_2) & -2(u_2+2v_2)(u_2-v_2) & -2(2u_2+v_2)(u_2-v_2)\\

0& 0& (u_2-v_2)^2 &  -(u_2-v_2)^2 &0& 0 \\
-12 & 12 & 6(u_1-v_1) & 6(u_1-v_1) & 6(u_2-v_2) & 6(u_2-v_2)
\end{pmatrix}.
$$
}

By inspection of $H_{t\Z}^{*,-1} V_{t\Z}$, we see that
\begin{itemize}
\item $h_1 = \Oo(t^2)$ and $h_j = \Oo(t)$ for all $j=2,\ldots, 6$.
\item The terms $h_1/t^2$ and $h_j/t$ for $j\geq 2$ are uniformly bounded in $t$ for $\abs{t}\leq t_0$, and when considered as functions of $u_1,u_2,v_1,v_2$, they are continuous and differentiable everywhere except at $u_2=v_2$. So, $\diag(1/t^2,1/t,\cdots,1/t)(\Lambda^*_{t\Z}  - \Lambda^*_{t\Z_0}) =  \Oo(\norm{\Z-\Z_0})$  provided that $\Z_0 = \{(a,b),(c,d)\}$ is such that $b\neq d$. Note that the case where $b=d$ can be dealt with similarly by changing the order of Gauss elimination.
\end{itemize}

To see that $\Psi^*_{\Z} = \Psi^*_{\Z_0}  + \Oo(\norm{\Z-\Z_0})$, observe that when considering $\Psi^*_\Z p$ as a function of $\Z$, it is differentiable everywhere except at $u_2=v_2$ and  provided that $\Z_0 = \{(a,b),(c,d)\}$ is such that $b\neq d$. Again, the case where $b=d$ can be dealt with similarly by changing the order of Gaussian elimination.

For the last claim, note that $L_{t\Z} \binom{1_2}{0_4} = \delta_6$. So, $\Gamma_{t\Z}^{*}p = \binom{1_2}{0_4}$ if and only if $(\Psi_{\Z}^* + \Lambda_{t\Z}^*) p = \delta_{6}$. From
$$
\Psi_{\Z} + \Lambda_{t\Z}= \Psi_{\Z_0} + \Oo(\norm{\Z-\Z_0}) + \Oo(t),
$$
we see that $\Psi_{\Z}$ is full rank whenever $\Psi_{\Z_0}$ is full rank and provided that $\norm{\Z-\Z_0}+\abs{t}$ is sufficiently small. Therefore,
$$
(\Psi_{\Z}^* + \Lambda_{t\Z}^*)^\dagger = \Psi_{\Z_0}^{*,\dagger} + \Oo(\norm{\Z-\Z_0}) + \Oo(t).
$$
\end{proof}

%

%
%

In the case of $\Z_0 = \{(0,0), (0,1)\}$,  the de Boor basis associated with $\Z_0$ is $\Bb_{\Z_0} = \{1,y,x,y^2,xy, y^3\}$. 
Moreover, in this case, by writing  $a_{j,k} = \dotp{\partial_x^j\partial_y^k \varphi(0)}{p}$ for $p\in \Hh$, 
\begin{equation}\label{eq-2spikes-Gamma}
\Psi_{\Z_0}^* p \begin{pmatrix}
a_{0,0}\\
a_{0,1}\\
a_{1,0}\\
a_{0,2}\\
a_{1,1}\\
a_{0,3}
\end{pmatrix}
,\qquad \Lambda_{t\Z_0}^* p =
\begin{pmatrix}
0\\0\\0\\
\sum_{j\geq 4} \frac{t^{j-2}}{j!} a_{j,0} (6-2j)\\
\sum_{j\geq 2} \frac{t^{j-1}}{j!} a_{1,j}\\
\sum_{j\geq 4} \frac{t^{j-3}}{j!} a_{0,j} (6j-12)
\end{pmatrix},
\end{equation}
and
$$
H_{t\Z}^{*,-1} = \begin{pmatrix}
1 & 0 & 0 & 0 & 0 & 0\\
0 & 0 & 0 & 0 & 1 & 0\\
0 & 0 & 1 & 0 & 0 & 0 \\
-6/t^2 & 6/t^2 & 0 & 0 & -4/t & -2/t\\
0 & 0 & -1/t & 1/t & 0 & 0 \\
12/t^3 & -12/t^3 & 0 & 0 & 6/t^2 & 6/t^2\\
\end{pmatrix},\quad H_{\Z}^* = \begin{pmatrix}
1 & 0 & 0 & 0 & 0 & 0\\
1 & t & 0 & t^2/2 & 0 & t^3/6\\
0 & 0 & 1 & 0 & 0 & 0\\
0 & 0 & 1 & 0 & t & 0\\
0 & 1 & 0 & 0 & 0 & 0\\
0 & 1 & 0 & t & 0 & t^2/2
\end{pmatrix}.
$$

In the following, let $\Pi_{t\Z} \eqdef P_{(\Im \Gamma_{t\Z})^\perp} = \Id - \Gamma_{t\Z}\Gamma_{t\Z}^\dagger$ be the orthogonal projection of $(\Im \Gamma_{t\Z})^\perp$.

\begin{prop}\label{prop:bd3}
Let $\Z_0 = \{(0,0), (0,1)\}$, $a\in \RR^2$ and $\Z \in \Xx^2$. Then, there exists a constant $C$ dependent only on $\varphi$ such that
$$\norm{ \Pi_{t\Z} \Gamma_{t\Z_0} \binom{a}{0_4}}_\Hh \leq C \max\{t^2 \norm{\Z-\Z_0}^2, t^3 \norm{\Z-\Z_0} \}.$$
\end{prop}

\begin{proof}
\begin{align*}
 \Pi_{t\Z} \Gamma_{t\Z_0} & =  \Pi_{t\Z} ( \Psi_{\Z_0} + \Lambda_{t\Z_0}) H_{t\Z_0}\\
 & =  \Pi_{t\Z} ( \Psi_{\Z} + \Lambda_{t\Z} + (\Psi_{\Z_0}-\Psi_{\Z}) + (\Lambda_{t\Z_0}  - \Lambda_{t\Z})) H_{t\Z_0}\\
 &=  \Pi_{t\Z} (  (\Psi_{\Z_0}-\Psi_{\Z}) + (\Lambda_{t\Z_0}  - \Lambda_{t\Z})) H_{t\Z_0}
 \end{align*}
Note that
$$
H_{t\Z_0} \binom{a}{0_4} = \begin{pmatrix}
a_1+a_2\\ ta_2\\ 0\\ t^2 a_2/2\\0 \\ t^3 a_2 /6 
\end{pmatrix}.
$$
Let $\Bb_{\Z}$ be the de Boor basis associated with Hermite interpolation at $\Z$. Note that $\Bb_{\Z} \supset \{ 1,x,y\}$ for all $\Z$. So, the first 3 entries of $\Psi_{\Z_0}-\Psi_{\Z}$ are all zero.
Let $\Z = \{(u_1,u_2),(v_1,v_2)  \}$, $p\in \Hh$, and let $a_{j,k } = \dotp{\partial_x^j\partial_y^k\varphi(0)}{p}$. Then, the 4th entry of $(\Psi_{\Z}-\Psi_{\Z_0}) p$ is 
$$a_{2,0} (u_1-v_1)^2/(u_2-v_2)^2 \lesssim \norm{\Z-\Z_0}^2
$$
and the 6th entry of $(\Psi_{\Z}-\Psi_{\Z_0}) p$ is 
$$
a_{3,0} \frac{(u_1-v_1)^3}{(u_2-v_2)^3} + 3a_{1,2}\frac{(u_1-v_1)}{(u_2-v_2)} + 3 a_{2,1} \frac{(u_1-v_2)^2}{(u_2-v_2)^2} \lesssim \norm{\Z-\Z_0}.
$$
Therefore, \begin{equation}\label{eq:nb1}
\norm{\Pi_{t\Z} (\Psi_{\Z_0}-\Psi_{\Z}) H_{t\Z_0}}_\Hh \leq C \max\{t^2\abs{\Z-\Z_0}_\infty^2, t^3 \norm{\Z-\Z_0} \},
\end{equation}
where $C$ depends only on $\varphi$.

 For any $\Z$, the 4th and 6th entries of $\Lambda_{t\Z}$ are $\Oo(t)$, the 1st entry of $(\Lambda_{t\Z_0} - \Lambda_{t\Z} )p$ is
 $$
 -t^2\frac{(u_1v_2-u_2v_1)^2}{2(u_2-v_2)^2} a_{2,0} + t^3\frac{ P_1(u_1,u_2,v_1,v_2)}{(u_2-v_2)^2},
 $$
 and the 2nd entry of $(\Lambda_{t\Z_0} - \Lambda_{t\Z})p$ is
 $$
 t \frac{(u_1-v_2)(u_1v_2-u_2v_1)}{(u_2-v_2)^2} a_{2,0} + t^2 \frac{P_2(u_1,u_2,v_1,v_2)}{(u_2-v_2)^2},
 $$
 where $P_1$ and $P_2$ are 4-variate polynomials.
 Therefore,
 $$\norm{\Pi_{t\Z} (\Lambda_{\Z_0}-\Lambda_z) H_{t\Z_0}}_\Hh \leq C' \max\{t^2 \norm{\Z-\Z_0}^2, t^3 \norm{\Z-\Z_0} \},$$
 where $C'$ depends only on $\varphi$.
 Combining this bound with \eqref{eq:nb1} gives the required result.
\end{proof}

\subsection{Non-degeneracy Transfer}\label{sec:nondegen-transfer}

We have already seen that $\eta_{V,t\Z_0}$ converges to $\eta_{W,\Z_0}$ as $t\to 0$. In this section, we show in Proposition~\ref{prop:degen_transfer} that non-degeneracy of $\eta_{W,\Z_0}$ in the sense of Definition~\ref{def:eta_W-nondegen} implies that any certificate defined via Hermite interpolation conditions at $\Z$ and sufficiently close to $\eta_{V,\Z_0}$ will also be a valid certificate saturating only at $\Z$. Furthermore, we show that in Proposition \ref{prop:bd2} that  $\eta_{V,t\Z_0}$ is non-degenerate and as a direct consequence of the main result of \cite{duval2015exact}, the solution of \eqref{eq-blasso} is stable with respect to $m_0$.

\begin{prop}\label{prop:degen_transfer}
Let $\Z_0=(z^*_1,z^*_2)\in \Xx^2$ and let $d_{\Z_0} = z^*_1-z^*_2$. Suppose that  $\eta_{W,\Z_0}(x)$ is non-degenerate.
Then, there exists $c_1,c_2,c_3>0$ such that given any $\eta\in \Cder{\infty}$, $t\in (0,c_2)$ and $\Z = (z_1,z_2)\in B(\Z_0,c_3)$ satisfying
\begin{itemize}
\item[(i)] $\eta(t z_i) = 1$, $\nabla \eta(t z_i) = 0$ for $i=1,2$,
\item[(ii)]  $\norm{\nabla^j \eta - \nabla^j \eta_{W,\Z_0}}_\infty \leq c_1$ for $\abs{j}\leq 5$,
\end{itemize}  we have that $\eta(x) <1$ for all $x\not\in t\Z$.
\end{prop}

\begin{proof}

For a contradiction, suppose that  for all $n>0$, there exists $\eta_n$, $\Z_n\in \Xx^2$ with $\norm{\Z_n-\Z_0}\leq 1/n$  and $t_n\in (0,1/n)$ such that for $j=1,2$,  $\eta_n(t_n z_{n,j}) = 1$, $\nabla \eta_n(t_n z_{n,j}) = 0$ with  $\norm{\eta_n - \eta_{W,\Z_0}}_\infty \leq 1/n$ and $x_n \not\in t_n \Z_n$ such that $\eta_n(x_n) = 1$. Note that since $\eta_{W,\Z_0}(x)<1$ for all $x\neq 0$, we must have that $x_n \to 0$ as $n\to \infty$. Let $x_n =  (f(t_n) u_n, g(t_n) v_n)$, where $\lim_{t\to 0}f(t) =0$ and $\lim_{t\to 0} g(t) = 0$. 

We first fix $n$ and derive some equations satisfied by $\eta_n$ and its derivatives. Without loss of generality, let $z_{n,1} = (1,0)$ and $z_{n,2} = (-1,0)$ (otherwise, we simply consider derivatives with respect to direction $d_\Z$ and $d_\Z^\perp$ instead of the canonical directions). 
To simplify notation, let us drop the subscript $n$ and simply write $\eta,t,c_1,c_2,u,v, \Z$ for $\eta_n,t_n,c_{n,1},c_{n,2},u_n,v_n, \Z_n$.
By expanding $\eta$ about $0$, we obtain
\begin{align*}
\eta(X) &= \sum_{\abs{\al}\leq 3} b_\alpha X^\al  + R_{0,4}(X) X_2^4 + R_{1,3}(X) X_1 X_2^3 + R_{2,2}(X) X_1^2 X_2^2 + R_{3,1}(X) X_1^3 X_2 \\
&+ b_{4,0} X_1^4 + R_{4,1}(X) X_1^4 X_2 + R_{5,0}(X) X_1^5,
\end{align*}
where given $\alpha\in \NN_0^2$, $$b_{\al}\eqdef \frac{\partial^\al \eta(0)}{\al !},\qquad
R_{\al }(X) \eqdef \frac{\abs{\al}}{\al !} \int_0^1 (1-s)^{\abs{\al}-1} \partial^\al \eta(s X) \mathrm{d}s.
$$
To simplify notation, in the following, we write
$
b_{\al} = R_{\al}
$ and note that thanks to assumption (ii), each of these terms is uniformly bounded in $n$.
Let
\begin{align*}
&\iota_0 \eqdef \eta((f(t) u, g(t) v)),\\
&\iota_1 \eqdef \eta(t z_1),
&\partial_{X_1} \iota_1 \eqdef \partial_{X_1} \eta(t z_1),\quad
&\partial_{X_2} \iota_1 \eqdef \partial_{X_2} \eta(t z_1),\\
&\iota_2 = \eta(t z_2),\quad
&\partial_{X_1} \iota_2 \eqdef \partial_{X_1} \eta(t z_2),\quad
&\partial_{X_2} \iota_2 \eqdef \partial_{X_2} \eta(t z_2).
\end{align*}
Then,
\begin{align*}
\gamma_1 \eqdef \frac{\iota_1+\iota_2}{2} =  b_{0,0} +t^{4} b_{4,0} = 1 \\
\gamma_2 \eqdef \frac{\iota_1-\iota_2}{2t} =  b_{1,0} + t^{2} b_{3,0}  + t^{4} b_{5,0} = 0 \\
\gamma_3 \eqdef \frac{\partial_{X_2} \iota_1 + \partial_{X_2}\iota_2}{2} = b_{0,1} + t^{2} b_{2,1} + t^{4} b_{4,1} =0 \\
\gamma_4 \eqdef \frac{\partial_{X_2} \iota_1 - \partial_{X_2} \iota_2}{2t} = b_{1,1} + t^{2} b_{3,1} = 0 \\
\gamma_5 \eqdef \frac{\partial_{X_1} \iota_1 - \partial_{X_1} \iota_2}{4t} =  b_{2,0} + 2 t^{2} b_{4,0} =0  \\
\gamma_6 \eqdef \frac{1}{2t^2}(\frac{\partial_{X_1} \iota_1-\partial_{X_1} \iota_2}{2}-\gamma_2) =  b_{3,0} +2 t^{2} b_{5,0}  = 0
\end{align*}

By subtracting appropriate multiples of $\{\gamma_j\}_{j=1}^6$ from $\iota_0$, we obtain
$$
0 =  \gamma \eqdef \iota_0 - \gamma_1 - g(t) u \gamma_2 - g(t) v \gamma_3 - f(t) g(t)  uv \gamma_4 - f(t)^{2}u^2 \gamma_5 - (f(t)^3 u^3 - t^{2} f(t) u)\gamma_6,
$$
and so,
\begin{equation}\label{eq:7cases}
\begin{split}
0 =& g(t)^{2} v^2 b_{0,2} + g(t)^{2}\left( g(t)^{2}v^3 b_{0,3} + f(t) uv^2 b_{1,2} + g(t)^{1} v^2 b_{0,4} + f(t)g(t)  uv^3 b_{1,3}+ f(t)^2 u^2 v^2 b_{2,2} \right)\\
&+  (f(t)^2 g(t) u^2 v - t^2 g(t) v) b_{2,1}  + u f(t) ( f(t)^1 g(t)  u^2 v -t^2 g(t) v) b_{3,1}
\\&+ (f(t)^{2} u^2 - t^2)^2 b_{4,0} + f(t) u (f(t)^{2} u^2 - t^2)^2  b_{5,0} + g(t) v(f(t)^{4} u^4-t^4) b_{4,1}.
\end{split}
\end{equation}

Note that since $\Z_n\to \Z$, we have $\lim_{n\to \infty} \partial^j_{d_{\Z_n}} \partial^k_{d_{\Z_n}^\perp} \eta_n(0) = \partial^j_{d_{\Z_0}} \partial^k_{d_{\Z_0}^\perp} \eta_{W,\Z_0}(0)$. By possible extracting a subsequence, assume that $(u_n,v_n)\to (u,v)$. Then, by considering the limit of \eqref{eq:7cases}, we arrive at one of the 7 following cases:
\begin{itemize}
\item[(i)] If $\lim_{t\to 0} f(t)/t \leq 1$ and $\lim_{t\to 0} g(t)/t^2 = 0$, then $\partial_{d_{\Z_0}}^4\eta_{W,\Z_0}(0) = 0$.
\item[(ii)] If $\lim_{t\to 0} t/f(t) \leq 1$ and $\lim_{t\to 0} g(t)/f(t)^2 = 0$, then $\partial_{d_{\Z_0}}^4\eta_{W,\Z_0}(0) = 0$.
\item[(iii)] If $\lim_{t\to 0} f(t)/t \leq 1$ and $\lim_{t\to 0} t^2/ g(t)= 0$, then $\partial_{d_{\Z_0}^\perp }^2\eta_{W,\Z_0}(0) = 0$.

\item[(iv)] If $\lim_{t\to 0} f(t)/t =0$ and $\lim_{t\to 0} g(t)/t^2 = 1$, then  $v^2\partial_{d_{\Z_0}^\perp}^2\eta_{W,\Z_0}(0) - v\partial_{d_{\Z_0}^\perp} \partial_{d_{\Z_0}}^2\eta_{W,\Z_0} (0) + \partial_{d_{\Z_0}}^4\eta_{W,\Z_0}(0)= 0$.

\item[(v)] If $\lim_{t\to 0} f(t)/t = 1$ and $\lim_{t\to 0} g(t)/t^2=1$, then $v^2\partial_{d_{\Z_0}^\perp}^2\eta_{W,\Z_0}(0) + v(u^2-1)\partial_{d_{\Z_0}^\perp} \partial_{d_{\Z_0}}^2\eta_{W,\Z_0}(0)  + (u^2-1)^2 \partial_{d_{\Z_0}}^4\eta_{W,\Z_0}(0) = 0$.

\item[(vi)] If $\lim_{t\to 0} t/f(t) = 0$ and $\lim_{t\to 0} g(t) /f(t)^2 = 1$:   $v^2\partial_{d_{\Z_0}^\perp}^2\eta_{W,\Z_0}(0) + v u^2  \partial_{d_{\Z_0}^\perp} \partial_{d_{\Z_0}}^2\eta_{W,\Z_0}(0)  + u^4 \partial_{d_{\Z_0}}^4\eta_{W,\Z_0}(0)= 0$.
\item[(vii)] If $\lim_{t\to 0} t/ f(t) \leq 1$ and $\lim_{t\to 0} f(t)^2/ g(t) = 0$, then $\partial_{d_{\Z_0}^\perp }^2\eta_{W,\Z_0}(0) = 0$.
\end{itemize}
In the above, the conclusion of cases (i)-(v) is obtained by taking the limit of \eqref{eq:7cases} after dividing by $t^4$, and the conclusion of cases (vi)-(vii) is obtained by taking the limit of \eqref{eq:7cases} after dividing by $g(t)^2$.

Therefore, it follows that there exists $(a,b)\in \RR^2$ such that 
$$
a^2 \partial_{d_{\Z_0}^\perp}^2 \eta_{W,\Z_0}(0) +a b \partial_{d_{\Z_0}}^2 \partial_{d_{\Z_0}^\perp} \eta_{W,\Z_0}(0) + \frac{1}{12} b^2 \partial_{d_{\Z_0}}^4 \eta_{W,\Z_0}(0) =0,
$$
which is a contradiction to the assumption that $\eta_{W,\Z_0}$ is non-degenerate. 
\end{proof}

\begin{rem}
From Proposition \ref{prop:degen_transfer}, it follows that $\eta_{V,t\Z_0}$ is a valid certificate for all $t$ sufficiently small. The following result shows that $\eta_{V,t\Z_0}$ is in fact non-degenerate, and therefore, as a direct consequence of the main result of \cite{duval2015exact}, the solution of \eqref{eq-blasso} is support stable with respect to $m_0$.
\end{rem}

\begin{prop}\label{prop:bd2}
Let $\Z \in \Xx^2$. Assume that $\partial_{d_\Z}^2\partial_{d\Z^\perp} \eta_{W,\Z}(0) = 0$. Then,
\begin{align*}
&\partial_{d_\Z}^2 \Phi_{t\Z}^* p_{V,t\Z} = \frac{t^2}{12} \partial_{d_\Z}^4 \eta_{W,\Z}(0) \binom{1}{1} + \Oo(t^3), \quad \partial_{d_\Z^\perp} \partial_{d_\Z} \Phi_{t\Z}^* p_{V,t\Z} = \frac{t}{2} \partial_{d_\Z}^2\partial_{d_\Z^\perp} \eta_{W,\Z}(0)\binom{-1}{1}  + \Oo(t^2), \\
& \partial_{d_\Z^\perp}^2 \Phi_{t\Z}^* p_{V,t\Z} = \partial_{d_\Z^\perp}^2 \eta_{W,\Z}(0) \binom{1}{1} + \Oo(t).
\end{align*}
Therefore, provided that $\eta_{W,\Z}$ is non-degenerate, then for all $t$ sufficiently small, $\nabla^2 \eta_{V,t\Z}(tz)\prec 0$ for all $z\in Z$. 
\end{prop}
\begin{proof}
Without loss of generality, let $\Z =\{(0,0), (0,1)\}$.  We shall show that
\begin{align*}
&\partial_y^2 \Phi_{t\Z}^* p_{V,t\Z} = \frac{t^2}{12} \partial_y^4 \eta_{W,\Z}(0) \binom{1}{1} + \Oo(t^3), \quad \partial_x \partial_y \Phi_{t\Z}^* p_{V,t\Z} =\frac{t}{2} \partial_y^2\partial_x \eta_{W,\Z}(0) \binom{-1}{1}+ \Oo(t^2), 
\\
& \partial_x^2 \Phi_{t\Z}^* p_{V,t\Z} = \partial_x^2 \eta_{W,\Z}(0) \binom{1}{1} + \Oo(t).
\end{align*}
Let $a\in \RR^2$.
By Taylor expanding about 0, we obtain
\begin{align*}
\partial_y^2 \Phi_{t\Z} a= \Psi_\Z \tilde V_{t\Z}^{(1)} a + a_2 \sum_{j=2}^\infty \frac{t^j}{j!}\partial_y^{j+2}\varphi(0), 
\qquad \partial_x\partial_y \Phi_{t\Z} a = \Psi_\Z \tilde V_{t\Z}^{(2)} a + a_2 \sum_{j\geq 1} \frac{t^j}{j!} \partial_y^{j+1}\partial_x \varphi(0)
\end{align*}
where
$$
\tilde  V_{t\Z}^{(1)} = \begin{pmatrix}
0& 0\\
0&0\\
0&0\\
1&1\\
0&0\\
0&t
\end{pmatrix}\qquad \tilde  V_{t\Z}^{(2)} = \begin{pmatrix}
0& 0\\
0&0\\
0&0\\
0&0\\
1&1\\
0&0
\end{pmatrix},
$$
and
$$
\partial_x^2 \Phi_{t\Z} a = a_1 \partial_x^2\varphi(0) + a_2 \left( \sum_{j\geq 0} \frac{t^j}{j!}\partial_y^j\partial_x^2 \varphi(0)  \right).
$$

We first consider $\partial_y^2 \Phi_{t\Z}^* p_{V,t\Z} $:
Recall the definition of $\Lambda_{t\Z}$ from \eqref{eq-2spikes-Gamma} and observe that
\begin{equation}\label{bd1}
\tilde V_{t\Z}^{(1),*}(\Psi^*_{\Z} + \Lambda_{t\Z}^*)\Gamma_{t\Z}^{*,\dagger}\binom{1_2}{0_4} = \tilde V_{t\Z}^{(1),*}H_{t\Z}^{*,-1} \Gamma_{t\Z}^* \Gamma_{t\Z}^{*,\dagger}\binom{1_2}{0_4}=\tilde V_{t\Z}^{(1),*} \delta_6 = 0.
\end{equation}
Moreover, using the fact that $p_{V,t\Z} = p_{W,\Z} + \Oo(t)$,
\begin{equation}\label{bd2}
\binom{0}{ \sum_{j=2}^\infty \frac{t^j}{j!}\partial_y^{j+2}\varphi(0)} p_{V,t\Z} = \binom{0}{  \frac{t^2}{2}\dotp{p_{V,t\Z}}{\partial_y^{4}\varphi(0)} + \Oo(t^3)}
=  \binom{0}{  \frac{t^2}{2} \partial_y^{4} \eta_{W,\Z}(0) + \Oo(t^3)}.
\end{equation}
Also,
\begin{equation}\label{bd3}
\begin{split}
\tilde V_{t\Z}^{(1),*}\Lambda_{t\Z}^* p_{V,t\Z} &=\tilde V_{t\Z}^{(1,*)} \begin{pmatrix}
0\\0\\0\\
\frac{-t^2}{12} \dotp{\partial_y^4 \varphi(0)}{p_{V,t\Z}} + \Oo(t^3)
\\
\frac{t}{2} \dotp{\partial_x\partial_y^2 \varphi(0)}{p_{V,t\Z}} + \Oo(t^2)\\
\frac{t}{2}\dotp{\partial_y^4\varphi(0)}{p_{V,t\Z}} + \Oo(t^2)
\end{pmatrix}\\
&
= \tilde V_{t\Z}^{(1),*} \begin{pmatrix}
0\\0\\0\\
\frac{-t^2}{12} \partial_y^4 \eta_{W,\Z}(0) + \Oo(t^3)
\\
t \partial_y^2 \partial_x \eta_{W,\Z}(0)  + \Oo(t^2) \\
\frac{t}{2}\partial_y^4\eta_{W,\Z}(0)  + \Oo(t^2)
\end{pmatrix} = \binom{\frac{-t^2}{12} \partial_y^4 \eta_{W,\Z}(0)}{\frac{5t^2}{12}\partial_y^4\eta_{W,\Z}(0)} + \Oo(t^3).
\end{split}
\end{equation}
Summing \eqref{bd1}, \eqref{bd2} and \eqref{bd3} gives the required bound on $\partial_y^2\Phi^*_{t\Z}p_{V,t\Z}$.

For the second bound,
$$
\partial_x\partial_y\Phi_{t\Z} = \Psi_{\Z} \tilde V_{t\Z}^{(2)} + \binom{0}{\sum_{j\geq 1} \frac{t^j}{j!}\partial_y^{j+1}\partial_x \varphi(0)} = (\Psi_z+\Lambda_{t\Z})\tilde V^{(2)}_{t\Z} -\Lambda_{t\Z} \tilde V^{(2)}_{t\Z} + \binom{0}{\sum_{j\geq 1} \frac{t^j}{j!}\partial_y^{j+1}\partial_x \varphi(0)}.
$$
As before, $\tilde V^{(2),*}_{t\Z} (\Psi_z+\Lambda_{t\Z})^* p_{V,t\Z} = 0$, 
$$
\dotp{\sum_{j\geq 1} \frac{t^j}{j!}\partial_y^{j+1}\partial_x \varphi(0)}{p_{V,t\Z}} = t \partial_y^2\partial_x \eta_{W,\Z}(0) + \Oo(t^2),
$$
and
$$
\tilde V_{t\Z}^{(2),*} \Lambda_{t\Z}^* p_{V,t\Z} = \frac{t }{2} \partial_y^2\partial_x \eta_{W,\Z}(0) \binom{1}{1} + \Oo(t^2).
$$
So, $\partial_y\partial_x\Phi_{t\Z}^* p_{V,t\Z} =\frac{t}{2} \partial_y^2\partial_x \eta_{W,\Z}(0) \binom{-1}{1} + \Oo(t^2)$.

The proof of the last bound follows because 
$$
\partial_x^2\Phi_{t\Z}^* p_{V,t\Z} = \partial_x^2\Phi_{t\Z}^* (p_W + \Oo(t)) = \partial_x^2 \eta_{W,\Z}(0) \binom{1}{1} + \Oo(t).$$
\end{proof}

\subsection{Proof of Theorem \ref{thm-twospikes}}

First note that if Theorem \ref{thm-twospikes} is true for $\Z_0 \eqdef \Z$ for some fixed $\Z\in \Xx^2$, then given any $c\in \Xx$, the result is also true for $\Z_0\eqdef c+\Z$. Let $\eta^{\Phi,y,\la}$ be the solution to the dual formulation of \eqref{eq-blasso}, and let $\eta_{V}^{\Psi,y}$ the the associated precertificate. Then, by letting $T:\Xx\to \Xx$, $z\mapsto z+ tc$ and thanks to the reparametrization observations of Appendix \ref{sec-reparam}, we have 
$$
\eta_V^{\Phi,\Phi m_{a,t(\Z_0+c)}} = \eta_V^{\Phi\circ T_\sharp,\Phi \circ T_\sharp m_{a,t \Z_0}}(\cdot - tc).
$$
Therefore $\eta_{W,\Z_0+c}^\Phi = \eta_{W,\Z_0}^{\Phi\circ T_\sharp}$. So, $\eta_{W,\Z_0+c}^\Phi $ is non-degenerate if and only if $\eta_{W,\Z_0}^{\Phi\circ T_\sharp}$ is non-degenerate. Moreover, since (provided that $\la$ and $w$ satisfies the conditions of Theorem \ref{thm-twospikes}) non-degeneracy of $\eta_{W,\Z_0}^{\Phi\circ T_\sharp}$ implies that $\eta^{\Phi\circ T_\sharp,\Phi\circ T_\sharp m_{a_0,t\Z_0}+w,\la}$ saturates only at $\Z$ and $m_{a,\Z}$ is the unique solution with $a$ and $\Z$ satisfying \eqref{eq:error-main}, we know that  $\eta^{\Phi ,\tilde y,\la}$ with $\tilde y = \Phi   m_{a_0,t(\Z_0+c)}+w$, saturates only at $\Z+tc$ and $m_{a,\Z+tc}$ is the unique solution. Therefore, without loss of generality, it suffices to prove Theorem \ref{thm-twospikes} for $\Z_0 = \{(0,0), (a,b)\}$ for some $(a,b)\in \Xx$. Furthermore,  we simply consider  $\Z_0 = \{(0,0), (0,1)\}$, since otherwise, in the following, we can simply consider derivatives with respect to $d_{\Z_0}$ and $d_{\Z_0}^\perp$ instead of the canonical directions.

\subsubsection{Implicit Function Theorem}
\label{sec-proof-implicit-func-thm}

From the first order optimality conditions of \eqref{eq-blasso}, we have that $m_{a,\Z}$ solves \eqref{eq-blasso} if and only if
\begin{equation}\label{eq:dual-cert-cond}
\frac{\Phi^*(\Phi m_{a_0,\Z_0} + w -\Phi m_{a,\Z})}{\lambda} \in \partial \abs{\cdot}(m_{a,\Z}).
\end{equation}
Therefore,  we aim to construct a $\Cder{1}$ mapping  $g: (\lambda,w)\in \RR \times \Hh \mapsto (a,\Z)\in \RR^2\times \Xx^2$ such that $(a,\Z)$ satisfies \eqref{eq:dual-cert-cond}. Furthermore, bounds on the derivatives of $g$ will provide conditions on the required speed at which $(\lambda,w)$ converge to 0. To this end, following \cite{duval2015exact} and \cite{2017-denoyelle-jafa}, let $u= (a,\Z)$ and $v=(\lambda,w)$, and define
$$
f_t(u,v)\eqdef \Gamma_{t\Z}^* (\Phi_{t\Z} a- \Phi_{t\Z_0} a_0 - w) + \lambda \binom{1_N}{0_{2N}}.
$$
To construct candidate solutions to \eqref{eq-blasso},  we search for parameters $u$ and $v$ for which $f_t(u,v)=0$.

For $a\in\RR^2$ and $\al\in \NN_0^2$, let $(\partial^\alpha \Phi_{t\Z})a \eqdef \sum_{j=1}^2 a_j \partial^\alpha \varphi(t z_j) $. Then,
the derivatives of $f_t$ are
\begin{equation}\label{eq:f_t_d}
\begin{split}
\partial_u f_t(u,v) &= \Gamma_{t\Z}^* \Gamma_{t\Z} J_{ta} + t \begin{pmatrix}
0 & \diag(\partial_x \Phi_{t\Z}^* A) & \diag(\partial_y \Phi_{t\Z}^* A)\\
0 & \diag(\partial_x^2 \Phi_{t\Z}^* A) & \diag(\partial_y \partial_x\Phi_{t\Z}^* A)\\
0 & \diag(\partial_y\partial_x \Phi_{t\Z}^* A) & \diag(\partial_y^2 \Phi_{t\Z}^* A)
\end{pmatrix}\\
\partial_v f_t(u,v) &= \begin{pmatrix}
\binom{1_N}{0_{2N}} & -\Gamma_{t\Z}^*
\end{pmatrix}
\end{split}
\end{equation}
where
$$
J_{ta} \eqdef \begin{pmatrix}
\Id_N & 0 & 0\\
0 & t\diag(a) & 0 \\
0 & 0 & t\diag(a)
\end{pmatrix}, \qquad A \eqdef \Phi_{t\Z} a- \Phi_{t\Z_0} a_0 - w.
$$

So, $f_t$ is a continously diffferentiable fucntion, $\partial_u f_t(u_0,0) = \Gamma_{t\Z_0}^* \Gamma_{t\Z_0} J_{ta_0}$ is invertible by Proposition~\ref{prop:cvgence} and $f_t(u_0,0) = 0$. Therefore, we may apply the Implicit Function Theorem to deduce that there exists a neighbourhood $V_t$ of $0$ in $\RR\times \Hh$, a neighbourhood $U_t$ of $u_0$ in $\RR^2\times \Xx^2$ and a $\Cder{1}$ function $g_t:V_t\to U_t$  such that for all $(u,v)\in U_t\times V_t$, $f_t(u,v) = 0$ if and only if $u=g_t(v)$. Furthermore, the derivative of $g_t$ is
\begin{equation}\label{eq:deriv-g}
\mathrm{d}g_t(v) = -\left( \partial_u f_t(g_t(v),v)\right)^{-1} \partial_v f_t(g_t(v),v).
\end{equation}
So, to prove Theorem~\ref{thm-twospikes}, given $(\lambda,w)$, for $(a,\Z)=g((\lambda,w))$, we simply need to establish the following two facts.\begin{enumerate}
\item  $g_t$ is well defined on a region $V_t$ which contains a ball of radius on the order of $t^4$.

\item $m_{a,\Z}$ is a solution of \eqref{eq-blasso}, i.e. it satisfies \eqref{eq:dual-cert-cond}. To this end, we define the associated certificate as
\begin{equation}\label{eq:candidate_pV}
p_{\lambda,t} \eqdef \frac{\Phi^*(\Phi_{\Z} a - \Phi_{t\Z_0} a_0 -w)}{\lambda},\qquad \eta_{\lambda,t}\eqdef \Phi^* p_{\lambda,t}
\end{equation} 
and show (Proposition~\ref{prop:implicit_dual_convergence}) that $p_{\lambda}$ converges to $p_{V,t}$ as $(\lambda,w)\to 0$. Therefore, by Lemma~\ref{prop:cvgence} and Theorem~\ref{prop:degen_transfer}, $p_{\lambda}$ must satisfy \eqref{eq:dual-cert-cond}.
\end{enumerate}

We remark that although the key steps of this proof are the same as the 1-D proof presented in \cite{2017-denoyelle-jafa}, the technical details differ due to the anisotropy of the limiting certificate $\eta_{W,\Z_0}$ and since some of the proofs in \cite{2017-denoyelle-jafa} rely on purely 1-D tools.

\subsubsection{Bounds on $V_t$}
\label{sec-proof-thm-boundvt}

For $r>0$, let $B(0,r) \subset \RR\times \Hh$ be defined as $B(0,r)\eqdef \enscond{(\la,w)}{\la\in [0,r), \norm{w}_\Hh < r}$. 

To show that we can construct a function $g_t^*$ which is defined on a ball of radius $t^4$, let $V_t^*$ be defined as follows: $V^*_t \eqdef \bigcup_{v\in \Vv} V$, where $\Vv$ is the collections of all open sets $V\subset \RR\times \Hh$ such that
\begin{itemize}
\item $0\in V$,
\item $V$ is star-shaped with respect to $0$,
\item $V\subset B(0,c_* t^{4}
)$ where $c_*>0$ is the constant defined in Lemma~\ref{lem:Ginv}.
\item there exists a $\Cder{1}$ function $g:V\to \RR^2\times \Xx^2$ such that $g(0) = u_0$, $f_t(g(v),v)=0$ for all $v\in V$,
\item $g(V)\subset   \Bb_{c_*}(a_0) \times \Bb_{t c_*}(\Z_0) $.

\end{itemize}
Note that the definition of this set $V^*_t$ is the same as in \cite{2017-denoyelle-jafa}, except for the last condition, where we require that $\norm{\Z-\Z_0}\leq c_* t$ for all $\Z$ such that $ (a,\Z)\in g(V)$. This is natural, since we eventually require that the distance between $p_{\lambda,t}$ and $p_{V,\Z_0}$ is $\Oo(t)$. As explained in \cite[Section 4.3]{2017-denoyelle-jafa}, this set $V^*_t$ is well defined and non-empty.
We may therefore define a function $g_t^*:V_t^*\to \RR^2 \times \Xx^2$ where
$$
g_t^*(v) \eqdef g(v), \quad\text{if} \; v\in V, \; V\in\Vv, \; \text{and} \; g \text{ is the corresponding function}.
$$

 The goal of the remainder of this subsection is to show that $V_t^*$ contains a ball of radius $t^4$.

\begin{lem}\label{lem:Ginv}
Let
$$
	G_{t\Z}(\lambda,w) \eqdef \Psi_{t\Z}^* \Psi_{t\Z} + t H_{t\Z}^{*,-1} F_{t\Z} J_{ta}^{-1} H_{t\Z}^{-1}
$$
where
$$
F_{t\Z} \eqdef -\begin{pmatrix}
0 & 0 & 0\\
0 & \diag(\partial_x^2 \Phi_{t\Z}^* q_{t\Z}) & \diag(\partial_y \partial_x\Phi_{t\Z}^* q_{t\Z})\\
0 & \diag(\partial_y\partial_x \Phi_{t\Z}^* q_{t\Z}) & \diag(\partial_y^2 \Phi_{t\Z}^* q_{t\Z})
\end{pmatrix}
$$
and
$$
q_{t\Z} \eqdef \lambda \Gamma_{t\Z}^{*,\dagger} \binom{1_N}{0_{2N}} + \Pi_{t\Z} w + \Pi_{t\Z} \Gamma_{t\Z_0} \binom{a_0}{0}.
$$
There exists $c_*>0$ such that for all $\Z\in \Xx^2$ and $\la \in \RR$ and $w\in \Hh$ with  $\norm{\Z-\Z_0}_\infty\leq c_* t$, $\lambda \leq c_* t^4$ and $\norm{w}\leq c_* t^4$,
$G_{t\Z}(\lambda,w)$ is invertible and has inverse bounded by $3\norm{(\Psi_{\Z_0}^* \Psi_{\Z_0})^{-1}}$.
Moreover, if $f_t(u,v)=0$, then $\partial_u f_t(u,v) = H_{t\Z}^* G_{t\Z}(\lambda,w) H_{t\Z} J_{ta}$.
\end{lem}
\begin{proof}
First observe that since for $\abs{\alpha}=2$, $\partial^\alpha \Phi^* \Gamma_{t\Z}^{*,\dagger} \binom{1_N}{0_{2N}}$ converges uniformly to $\partial^\alpha \Phi^* p_{W,z}$ and $\norm{\Z-\Z_0}_\infty \leq c_0 t$,  $\abs{\partial^\alpha \Phi_{t\Z}^* \Gamma_{t\Z}^{*,\dagger} \binom{1_N}{0_{2N}} }$ is uniformly bounded.
Therefore, from Proposition~\ref{prop:bd3}, for $\abs{\alpha}=2$,
\begin{align*}
\abs{\partial^\alpha \Phi_{t\Z}^* q_{t\Z})}_\infty &\lesssim \lambda + \norm{w}  + t^2 \abs{\Z-\Z_0}_\infty^2 + t^3 \abs{\Z-\Z_0}_\infty.
\end{align*}
Therefore, $\norm{F_{t\Z}^{-1}}\leq C c_0 t^4$ for some constant $C$ which depends only on $\varphi$.

Recalling the definition of $ H_{t\Z}$ and $J_{ta}$, we have that 
\begin{align*}
\norm{tH_{t\Z}^{*,-1}F_{t\Z} J_{ta}^{-1} H_{t\Z}^{-1}}
&= \norm{t^2 \diag(1,\frac{1}{t},\frac{1}{t},\frac{1}{t^2},\frac{1}{t^2},\frac{1}{t^3}) H_{\Z}^{*,-1} F_{t\Z} J_a^{-1} H_{\Z}^{-1} \diag(1,\frac{1}{t},\frac{1}{t},\frac{1}{t^2},\frac{1}{t^2},\frac{1}{t^3})}
\\
&\leq \frac{\norm{H_{\Z}^{-1} }^2 \norm{J_{a}^{-1}} \norm{F_{t\Z}}}{t^4}
\leq C \norm{H_{\Z}^{-1} }^2 \norm{J_{a}^{-1}} c_0.
\end{align*}
Therefore, by the above bound and by Lemma~\ref{prop:cvgence}, we have that  $$\norm{G_{t\Z}- \Psi_{\Z_0}^* \Psi_{\Z_0} }\leq  C' \abs{\Z-\Z_0}_\infty + C \norm{H_{\Z}^{-1} }^2 \norm{J_{a}^{-1}} c_0,
$$ where $C'$ depends only on $\varphi$ 
and the required result follows by choosing $c_0$ to be sufficiently small.

From \eqref{eq:f_t_d} and since $\Gamma_{t\Z} = \Psi_{t\Z} H_{t\Z}$ by Lemma \ref{prop:cvgence}, if $f_t(u,v) = 0$, then
$$
\partial_u f_t(u,v) = H_{t\Z}^*\left( \Psi_{t\Z}\Psi_{t\Z} + t H_{t\Z}^{*,-1} g \begin{pmatrix}
0 & 0 & 0\\
0 & \diag(\partial_x^2 \Phi_{t\Z}^* A) & \diag(\partial_y \partial_x\Phi_{t\Z}^* A)\\
0 & \diag(\partial_y\partial_x \Phi_{t\Z}^* A) & \diag(\partial_y^2 \Phi_{t\Z}^* A)
\end{pmatrix} J_{ta}^{-1} H_{t\Z}^{-1} \right) H_{t\Z} J_{ta}
$$
where $A = \Phi_{t\Z} a - \Phi_{t\Z_0} a_0 - w$. Therefore, it is enough to show that $A = -q_{t\Z}$.

 Since $f_t(u,v) = 0$,
\begin{equation}\label{eq:app1}
\Gamma_{t\Z}^*(\Phi_{t\Z} a - \Phi_{t\Z_0} a_0 - w_n) + \la \binom{1_N}{0_{2N}} = 0.
\end{equation}
We can rewrite \eqref{eq:app1} as
\begin{equation}\label{eq:app2}
-\Gamma_{t_n \Z_n}^*\Gamma_{t_n \Z_n} \binom{a_n}{0_{2N}} =  - \Gamma_{t_n \Z_n}^*\Gamma_{t\Z_0} \binom{a_0}{0_{2N}} -\Gamma_{t_n \Z_n}^* w_n + \la_n\binom{1_N}{0_{2N}}.
\end{equation}
By applying $\Gamma_{t\Z}(\Gamma_{t\Z}^* \Gamma_{t\Z})^\dagger$ to both sides, we obtain
$$
- \Gamma_{t\Z} \binom{a}{0_{2N}} =  - \Gamma_{t\Z} \Gamma_{t\Z}^\dagger \Gamma_{t \Z_0} \binom{a_0}{0_{2N}} -\Gamma_{t\Z} \Gamma_{t\Z}^\dagger w + \la \Gamma_{t\Z}^{*,\dagger} \binom{1_N}{0_{2N}}.
$$
Therefore, 
\begin{align*}
A = 
\Gamma_{t\Z} \binom{a}{0_{2N}} - \Gamma_{t\Z_0} \binom{a_0}{0_{2N}} - w =-\left( \la \Gamma_{t\Z}^{*,\dagger} \binom{1_N}{0_{2N}} + \Pi_{t\Z} w + \Pi_{t\Z}\Gamma_{t\Z_0}\binom{a_0}{0_{2N}}\right) = -q_{t\Z},
\end{align*}
as required.

\end{proof}

\begin{cor}\label{cor}
Let $c_0\leq c_*$ where $c_*$ is as in Lemma \ref{lem:Ginv}.
Suppose that $\abs{\Z-\Z_0}_\infty\leq c_0 t$, $\lambda \leq c_0 t^4$ and $\norm{w}\leq c_0 t^4$.
Then, there exists a constant dependent only on $\varphi$, $a_0$, $\Z_0$ such that
$$
\norm{\partial_w g_t^*(v)} \leq  \frac{C}{t^3}.
$$
and 
 $$
\norm{\partial_\lambda g_t^*(v)} \leq C(c_0 t^{-3} + t^{-2}),
 $$
\end{cor}
\begin{proof}

\begin{align*}
\mathrm{d}g_t^*(v) &= -J_{ta}^{-1} H_{t\Z}^{-1} G_{t\Z}(\lambda, w)^{-1} H_{t\Z}^{*-1} \begin{pmatrix}
\binom{1_N}{0_{2N}} & -H_{t\Z}^* \Psi_{t\Z}^*
\end{pmatrix}\\
&= J_a^{-1} H_{\Z}^{-1} \diag(1,t^{-1},t^{-1},t^{-2},t^{-2}, t^{-3}) G_{t\Z}(\lambda, w)^{-1} (\delta_{3N}, \Psi^*_{t\Z}).
\end{align*}
Therefore,
$$
\partial_w g_t^*(v) = \Oo(t^{-3}).
 $$
Recall from Lemma~\ref{lem:Ginv} that
$G_{t\Z}=\Psi_{\Z_0}^* \Psi_{\Z_0} + C' \abs{\Z-\Z_0}_\infty + C \norm{H_{\Z}^{-1} }^2 \norm{J_{a}^{-1}} c_0$.

Note that by ordering $\Psi_{\Z_0}$ as $\{\varphi_{0,0},\varphi_{0,1}, \varphi_{0,2}, \varphi_{1,0}, \varphi_{1,1}, \varphi_{0,3}\}$, $\Psi_{\Z_0}^* \Psi_{\Z_0}$ is a checkerboard matrix and its $(6,1)^{th}$ entry is zero. Therefore, $(\Psi_{\Z_0}^* \Psi_{\Z_0})^{-1}$ is also a checkerboard matrix with  zero as its $(6,1)^{th}$ entry. So, $G_{t\Z}^{-1} = (\Psi_{\Z_0}^* \Psi_{\Z_0})^{-1} + C'' c_0$, where $C''$ is a constant dependent only on $\varphi$, $a_0$ and $\Z_0$. 
So, $$
\norm{\partial_\lambda g_t^*(v)} \leq  C(c_0 t^{-3} + t^{-2}),
 $$
where $C$ is a constant dependent only on $\varphi$ and $(a_0, \Z_0)$.
\end{proof}

We are finally ready to show that $V_t^*$ contains a ball with radius on the order of $t^4$:
\begin{prop}\label{prop:V_t_radius}
There exists $C>0$ such that for all $t\in (0,t_0)$,
$$
V_t^* \supset B(0, C t^4).
$$
where $C\sim c_*$.
\end{prop}
\begin{proof}
Let $v\in \RR\times \Hh$ be such that $\max(\lambda,\norm{w})=1$. Let
$$
R_v = \sup\enscond{r\geq 0}{rv\in V_t^*}.
$$
First note that $R_v\in (0,C t^4)$, and since $g_t^*$ is uniformly continuous on $V_t^*$, $g_t^*(R_v v) \eqdef \lim_{r\to R_v} g(rv)$ is well defined. Moreover, $f_t(g_t^*(R_v v) ,R_v v) = 0$.

By maximality of $V_t^*$, it is necessarily the case that $g_t^* (R_v v) \in \partial(\Bb_{c_0}(a_0)\times \Bb_{t c_*}(\Z_0))$ (otherwise, we can apply the implicit function theorem to construct a neighbourhood $V\in \Vv$ such that $V_t^* \subsetneq V$).

Suppose that $g_t^*(R_v v) \in \overline{ \Bb_{c_*}(a_0)} \times \partial(\Bb_{tc_*}(\Z_0))$. Then, for $(a,\Z) = g_t^*(R_v v)$,
$$
c_* t = \norm{\Z-\Z_0} \leq \int_0^1\abs{\mathrm{d}g_t^*(sR_v v) \cdot R_v v}_\infty \mathrm{d}s \leq \frac{M}{t^3}R_v\implies R_v \geq \frac{c_* t^4}{M}.
$$
On the other hand, if $g_t^*(R_v v) \in \partial(\Bb_{c_*}(a_0)) \times \overline{\Bb_{tc_*}(\Z_0)}$, then $R_v \geq \frac{c_* t^3}{M}$.
Repeating this for all $v\in \RR\times \Hh$ with unit norm yields the required result.

\end{proof}

\subsubsection{Use of Non-degeneracy}
\label{sec-proof-thm-use-nondegen}

Throughout this section, given $(a,\Z) = g_t(\lambda,w)$, recall the definition of $p_{\lambda,t}$ and $\eta_{\lambda,t}$ from \eqref{eq:candidate_pV}.

\begin{prop}\label{prop:implicit_dual_convergence}
Let $\epsilon >0$. Then, there exists $c_0>0$ and $t_0>0$ such that
for all $\Z,\lambda,w,t$ with $0<t< t_0$,  $\lambda \leq c_0 t^4$ and $\norm{w}\leq c_0 t^4$ and $\norm{w}/\lambda\leq c_0$,
we have that $$
\norm{p_{\lambda,t} -p_{W,\Z_0}} \leq \epsilon.
$$
\end{prop}

\begin{proof}
By Proposition \ref{prop:V_t_radius}, there exists $c$ such that for all $(\la,w) \in B(0,c_0 t^4)$, with $c_0\leq c$,  $g_t^*$ is well defined.  For $(a,\Z) = g_t^*(\lambda,w)$,
\begin{align*}
p_{\lambda,t} = \Gamma_{t\Z}^{*,\dagger} \binom{1_N}{0_{2N}} + \Pi_{t\Z} \frac{w}{\lambda} + \frac{1}{\lambda} \Pi_{t\Z} \Gamma_{t\Z_0}\binom{a_0}{0}
= p_{W,\Z_0} + \Oo(t) + \Oo(\frac{\norm{w}}{\lambda}) +  \frac{1}{\lambda} \Pi_{t\Z} \Gamma_{t\Z_0}\binom{a_0}{0}.
\end{align*}
To bound the last term on the RHS,
\begin{align*}
&\norm{ \frac{1}{\lambda} \Pi_{t\Z} \Gamma_{t\Z_0}\binom{a_0}{0}}
\leq \frac{C}{\lambda}\max\{t^2\abs{\Z-\Z_0}_\infty^2, t^3\abs{\Z-\Z_0}_\infty \}\\
&\leq \frac{C}{\lambda}\max\{\frac{\norm{w}^2}{t^4}, \norm{w}\} +  \frac{C}{\lambda}\max\{t^2(L^2 c_0^2 t^{-6} + t^{-4})\lambda^2, Lc_0\lambda \}\\
&\leq C\left(\frac{\norm{w}}{\lambda} + t^{-4} c_0^2 \lambda + c_0 \right),
\end{align*}
where the first inequality follows from Proposition~\ref{prop:bd3} and the second inequality follows from Corollary~\ref{cor}. The result now follows by choosing $c_0$ sufficiently small.
\end{proof}

\begin{proof}[Proof of Theorem~\ref{thm-twospikes}]
By Proposition~\ref{prop:implicit_dual_convergence}, if $(a,\Z) = g_t^*(\lambda,w)$, then since $p_{\la,t}$ can be made arbitrarily close to $p_{W,\Z_0}$, we can apply Proposition \ref{prop:degen_transfer} to conclude that $p_{\lambda,t}$ is a valid certificate and hence the (unique) solution to the dual problem of \eqref{eq-blasso}. Moreover, $\eta_{\lambda,t}$  attains the value 1 only at the points in $\Z$. Therefore, the support of any solution of \eqref{eq-blasso} is contained in $\Z$ and by invertibility of $\Phi_{\Z}^*\Phi_{\Z}$, it follows that $m_{a,\Z}$ is the unique solution of \eqref{eq-blasso}. Finally, the bounds on $\norm{(a,\Z)-(a_0,\Z_0)}$ is a direct consequence on the bounds on the differential $\mathrm{d}g_t^*$.
\end{proof}

\subsection{Limitations}

The key idea behind the stability result of Theorem~\ref{thm-twospikes} is Proposition~\ref{prop:degen_transfer}: any certificate which is sufficiently close to $\eta_{W,\Z_0}$ is also a valid certificate. We have only proved this result in the case of a pair of spikes, although a similar proof technique can be applied to the case where $\Z_0$ consists of $N$ aligned points in direction $d_{\Z_0}$, with the natural extension of the non-degeneracy condition (c.f. construction of $\eta_{W,\Z_0}$ from Example \ref{exp:further-etaW}) being:
$$
\begin{pmatrix}
\partial_{d_{\Z_0}^\perp}^2 \eta_{W,\Z_0}(0) &  \frac{1}{N!} \partial_{d_{\Z_0}^\perp} \partial_{d_{\Z_0}}^N \eta_{W,\Z_0}(0) \\
\frac{1}{N!}
\partial_{d_{\Z_0}^\perp} \partial_{d_{\Z_0}}^N \eta_{W,\Z_0}(0)  & \frac{2}{(2N)!}\partial_{d_{\Z_0}}^{2N} \eta_{W,\Z_0}(0) 
\end{pmatrix}\prec 0.
$$ 
However, Proposition~\ref{prop:degen_transfer} is in general not valid and therefore, the question of whether there is support stability in the case of more than 2 spikes remains open. The purpose of this section is to present some examples to illustrate this phenomenon. Note also that there exists examples (see the Gaussian mixture example from Section~\ref{sec-numerics}) where one can numerically observe support stability when recovering a pair of spikes, but not in the case of 3 or more spikes.

In the following examples, consider let $\Phi$ be a convolution operator, i.e.  $\varphi(x) = \tilde \varphi(x-\cdot)$.

\begin{prop}[Case $N=3$ ] Let $\Z = \{z_1,z_2,z_3\} \in \Xx^4$ be 3 points which are not colinear. Let  $x$ is any point in the interior of the convex hull of $\Z$. 
Let $$
p_t = \argmin\enscond{\norm{p}}{(\Phi^* p)(tv)=1, \nabla (\Phi^*p)(tv)=0, \; \forall v\in \Z, \; \Phi^*p(t x)=1}.
$$
Then, $\lim_{t\to 0}\norm{p_t - p_{W,z}} = 0$.

\end{prop}
\begin{proof}
The least interpolant space associated to Hermite interpolation at $\Z$ contains the polynomial space of degree 2,
$
\Pi_2^2.
$
Moreover, by Lemma~\ref{lem:odd_vanish}, since $\Phi$ is a convolution operator, $\nabla^k\eta_{W,\Z}(0) = 0$ for all odd integers $k$. Therefore, $\nabla^3 \eta_{W,\Z}(0)=0$ for $k=1,2,3$.
On the other hand, the least interpolant space associated to Hermite interpolation at $\Z$ plus Lagrange interpolation at  $x$ is  $\Pi_3^2$. Therefore, $p_{W,\Z} = p_{t}+\Oo(t)$.  
\end{proof}

\begin{prop}
Let $\tilde \varphi$ be the Gaussian kernel.
Let $\Z=\{(1,1), (-1,1), (1,-1), (-1,-1)\}$. Let $(u,v)\in \RR^2$ be such that $u^2+v^2=1$ and let $\tilde \Z = \{(u,v)\}\cup \Z$. 
$$
p_t = \argmin\enscond{\norm{p}}{(\Phi^* p)(tx)=1, \nabla (\Phi^*p)(tx)=0, \; \forall x\in \Z, \; \Phi^*p(t (u,v))=1}.
$$
Then, $p_{W,z} =\lim_{t\to 0} p_{t}$.
\end{prop}
\begin{proof}
First note that the least interpolant space associated with Hermite interpolation at $\Z$ is spanned by the following basis:
\begin{equation}\label{basis_4spikes}
\Bb_\Z =
\enscond{X^\alpha}{\abs{\alpha}\leq 3} \cup \enscond{X^\beta}{\beta\in \{(1,3), (3,1) \}}.
\end{equation}
Let $\tilde \Z = \{(1,1), (-1,1), (1,-1)\}$. Then, we have that $\nabla^j\eta_{W,\tilde \Z}(0) = 0$ for  $j=1,2,3$ and $\partial_x^3\partial_y\eta_{W,\tilde \Z}(0) = \partial_y^3\partial_x\eta_{W,\tilde \Z}(0) = 0$.
 Therefore, $\eta_{W,\tilde \Z} = \eta_{W,\Z}$.

 Observe now that the de Boor basis associated with Hermite interpolation on $\Z$ and Lagrange interpolation  on $(u,v)$ is 
$$
\Bb_{\Z} \cup \left\{p(x,y) =  y^4 + 6 x^2 y^2 \left(\frac{u^2-1}{v^2-1}\right) + x^4 \left(\frac{u^2-1}{v^2-1}\right)^2\right\}
$$
 Moreover, by the explicit formula given in Proposition~\ref{prop:gaussian}, we have that $\partial^4_y\eta_{W,\Z}(0) = \partial^4_x\eta_{W,\Z}(0)=-48$ and $\partial^2_y\partial_x^2 \eta_{W,\Z}(0) = -16$. Therefore,
$
p(\partial_x,\partial_y)\eta_{W,Z}(0) = 0
$
whenever
$$
-48 - 96 \left( \frac{u^2-1}{v^2-1}\right) - 48  \left( \frac{u^2-1}{v^2-1}\right)^2 = 0.
$$
i.e. $u^2 + v^2 = 2$. So, provided that $u^2 + v^2 = 2$, then $p_{W,z} = \lim_{t\to 0}p_t$.

\end{proof}


\section{Numerical Study}
\label{sec-numerics}

\subsection{Considered Setups}

We consider three different imaging operators $\Phi$, intended to be representative of three different setups routinely encountered in imaging or machine learning. For each setup, in order to perform the computations of $\eta_{V,\Z}$, $\eta_{W,\Z}$ and to implement the Frank-Wolfe algorithm detailed in Section~\ref{sec-fw}, the only requirement is to be able to evaluate the correlation kernel $\Corr$ defined in~\eqref{eq-etaV-corr} and its derivatives. 

In these examples, we consider the clustering of the spikes positions at a fixed point $z_0 \in \Xx$, i.e. consider for $t>0$ the positions $\Z_t = (z_0+t(z_i-z_0))_{i=1}^N \in \Xx^N$.
For the purpose of simplifying notation, the previous sections detailed only  the case  of $z_0 = 0$, i.e. $\Z_t = t\Z$, however, all previous results also hold in this more general setting by a change of variable $x \in \Xx \rightarrow x-z_0 \in \Xx$. Note that if $\Xx$ is not translation invariant, one should restrict the translation around $z_0$ and extend it into a smooth diffeomorphism on $\Xx$, see Appendix~\ref{sec-reparam} for a proof of the reparametrization invariance of $\eta_{V,Z}$.  

\begin{itemize}
	\item \textit{Gaussian convolution:} this corresponds to a translation invariant setup, which is typical in the modelling of acquisition blur in image processing. We consider  $\phi(x)=e^{-\frac{\norm{x-\cdot}^2}{2\si^2}} \in \Hh=L^2(\RR^2)$ on $\Xx=\RR^2$, and one has
		\eql{\label{eq-corr-gauss}
			\Corr(x,x') = e^{-\frac{\norm{x-x'}^2}{4\si^2}}.
		}
		In this case, the clustering point is set to be $z_0=0$. 
	
	\item \textit{Gaussian mixture estimation:} 
		In machine learning, an important problem is to estimate the parameters $(z_i)_{i=1}^N \in \Xx^N$ of a mixture $\sum_{i=1}^N a_i \phi(z_i)$ of $N$ elementary distributions parameterized by $\phi$ from samples or moments observations, see~\cite{gribonval2017compressive} for an overview of this problem. This problem can be recast as a super-resolution problem, where one seeks to recover the measure $m_0 = \sum_i a_i \de_{z_i}$ from observations of the form~\eqref{eq-fwd-model} where the noise $w$ accounts for the sampling scheme (in a real-life machine learning setup, the operator $\Phi$ itself is noisy to account for the sampling scheme). 
	We consider here a classical instance of this setup, where one looks for a mixture of 1-D Gaussians, parameterized by mean $m \in \RR$ and standard deviation $s \in \RR_+^*$, i.e. $x=(m,s) \in \Xx = \RR \times \RR_+^{*}$, so that $\phi(x) = \frac{1}{s} e^{ -\frac{(\cdot-m)^2}{2s^2} } \in \Hh=L^2(\RR)$ and the correlation operator reads
	\eql{\label{eq-gmixture}
		\Corr((m,s),(m',s')) =  \frac{1}{\sqrt{s^2+s'^2}} e^{ -\frac{(m-m')^2}{2(s^2+s'^2)} }.
	}
	In this case, the clustering point is set to be $z_0=(m_0,s_0) = (0,2)$. 
			
	\item \textit{Neuro-imaging:} for medical and neuroscience imaging applications, a standard goal is to estimate pointwise sources inside some domain $\Xx \subset \RR^d$ (where $d=2$ or $3$) from measurements on the boundary $\partial \Xx$. The operator is thus of the form $\phi(x) = ( \psi(x,u) )_{u \in \partial \Xx} \in \Hh = L^2(\partial \Xx)$ (equipped with the uniform measure on the boundary) where the kernel $\psi(x,u)$ corresponds to the impulse response of the measurement operator.
	To model MEG or EEG acquisition~\cite{gramfort2013time}, we consider a singular kernel $\psi(x,u)=\norm{x-u}^{-2}$ which accounts for the decay of the electric or magnetic field in a stationary regime. We consider a disk domain $\Xx=\enscond{x \in \RR^2}{\norm{x} < 1}$ which could model a slice of a head. The correlation function associated to this problem is
	\eql{\label{eq-cor-neuro}
		\Corr(x,x') = 2\pi \frac{ 1-\norm{x}^2 \norm{x'}^2 
			}{
				(1-\norm{x}^2)(1-\norm{x'}^2)( (1 - \dotp{x}{x'})^2 + |x \wedge x'|^2 )
			}, 
	}
	see Appendix~\ref{sec-proof-cor-neuro} for a proof.
	In this case, the clustering point is set to be $z_0=(0.4, 0.3) \in \Xx$. 
\end{itemize}

As it is customary for sparse regularization, we perform the BLASSO recovery using an $L^2$ normalized operator, i.e. perform the replacement
\eq{
	\phi(x) \leftarrow \frac{\phi(x)}{\norm{\phi(x)}_\Hh}
	\quad\Longrightarrow\quad
	\Corr(x,x') \leftarrow \frac{\Corr(x,x')}{ \sqrt{\Corr(x,x)\Corr(x',x')} }.
}
Note that for translation invariant operators (i.e. convolutions), the kernels are already normalized.

\newcommand{\MyFigEtaW}[2]{\includegraphics[width=.23\linewidth]{etaw/#1/#1-etaw-N#2}}
\newcommand{\MyFigGMixt}[1]{\includegraphics[width=.23\linewidth]{etaw/gmixture2/gmixture2-N2-nor1-#1}}

\begin{figure}
\centering
\begin{tabular}{@{}c@{\hspace{1mm}}c@{\hspace{1mm}}c@{\hspace{1mm}}c@{}}
\MyFigEtaW{gaussian2d}{2}&
\MyFigEtaW{gaussian2d}{3}&
\MyFigEtaW{gaussian2d}{4}&
\MyFigEtaW{gaussian2d}{5}\\
$N=2$ & $N=3$ & $N=4$ & $N=5$ \\
\MyFigGMixt{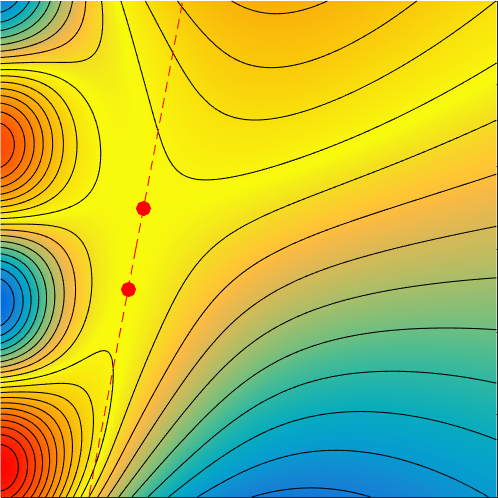} &
\MyFigGMixt{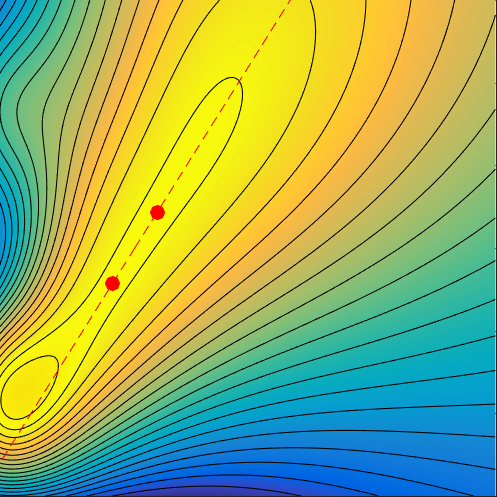} &
\MyFigGMixt{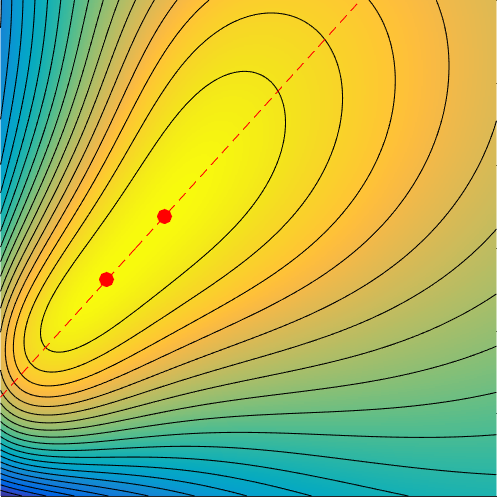} &
\MyFigGMixt{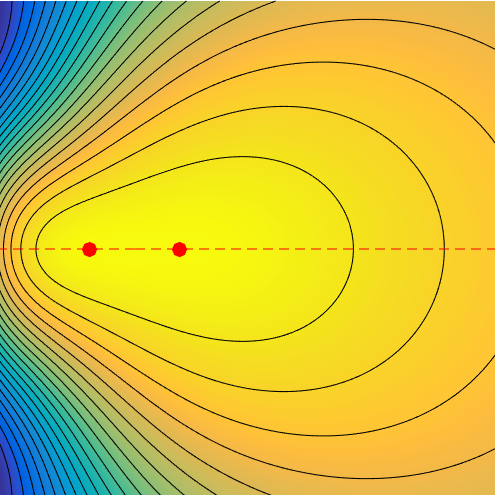} \\
$N=2$ & $N=2$ & $N=2$ & $N=2$\\
\MyFigEtaW{neuro-like-disc}{1}&
\MyFigEtaW{neuro-like-disc}{2}&
\MyFigEtaW{neuro-like-disc}{2bis}&
\MyFigEtaW{neuro-like-disc}{3}\\
$N=1$ & $N=2$ & $N=2$ & $N=3$ 
\end{tabular}
\caption{\label{fig-etaw}
Display of the evolution of $\eta_{W,\Z}$ for the three different operators $\Phi$.  The dashed red line shows the directions $(\Z_t)_{t>0}$ along which the spikes are converging. Red color indicates regions where $\eta_{W,\Z}(x) >1$, i.e. it is degenerated. 
Top: Gaussian convolution~\eqref{eq-corr-gauss}. 
Middle: Gaussian mixture~\eqref{eq-gmixture}, here the horizontal axis is the standard deviation $s \in [0.5,6]$ and the vertical axis is the mean $m \in [-3,3]$.
Bottom: neuro-imaging like~\eqref{eq-cor-neuro}.
}
\end{figure}

\subsection{Asymptotic Certificate $\eta_{W,\Z}$}
\label{sec-numerics-etaW}

Figure~\ref{fig-etaw} explores the behaviour of $\eta_{W,\Z}$ in the three considered cases:
\begin{itemize}
	\item \textit{Gaussian convolution~\eqref{eq-corr-gauss}:} we found numerically that $\eta_{W,\Z}$ is always non-degenerate, for any $N$ and spikes configuration $\Z$. This is inline with the theoretical results of Section~\ref{sec-gaussian-closedform}. This implies that one can hope (and provably do so for $N=2$ according to Theorem~\ref{thm-twospikes}) to achieve super-resolution  for Gaussian deconvolution (provided, of course, that the signal-to-noise ratio is large enough). 

	\item \textit{Neuro-imaging~\eqref{eq-cor-neuro}:} we observed numerically that $\eta_{W,\Z}$ is always non-degenerate for $N=2$ and more generally for aligned spikes. In contrast, for three non-aligned spikes, $\eta_{W,\Z}$ is not a valid certificate ($\norm{\eta_{W,\Z}}_\infty >1$) which means that in the presence of noise, one cannot stably super-resolve 3 close spikes.
	
	\item \textit{Gaussian mixture estimation~\eqref{eq-gmixture}:} here, the situation is more complicated, and for $N=2$ spikes, $\eta_{W,\Z}$ is non degenerate if $|m_2-m_1| \leq |s_2-s_1|$. This means that one can super-resolve with BLASSO a mixture of two Gaussians provided that the variation in the means is not too large with respect to the variation in standard deviations.
	Note also that in the special 1-D case where either the means or the standard deviation are equal and known (which leads to a 1-D super resolution problem along the $m$ or $s$ axis) then the resulting 1-D $\eta_W$ is non-degenerate.  It is the interplay between means and standard deviation that makes the super-resolution possibly problematic.
\end{itemize}

An important aspect to consider, which explains partly the above observations, is that, as explained in Section~\ref{sec-convolution-vanish}, convolution operators tend to have much better behaved $\eta_{W,\Z}$ than arbitrary operators (such as the  neuro-imaging and the Gaussian mixture), because their odd derivatives always vanish. In contrast, the vanishing of odd derivatives for a generic operator only occur for particular values of $N$ and spikes configuration (e.g. aligned spikes). Without having its odd derivatives vanishing, $\eta_{W,\Z}$ cannot be expected to be smaller than $1$ near the spikes position $(z_i)_i$.

\newcommand{\MyFigFW}[1]{\includegraphics[width=.21\linewidth]{fw/#1}}
\begin{figure}
\centering
\begin{tabular}{c}
\begin{tabular}{@{}c@{\hspace{1mm}}c@{}}
\MyFigFW{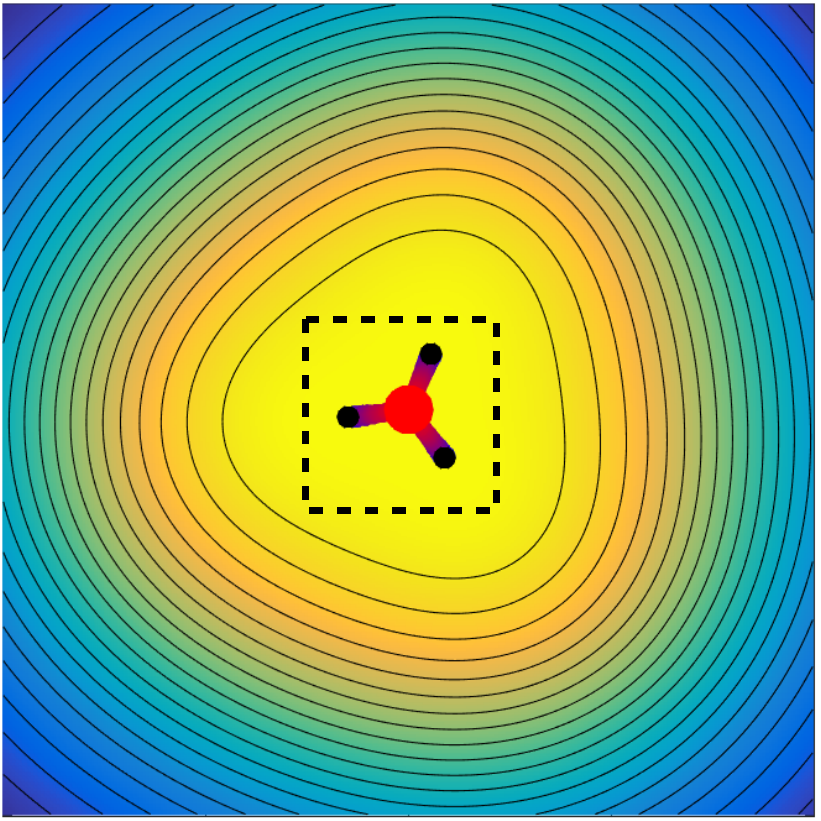} &
\MyFigFW{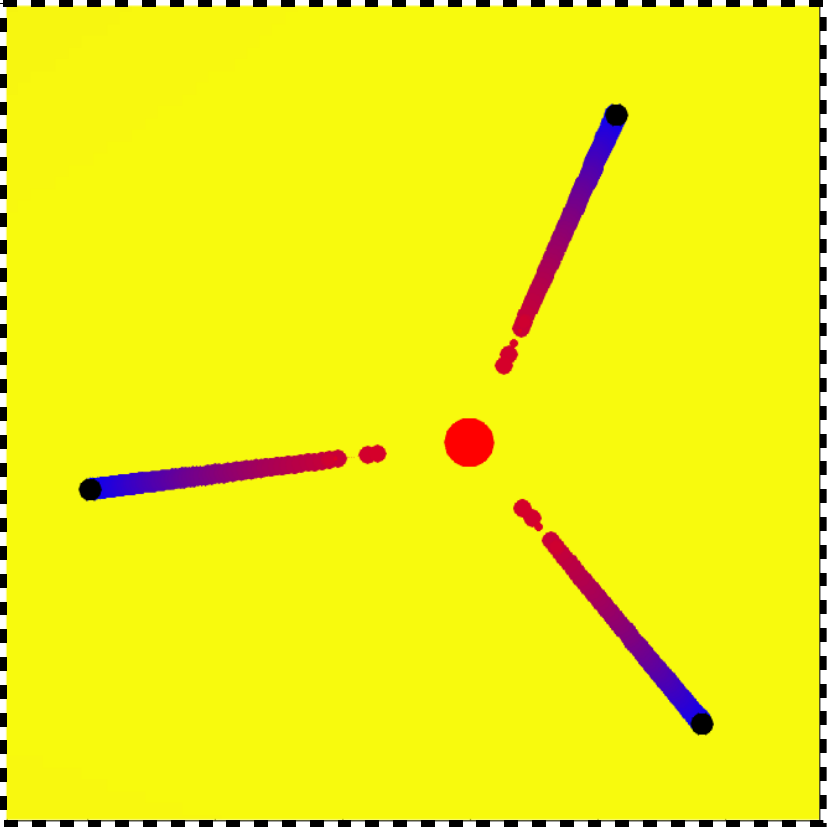} \\
$\eta_{V,\Z}$ and $m_\la$ & Zoom
\end{tabular}\\
Gaussian $N=3$
\end{tabular}
\begin{tabular}{c}
\begin{tabular}{@{}c@{\hspace{1mm}}c@{}}
\MyFigFW{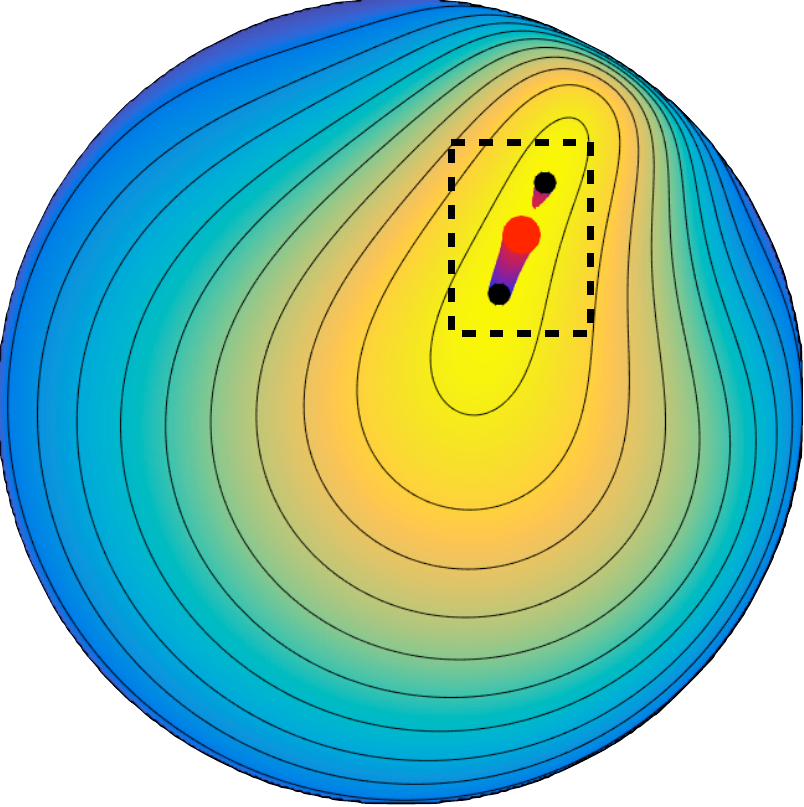} &
\MyFigFW{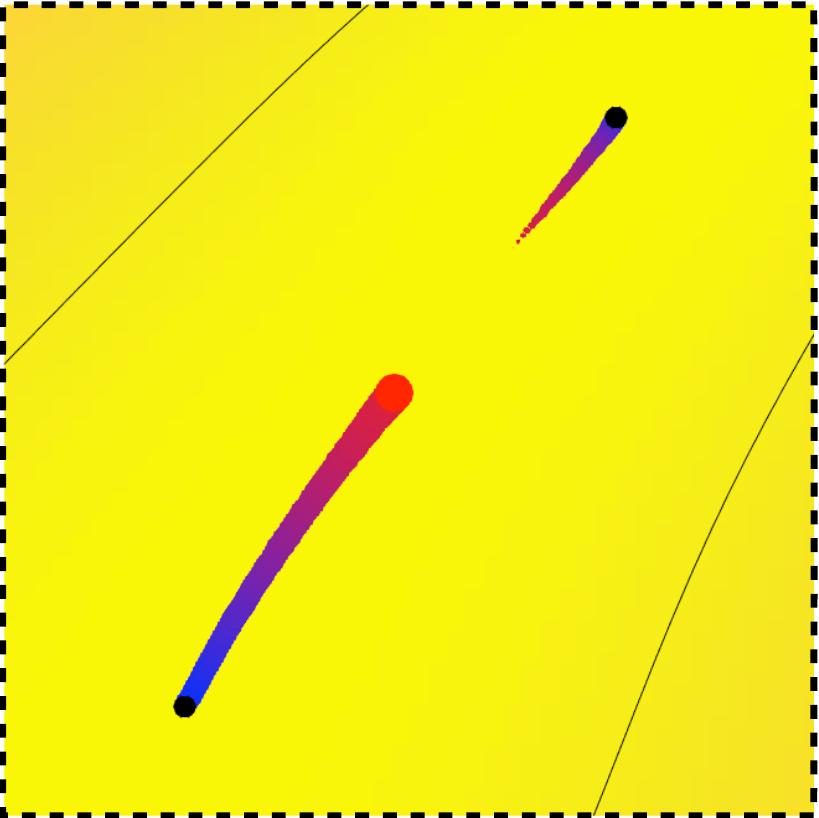} \\
$\eta_{V,\Z}$ and $m_\la$ & Zoom
\end{tabular}\\
Neuro-imaging $N=2$
\end{tabular}
\caption{\label{fig-frank-wolfe}
Display of the evolution of the solution $m_\la$ of~\eqref{eq-blasso} (computed using Frank-Wolfe algorithm) as a function of $\la$ for two different operators $\Phi$, in cases where $\eta_{V,\Z}$ is non-degenerated.  
The settings are the same as for Figure~\ref{fig-etaw}, and the bottom row is a zoom in the dashed rectangular region indicated on the top row. 
A spike $a \de_{x}$ of $m_\la$ is indicated with a disk centered at $x$ of radius proportional to $a$, and the color ranges between blue for $\la=0$ and red for $\la=\la_{\max}$. 
The background color image shows $\eta_{V,\Z}$ where $z$ are the spikes possitions of $m_0$ (plotted in black).
}
\end{figure}

\subsection{Spikes Recovery with Frank-Wolfe}
\label{sec-fw}

In order to solve numerically the BLASSO problem~\eqref{eq-blasso}, we follow~\cite{bredies-inverse2013,boyd2017alternating} and use the Frank-Wolfe algorithm (also known as conditional gradient) with improved non-convex updates. 
The algorithm starts with the initial zero measure $m^{(0)}=0$, and alternates between a ``matching pursuit'' step which generates a new spike location 
\eql{\label{eq-fw-step-1}
	\tilde x \eqdef \uargmax{x \in \Xx} | \eta^{(\ell)}(x) | \qwhereq
	\eta^{(\ell)}(x) \eqdef \frac{1}{\la} \dotp{\phi(x)}{ y-\Phi m^{(\ell)} }_{\Hh}, 
}
with associated amplitude $\tilde a \eqdef \la \eta^{(\ell)}(\tilde x_{\ell+1})$, and
a local non-convex minimization step, initialized with 
$r \leftarrow (x^{(\ell)}_1,\ldots,x^{(\ell)}_\ell,\tilde x) \in \Xx^{\ell+1}$ and 
$b \leftarrow (a^{(\ell)}_1,\ldots,a^{(\ell)}_\ell,\tilde a) \in \RR^{\ell+1}$
\eql{\label{eq-fw-step-2}
	(x^{(\ell+1)},a^{(\ell+1)}) \eqdef \uargmin{(r,b) \in \Xx^{\ell+1} \times \RR^{\ell+1}} \frac{1}{2\la} \norm{ y - \sum_{i=1}^{\ell+1} b_i \phi(r_i) }^2 + \norm{ b }_1. 
}
After each iteration, the measure is updated as 
\eq{
	m^{(\ell+1)} \eqdef \sum_{i=1}^{\ell+1}  a^{(\ell+1)}_i \de_{ x^{(\ell+1)}_i }. 
}
The termination criterion is $|\eta^{(\ell)}(\tilde x)| \leq 1$, which means that $m^{(\ell)}$ is a solution to~\eqref{eq-blasso} because $\eta^{(\ell)}$ is a valid dual certificate of optimality for $m^{(\ell)}$. 
The algorithm is known to converge in the sense of the weak topology of measures to a solution of~\eqref{eq-blasso}, see~\cite{bredies-inverse2013}. Without the non-convex update, convergence is slow (the rate on is only $O(1/\ell)$ on the BLASSO functional being minimized~\cite{jaggi2013revisiting}). However, as we illustrate next,  empirical observations suggest that by applying the non-convex update~\eqref{eq-fw-step-2}, convergence is often reached in a finite number of iteration.

Numerically, the low-dimensional optimization problems~\eqref{eq-fw-step-1} and~\eqref{eq-fw-step-2} are solved using a quasi-Newton (L-BFGS) solver.
Computing the gradient of the involved functionals only require the evaluation of the correlation operator $\Corr$ and its derivative, assuming the measure $m^{(\ell)}$ are stored using a list of (positions, amplitudes).

Figure~\ref{fig-frank-wolfe} explores the behaviour of the solution $m_{\la}$ of~\eqref{eq-blasso} as $(\la,w) \rightarrow 0$, in cases where $\eta_{V,\Z}$ is non-degenerate, so that support is stable in this low-noise regime. 
Inline with support stability theorems, we scale the noise linearly with $\la$, $y=\Phi m_0 + \la w$, and set the noise $w$ to be of the form $w=\Phi \bar m$ where $\bar m$ is a random measure $\sum_{j} b_j \de_{u_j}$ of $Q=20$ random points $(u_j)_{j=1}^Q \in \Xx^Q$ where $(b_j)_j$ is white noise with standard deviation $10^{-3}$.
Numerically, we found that in these cases where $\eta_{V,\Z}$ is non-degenerate, Frank-Wolfe with non-convex update converges in a finite number of steps. 
The color (from blue to red) allows to track the evolution with $\la$ of the solution, which highlight the smoothness of the solution path.

Figure~\ref{fig-frank-wolfe-fixed} shows, in contrast, cases where $\eta_{V,\Z}$ is degenerate. According to Section~\ref{sec-necessity-nondegen}, in this case, the support of the solution $m_\la$ is not stable for small $\la$, and one expects this solution to be composed of more than $N$ diracs. Numerically, in these case, Frank-Wolfe does not converge in a finite number of steps, and it keeps creating new spikes of very small amplitudes. The figure shows how these additional spikes are added to force $|\eta^{(\ell)}|$ to be smaller, while $\eta_{V,\Z}$ is not. 

\newcommand{\MyFigFWF}[1]{\includegraphics[width=.21\linewidth]{fw-fixed/#1}}

\begin{figure}
\centering
\begin{tabular}{c}
\begin{tabular}{@{}c@{\hspace{1mm}}c@{}}
\MyFigFWF{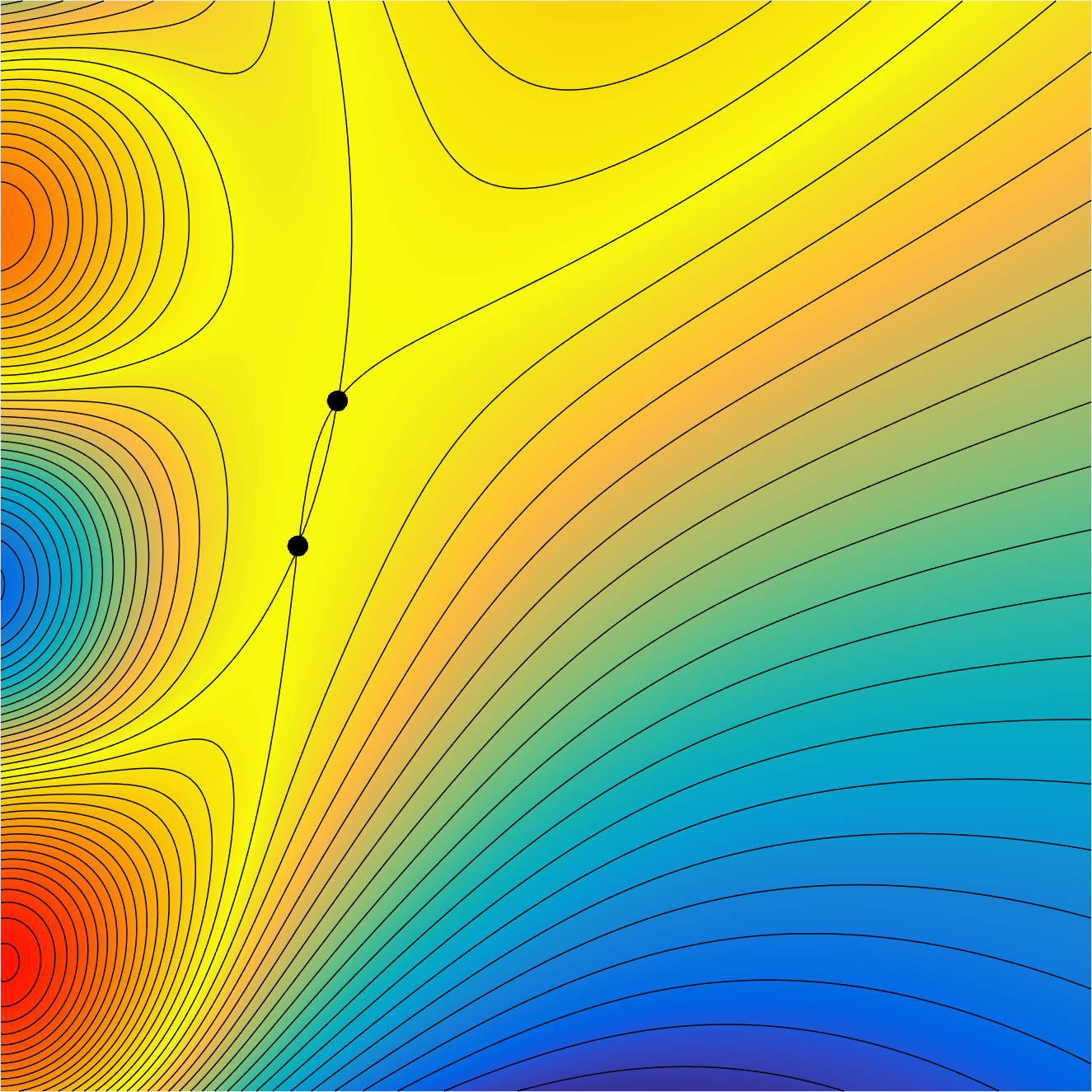} &
\MyFigFWF{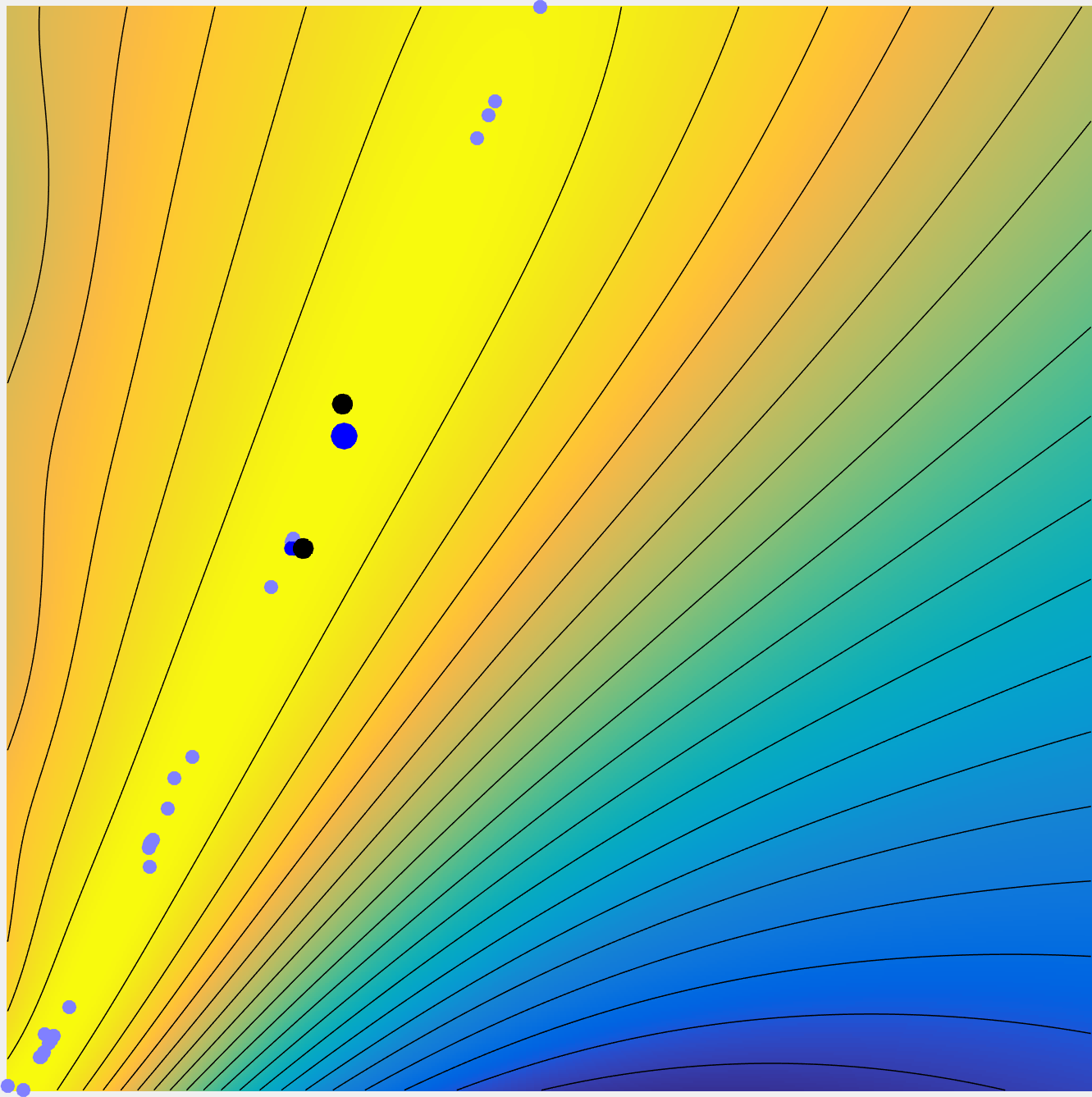} \\
$\eta_{V,\Z}$ and $m_0$ & $\eta^{(\ell)}$ and $m^{(\ell)}$
\end{tabular}\\
Gaussian mixture $N=2$
\end{tabular}
\begin{tabular}{c}
\begin{tabular}{@{}c@{\hspace{1mm}}c@{}}
\MyFigFWF{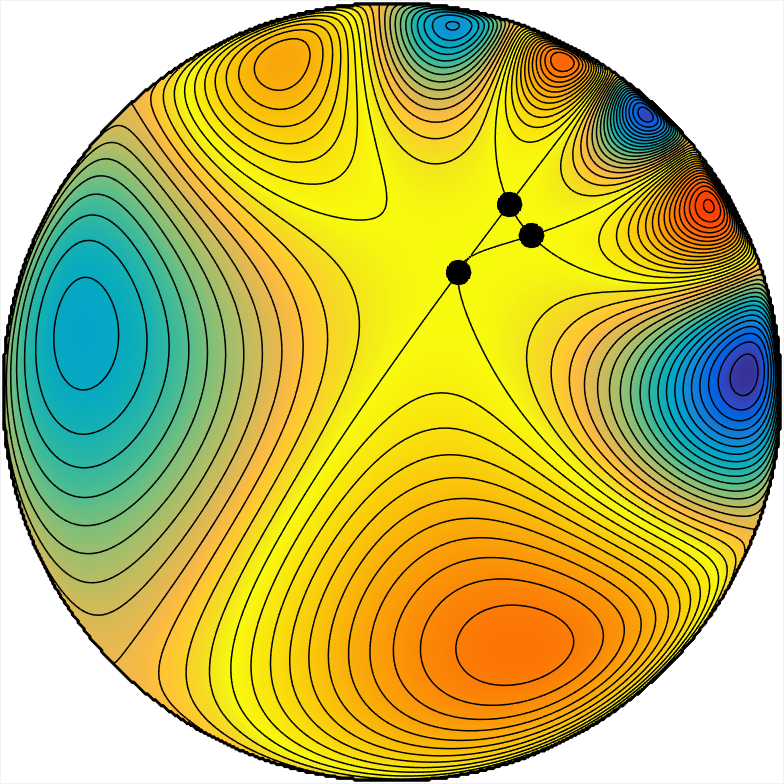} &
\MyFigFWF{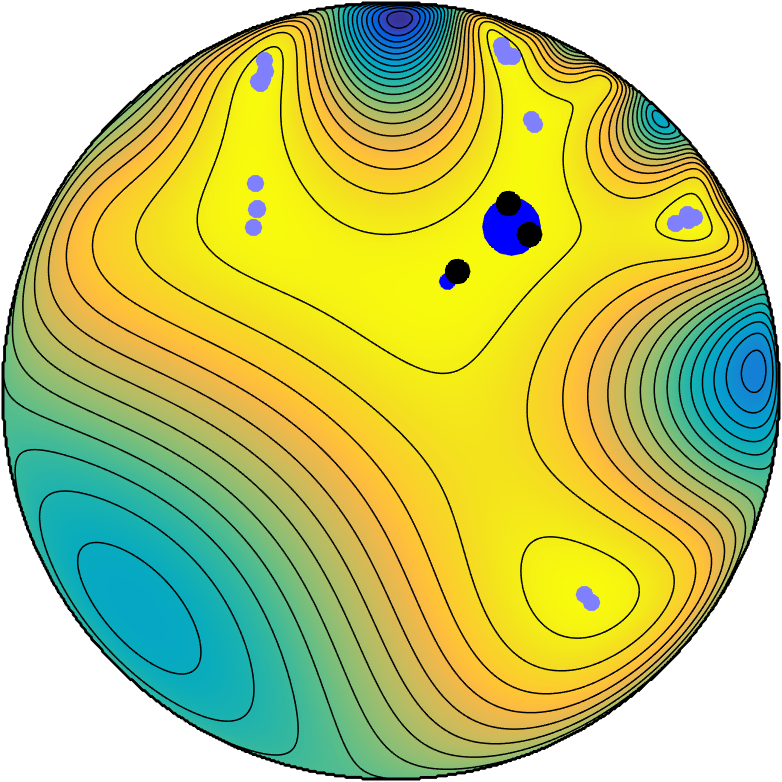} \\
$\eta_{V,\Z}$ and $m_0$ & $\eta^{(\ell)}$ and $m^{(\ell)}$
\end{tabular}\\
Neuro-imaging $N=3$
\end{tabular}
\caption{\label{fig-frank-wolfe-fixed}
Display of the solution $m^{(\ell)}$ computed using $\ell=40$ Frank-Wolfe iterations, in cases where $\eta_{V,\Z}$ is degenerated (as indicated by red regions).  
The settings are the same as for Figure~\ref{fig-etaw}.
The light blue dots indicate the support of $m^{(\ell)}$ (which thus allows one to locate spikes with very small amplitude) while blue dots are displayed with a size propositional to the amplitude of the corresponding spike.
}
\end{figure}


\section*{Aknowlegements}

We would like to thank Vincent Beck stimulating discussions about polynomial interpolation.

\section{Conclusion}

This article presented a study of the multivariate BLASSO problem in the case when recovering positive spikes positioned very close together. In particular, we focussed on the question of support stability. Previous studies \cite{duval2015exact,2017-denoyelle-jafa} have highlighted the importance of the precertificate for this question, and as a first contribution, we presented a procedure for computing the limit $\eta_{W,\Z}$ of the associated precertificates as the point sources converge towards a limit point. Since a necessary condition for support stability is that this certificate is valid (it is uniformly bounded by 1), one can quickly check whether this certificate is valid before proceeding with more detailed analysis. Our second main contribution is a detailed analysis in the case of recovering a superposition of 2 spikes. Here, we showed that under a nondegeneracy condition on $\eta_{W,\Z}$, support stability can be achieved provided that the norm of the additive noise $\norm{w}$ and the regularization parameter $\la$ decays like $t^4$, where $t$ is the spacing between the 2 spikes. The question of which conditions are necessary for support stability when recovering more than 2 spikes remains open. The final part of this paper presented numerical examples related to 3 different imaging situations, it is perhaps interesting to observe that breakdown of support stability in the Gaussian mixture case and the neuro-imaging case, and this is potentially an interesting area for further investigation. 

\appendix

\section{Proof of Proposition~\ref{prop:lin_indep} (Linear Independence)}
\label{proof:prop:lin_indep}

\textbf{Step I.} Let us first show that $\Psi_L\eqdef  \enscond{\partial^\alpha  \varphi(0)}{\abs{\alpha}\leq L, \; \alpha\in\NN_0^2}$ is linearly independent provided that $\hat \psi(\alpha)\neq 0$ for all $\al\in\NN_0^2$ with $\abs{\al}\leq L$.

Observe that $\partial^\alpha \varphi(0) = \partial^\alpha \psi$, and that
 $\Psi_L = \enscond{\partial^\alpha  \psi}{\abs{\alpha}\leq L, \; \alpha\in\NN_0^2}$ is linearly independent if and only if the Fourier coefficients of the elements in $\Psi_L$ are linearly independent. 
The Fourier Transform of $\partial^\alpha \tilde \varphi$ evaluated at frequencies $\xi\eqdef (\xi_j)_{j=1}^M$ are
$\left((\hat{ \psi}(\xi_j) (2\pi i \xi_j)^\alpha\right)_{j=1}^M$. Therefore, $\Psi_L$ is linearly independent if the columns of the matrix $\diag((\hat \psi(\xi_j)_{1\leq j\leq N})  M_\xi$ are linearly independent, where $M_\xi \eqdef ((2\pi i \xi_j)^\alpha)_{\substack{\abs{\al}\leq L\\ 1\leq j\leq M}}$ is the Lagrange interpolation matrix, with evaluation at points $\xi$ and using the polynomial basis $(X^\alpha)_{\abs{\al}\leq L}$. From \cite[Theorem 1]{lorentz2000multivariate}, we know that $M_\xi$ is invertible for almost every choice of $\xi$ where $M=\mathrm{\Pi^2_L}$. Furthermore, one possible choice of $\xi$ is
$$
\enscond{\alpha\in\NN_0^2}{\abs{\alpha}\leq L}. 
$$
Therefore, to ensure linear independence of $\Psi_L$, it is enough to check that $\hat \psi(\alpha) \neq 0$ for all $\alpha\in\NN_0^2$ such that $\abs{\al}\leq L$.

\textbf{Step II.} 
We are now ready to show that $\enscond{h(\partial)\varphi(0)}{h\in\Ss_z}$ is of dimension $3N$.

Recall from Remark \ref{rem:V_z_restr} that $\Ss_z$ associated with Hermite interpolation at $N$ points $\Z$ satisfies $\Ss_z\subset \Pi_{L}^2$, where $L\eqdef 2N-1$.
 Let $B$ be the coefficient matrix such that $B(X^\alpha)_{\abs{\al}\leq L} = (g_i(X))_{1\leq i\leq 3N}$. Note since $\Bb$ is a basis, given any $a\in\RR^{3N}$, $B^T a=0$ if and only if $a=0$, since otherwise, there would be an $a\neq 0$ such that
$$
0 = \dotp{a}{B(X^\alpha)_{\abs{\al}\leq L}}
$$
which would contradict the assumption that $\Bb$ is a basis. Therefore, if $\enscond{g_i(\partial)\varphi(0)}{i=1,\ldots, 3N}$ is linearly dependent, then there exists $0\neq a\in\RR^{3N}$ such that
$$
0=\dotp{a}{B(\partial^\alpha \varphi(0))_{\abs{\al}\leq L}} = \dotp{B^T a }{(\partial^\alpha \varphi(0))_{\abs{\al}\leq L}}.
$$
This leads to the required contradiction, we have shown in the first step that $\Psi_L$ is linearly independent, and therefore, $B^T a=0$ and hence $a=0$.

\section{Proof of Theorem~\ref{thm-etaw-lowwpass}}
\label{sec-proof-thm-etaw-lowwpass}




\subsection{Part 1: $\eta_V(x)<1$ for all $x\not\in \Z$}

Recall that $\eta_V$ is of the form
\begin{align*}
\eta_V(x) &= \sum_{j=1}^N \alpha_j \dotp{\varphi(z_j)}{\varphi(x)} + \sum_{j=1}^N \beta_j \dotp{\varphi'(z_j)}{\varphi(x)} \\
&=\sum_{j=1}^N \alpha_j \dotp{v(z_j)}{v(x)} + \sum_{j=1}^N \beta_j \dotp{v'(z_j)}{v(x)},
\end{align*}
where  $v:\TT\to \CC^{2f_c+1}$ is defined by $v(x) \eqdef (e^{2\pi i k x})_{\abs{k}\leq f_c}$. 
If there exists $\tau \notin \Z$ such that $\eta_V(\tau) = 1$, $L^* R$ is a singular matrix, where
$$
R\eqdef \begin{pmatrix}
e^{2\pi i f z_1}-e^{2\pi i f \tau} & \cdots & e^{2\pi i f z_N}-e^{2\pi i f \tau} &  (2\pi i f) e^{2\pi i f z_1}& \cdots & (2\pi i f) e^{2\pi i f z_N}\\
\vdots\\
e^{2\pi i z_1}-e^{2\pi i \tau} & \cdots & e^{2\pi i  z_N}-e^{2\pi i f \tau} &  (2\pi i ) e^{2\pi i z_1}& \cdots & (2\pi i ) e^{2\pi i  z_N}\\
e^{-2\pi i z_1}-e^{-2\pi i \tau} & \cdots & e^{-2\pi i  z_N}-e^{-2\pi i f \tau} &  (-2\pi i ) e^{-2\pi i z_1}& \cdots & (-2\pi i ) e^{-2\pi i  z_N}\\
\vdots\\
e^{-2\pi i f z_1}-e^{-2\pi i f \tau} & \cdots & e^{-2\pi i f z_N}-e^{-2\pi i f \tau} &  (-2\pi i f) e^{2\pi i f z_1}& \cdots & (-2\pi i f) e^{2\pi i f z_N}
\end{pmatrix}
$$ 
and
$$
L\eqdef
\begin{pmatrix}
e^{2\pi i f z_1} & \cdots & e^{2\pi i f z_N} &  (2\pi i f) e^{2\pi i f z_1}\cdots (2\pi i f) e^{2\pi i f z_N}\\
\vdots\\
e^{2\pi i z_1}  & \cdots & e^{2\pi i  z_N} &  (2\pi i ) e^{2\pi i  z_1}\cdots (2\pi i ) e^{2\pi i  z_N}\\
e^{-2\pi i z_1}  & \cdots & e^{-2\pi i  z_N} &  (-2\pi i ) e^{-2\pi i  z_1}\cdots (-2\pi i ) e^{-2\pi i  z_N}\\
\vdots\\
e^{-2\pi i f z_1}  & \cdots & e^{-2\pi i f z_N} &  (-2\pi i f) e^{2\pi i f z_1}\cdots (-2\pi i f) e^{2\pi i f z_N}
\end{pmatrix}.
$$ 
To show that this is impossible, first observe that the matrix $L$ has the same determinant as the following $(2N+1)\times (2N+1)$ matrix:
$$
\begin{pmatrix}
0& e^{2\pi i f z_1} & \cdots & e^{2\pi i f z_N} &  (2\pi i f) e^{2\pi i f z_1} &\cdots& (2\pi i f) e^{2\pi i f z_N}\\
\vdots\\
0& e^{2\pi i z_1}  & \cdots & e^{2\pi i  z_N} &  (2\pi i ) e^{2\pi i  z_1} &\cdots& (2\pi i ) e^{2\pi i  z_N}\\
1 & 1 & \cdots & 1 &0 &\cdots &0\\
0& e^{-2\pi i z_1}  & \cdots & e^{-2\pi i  z_N} &  (-2\pi i ) e^{-2\pi i  z_1} &\cdots& (-2\pi i ) e^{-2\pi i  z_N}\\
\vdots\\
0& e^{-2\pi i f z_1}  & \cdots & e^{-2\pi i f z_N} &  (-2\pi i f) e^{2\pi i f z_1} &\cdots & (-2\pi i f) e^{2\pi i f z_N}
\end{pmatrix}.
$$ 
So, if $\det(L) = 0$, then there exists $\alpha\neq 0$ such that
$$
F(x) = 1+ \sum_{j=1}^f \alpha_j x^j + \sum_{j=-f}^{-1}\alpha_j x^j
$$
has roots at $e^{2\pi i z_l}$ for $l=1,\ldots, N$ and at $0$. Moreover, $F'(e^{2\pi i z_l}) = 0$ for all $l=1,\ldots, N$. However, this would imply that $x^f F(x)$ has $2N+1$ roots, which is a contradiction to the fact that this is a polynomial of degree $2f = 2N$.
Therefore, $\det(L) \neq 0$.

So, to prove this theorem, it suffices to show that $R$ is nonsingular. 
The determinant of $R$ is equal to that of the following $(2N+1)\times (2N+1)$ matrix:
$$
\begin{pmatrix}
e^{2\pi i f \tau} & e^{2\pi i f z_1}  & \cdots & e^{2\pi i f z_N}  &  (2\pi i f) e^{2\pi i f z_1}&\cdots & (2\pi i f) e^{2\pi i f z_N}\\
\vdots\\
e^{2\pi i  \tau} & e^{2\pi i z_1}  & \cdots & e^{2\pi i  z_N}  &  (2\pi i ) e^{2\pi i z_1} & \cdots & (2\pi i ) e^{2\pi i  x_N}\\
1 & 1& \cdots & 1 & 0 &\cdots &0\\
e^{-2\pi i \tau} & e^{-2\pi i z_1} & \cdots & e^{-2\pi i  z_N}  &  (-2\pi i ) e^{-2\pi i z_1}& \cdots & (-2\pi i ) e^{-2\pi i  z_N}\\
\vdots\\
e^{-2\pi i f \tau} & e^{-2\pi i f z_1}  & \cdots & e^{-2\pi i f z_N}  &  (-2\pi i f) e^{-2\pi i f z_1}& \cdots & (-2\pi i f) e^{-2\pi i f z_N}
\end{pmatrix}.
$$ 
We must have $\det(R)\neq 0$ because otherwise, by the same argument as before, we would construct a polynomial of degree $2N$ with at least $2N+1$ roots (at least double roots at $e^{2\pi i z_l}$ for $l=1,\ldots, N$ and a single root at $e^{2\pi i \tau}$).

So, if $f_c=N$, then $\eta_V(x)\neq 1$ for all $x\notin \Z$. Therefore, since $\eta_V(z_j) = 1$, either $\eta_V(x)\geq 1$ for all $x$ or $\eta_V(x)\leq 1$ for all $x$. Note that both $\eta_V$ and $2-\eta_V$ satisfy the vanishing derivatives constraints. Suppose that $\eta_V(x)\geq 1$ for all $x$. Then $\norm{p_V}_2 = \norm{\eta_V}_2 > \norm{2-\eta_V}_2 = \norm{q_V}_2$ where $\Phi^* q_V  = 2-\eta_V$. This yields a contradiction. Therefore, $\eta_V(x)\leq 1$ for all $x$ and  is the minimal norm certificate.

\subsection{Step 2: closed form expression of $\eta_W$}

First note that there exists $a\in\RR^{2N}$ such that 
$$
\eta_W(x) = \sum_{j=1}^{2N} a_j \dotp{\partial^{j}
\varphi(0)}{\varphi(x)} = \sum_{j=1}^{2N} a_j \dotp{\partial^{j}
v(0)}{v(x)}.$$ Suppose that $\eta_W(\tau) = 1$ for some $\tau \neq 0$. Then, the equations $\eta_W(0)-\eta_W(\tau)=0$, $\partial^j \eta_W(0) = 0$ for $j=1,\ldots, 2N-1$ can be written as the linear system $M_\tau a = 0$, where
$$
M_\tau \eqdef 
\begin{pmatrix}
\dotp{v(0)}{v(0)-v(\tau)} & \dotp{v^{(1)}(0)}{v(0)-v(\tau)} &\ldots & \dotp{v^{(2N-1)}(0)}{v(0)-v(\tau)} \\
\dotp{v(0)}{v^{(1)}(0)} & \dotp{v^{(1)}(0)}{v^{(1)}(0)} &\ldots & \dotp{v^{(2N-1)}(0)}{v^{(1)}(0)} \\
\dotp{v(0)}{v^{(2)}(0)} & \dotp{v^{(1)}(0)}{v^{(2)}(0)} &\ldots & \dotp{v^{(2N-1)}(0)}{v^{(2)}(0)} \\
\vdots&\vdots&\cdots& \vdots\\
\dotp{v(0)}{v^{(2N-1)}(0)} & \dotp{v^{(1)}(0)}{v^{(2N-1)}(0)} &\ldots & \dotp{v^{(2N-1)}(0)}{v^{(2N-1)}(0)} \\
\end{pmatrix}.
$$
We will now proceed to show that $\det(M_\tau)\neq 0$ for all $\tau\neq 0$, and therefore, $\eta_W(\tau)<1$ for all $\tau\neq 0$.
Note that $M_\tau^* = L^* R$ where
\begin{align*}
R &\eqdef \begin{pmatrix}
v(0)-v(\tau), & v^{(1)}(0), & v^{(2)}(0),& \cdots &v^{(2N-1)}(0)
\end{pmatrix},\\
L &\eqdef \begin{pmatrix}
v(0), & v^{(1)}(0), & v^{(2)}(0),& \cdots &v^{(2N-1)}(0)
\end{pmatrix}.
\end{align*}
Let $n=2f_c$.
Since the row corresponding to frequency $k=0$ for the matrix $R$ is zero, we can write $M_\tau^* = \tilde L^* \tilde R$, where
\begin{align*}
\tilde R \eqdef \begin{pmatrix}
1- e^{2\pi i f_c\tau } & 2\pi i f_c & (2\pi i f_c)^2 & \cdots & (2\pi i f_c)^{n-1}\\
1- e^{2\pi i (f_c-1)\tau} & 2\pi i (f_c-1) & (2\pi i (f_c-1))^2 & \cdots & (2\pi i (f_c-1))^{n-1}\\
\vdots& \vdots&&&\vdots\\
1- e^{2\pi i \tau} & 2\pi i  & (2\pi i )^2 & \cdots & (2\pi i )^{n-1}\\
1- e^{-2\pi i \tau} & -2\pi i  & (-2\pi i )^2 & \cdots & (-2\pi i )^{n-1}\\
\vdots& \vdots&&&\vdots\\
1- e^{-2\pi i f_c\tau } & -2\pi i f_c & (-2\pi i f_c)^2 & \cdots & (-2\pi i f_c)^{n-1}\\
\end{pmatrix} \in \CC^{n\times n},
\end{align*}
and
\begin{align*}
\tilde L \eqdef \begin{pmatrix}
1  & 2\pi i f_c & (2\pi i f_c)^2 & \cdots & (2\pi i f_c)^{n-1}\\
1 & 2\pi i (f_c-1) & (2\pi i (f_c-1))^2 & \cdots & (2\pi i (f_c-1))^{n-1}\\
\vdots& \vdots&&&\vdots\\
1 & 2\pi i  & (2\pi i )^2 & \cdots & (2\pi i )^{n-1}\\
1 & -2\pi i  & (-2\pi i )^2 & \cdots & (-2\pi i )^{n-1}\\
\vdots& \vdots&&&\vdots\\
1  & -2\pi i f_c & (-2\pi i f_c)^2 & \cdots & (-2\pi i f_c)^{n-1}\\
\end{pmatrix} \in \CC^{n\times n}.
\end{align*}
Since $\tilde L$ is a Vandermonde matrix generated by $n$ distinct points $\{(2\pi i k)\}_{\abs{k}\leq f_c, k\neq 0}$, $$
\det(\tilde L) =  \prod_{\substack{l>j\\ l,j\in \{-f_c,\ldots,f_c\}\setminus\{0\}}} (2\pi i(l-j)) = (2\pi i)^{2f_c^2-f_c} \prod_{\substack{l>j\\ l,j\in \{-f_c,\ldots,f_c\}\setminus\{0\}}} (l-j) \neq 0.
$$
 So, it remains to show that $\det(\tilde R)\neq 0$ for all $\tau\neq 0$.
\begin{align*}
\det(\tilde R) &= \sum_{k=1}^{f_c} (-1)^{k+f_c} (1-e^{2\pi i k \tau}) \prod_{\substack{l>j\\ l,j\in \{-f_c,\ldots,f_c\}\setminus\{0,k\}}} (2\pi i(l-j)) \prod_{l=-f_c, l\neq 0,k}^{f_c} (2\pi i l)\\
&+\sum_{k=-f_c}^{-1} (-1)^{k+f_c+1} (1-e^{2\pi i k \tau}) \prod_{\substack{l>j\\ l,j\in \{-f_c,\ldots,f_c\}\setminus\{0,k\}}} (2\pi i(l-j)) \prod_{l=-f_c, l\neq 0,k}^{f_c} (2\pi i l)\\
&= -\det(\tilde L)\left( f_c!\right)^2
\left( 
\sum_{k=1}^{f_c} \frac{ (-1)^{k} (2-e^{2\pi i k \tau} - e^{-2\pi i k \tau})}{(f+k)! (f-k)!}
\right)
\end{align*}
Let $x = e^{-2\pi i \tau}$, then
\begin{align*}
 F(x) \eqdef &(2f_c)!\sum_{k=1}^{f_c} \frac{ (-1)^{k} (2-e^{2\pi i k \tau} - e^{-2\pi i k \tau})}{(f+k)! (f-k)!}
=  \sum_{k=1}^{f_c} \binom{2f_c}{f_c-k} (-1)^{k} (2- x^{k} - x^{-k}).
\end{align*}
Observe that
\begin{align*}
&   \sum_{k=1}^{f_c} (-1)^k \binom{2f_c}{f_c-k} (x^{+k}+ x^{-k}) = (-1)^{f_c} \frac{(1-x)^{2f_c}}{x^{f_c}} - \binom{2f_c}{f_c},
\end{align*}
and
$
2\sum_{k=1}^{f_c} (-1)^k \binom{2f_c}{f_c-k}  = -\binom{2f_c}{f_c}.
$
Therefore, $$F(x) = \frac{(-1)^{f_c+1} (1-x)^{2f_c}}{  x^{f_c}} = -2^{2f_c}\sin^{2f_c}(\pi \tau).$$  So, $\det(\tilde R)=0$ if and only if $\tau = 0$. In particular, $$\det(\tilde M) = \det(\tilde R)\det(\tilde L^*) = \frac{2^{2f_c}\abs{\det(\tilde L)}^2 }{\binom{2f_c}{f_c}} \sin^{2f_c}(\pi\tau)>0$$ for all $\tau\neq 0$.

Finally, for the explicit formula of $\eta_W$, note that  $$\det(P_\tau) = \det(P_0) - \det(M_\tau) \in \Span\enscond{\dotp{\partial^j\varphi(0)}{\varphi(x)}}{ j=0,\ldots, 2N-1},$$
where $P_\tau$ is the cross-Grammian matrix between the vectors $\{v^{(j)}(0)\}_{j=0}^{2N-1}$ and $\{v(\tau)\}\cup \{v^{(j)}(0)\}_{j=1}^{2N-1}$.  Moreover, $\det(P_0)>0$ since $P_0$ is positive definite. Therefore, since
the function
\begin{equation}\label{eq:g}
g(\tau)\eqdef  1 - C \sin^{2f_c}(\pi\tau), \quad C\eqdef \frac{2^{4f_c^2} \pi ^{(4f_c^2-2f_c)}}{\binom{2f_c}{f_c} \det(P_0)}\cdot  \prod_{\substack{l>j\\ l,j\in \{-f_c,\ldots,f_c\}\setminus\{0\}}} (l-j)^2,
\end{equation}
satisfies $g(0) = 1$ and $\partial^j g(0) = 0$ for $j=1,\ldots, 2N-1$, we have that $\eta_W = g$.

\section{Reparameterization Invariance}
\label{sec-reparam}

In the following, given a Borel map $T: \Xx \to \Yy$, and a Borel measure $\mu$ defined on $\Xx$, $T_\sharp \mu$ is the pushforward measure of $\mu$, so that for all integrable $f\in L^1(\Yy)$,
$$
\int_{\Yy} f(x) \mathrm{d}(T_\sharp \mu)(x) = \int_{\Xx} f(T(x)) \mathrm{d}\mu(x).
$$
\begin{prop}\label{prop:reparam}
Let $T: \Xx \to \Yy$ be a bijection such that the  Jacobian of $T^{-1}$ is invertible. Consider the following minimization problems:
\begin{equation}\label{eq:orig}
\min_m \frac{1}{2} \norm{y-\Phi m}_{\Hh}^2 + \lambda\abs{m}(\Xx).
\end{equation}
\begin{equation}\label{eq:reparm}
\min_m \frac{1}{2} \norm{y- (\Phi\circ T_\sharp^{-1}) m}_{\Hh}^2 + \lambda\abs{m}(\Yy).
\end{equation}
If $\mu$ solve \eqref{eq:orig}, then $\nu \eqdef T_\sharp \mu$ solves \eqref{eq:reparm}.
Let $\eta_D^{\Phi,y}$ and $\eta_D^{\Phi\circ T_\sharp^{-1},y}$ be  the dual certificates and let  $\eta_V^{\Phi,y}$ and $\eta_V^{\Phi\circ T_\sharp^{-1},y}$ be the precertificates associated to \eqref{eq:orig} and \eqref{eq:reparm} respectively. Then, $\eta_D^{\Phi\circ T_\sharp^{-1},y} = \eta_D^{\Phi,y}\circ T^{-1}$ and  $\eta_V^{\Phi\circ T_\sharp^{-1},y} = \eta_V^{\Phi,y}\circ T^{-1}$.

\end{prop}
\begin{proof}
The dual certificate of \eqref{eq:reparm} is 
$$\eta_D^{\Phi\circ T_\sharp^{-1},y} = \frac{1}{\lambda}(\Phi\circ T_\sharp^{-1})^*(\Phi\circ T_\sharp^{-1} \nu -y)
=\frac{1}{\lambda}(\Phi\circ T_\sharp^{-1})^*(\Phi \mu -y) = \eta_D^{\Phi,y}\circ T^{-1}.$$

For the precertificates, suppose that $y = \Phi m_{a,\Z} = (\Phi\circ T_\sharp^{-1})m_{a,T\Z}$. Then,
\begin{align*}
p_V^{\Phi\circ T_\sharp^{-1},y} &= \argmin\enscond{\norm{p}}{ [(\Phi\circ T_\sharp^{-1})^* p](Tz_j) = 1, \; [\nabla((\Phi\circ T_\sharp^{-1})^* p)](T z_j) = 0}\\
&= \argmin\enscond{\norm{p}}{ [\Phi^* p](z_j) = 1, \; [\nabla(\Phi^* p)](z_j) = 0} = p_V^{\Phi, y}
\end{align*}
where we have used the fact that, by letting $J_{T^{-1}}$ denote the Jacobian of $T^{-1}$,
$$
0 = [\nabla((\Phi\circ T_\sharp^{-1})^* p)](Tz_j) = 
\nabla[\Phi^* p(T^{-1}\cdot)](Tz_j)
= [J_{T^{-1}}(Tz_j)] \nabla(\Phi^* p)(z_j)
$$
implies that $\nabla(\Phi^* p)(z_j)$ since the Jacobian of $T^{-1}$ is invertible.
Therefore,
$$
\eta_V^{\Phi\circ T_\sharp^{-1}} = (\Phi\circ T_\sharp^{-1})^* p_V^{\Phi\circ T_\sharp^{-1},y} = (\Phi\circ T_\sharp^{-1})^* p_V^{\Phi, y} = \eta_V^{\Phi,y} \circ T^{-1}.
$$

\end{proof}
\begin{rem}
Let $T$ be as in Proposition \ref{prop:reparam}, let $m_{a,\Z}=\sum_j a_j \delta_{z_j}$ and $y = \Phi m_{a,\Z}$. 
Let $\tilde \Phi \eqdef (\Phi\circ T_\sharp^{-1})$ and $\tilde y \eqdef \tilde \Phi m_{a,T\Z}$. Then, $\tilde y= \Phi m_{a,\Z}$. So, \eqref{eq:reparm} can be rewritten as 
$$
\min_m \frac{1}{2} \norm{\tilde y - \tilde \Phi m}^2 + \lambda\abs{m}(\Yy).
$$
Therefore, to check that the dual certificate of this problem $\eta_{V}^{\tilde \Phi, \tilde y}$  is nondegenerate, it is enough to show that  $\eta_V^{\Phi,y}$ is nondegenerate.
\end{rem}

\begin{cor}\label{cor:transinv}
Let $T$ be as in Proposition \ref{prop:reparam}. Then,
$$\eta_{V}^{ \Phi,  \Phi m_{a,T\Z}} = \eta_V^{ \Phi, \Phi m_{a,\Z}} + \Oo(\norm{\Id-T}).
$$
\end{cor}
\begin{proof}
Given $\varphi : \Xx \to \Hh$, let $\Gamma_{\varphi,\Z}:\RR^{3N}\to \Hh$ be defined as in \eqref{eq:Gamma}, where the subscript $\varphi$ makes explicit the associated kernel.
By Proposition  \ref{prop:reparam}, we have that
$$
\eta_V^{\Phi, \Phi m_{a,T\Z}} = \eta_V^{\Phi \circ T_\sharp , \Phi\circ T_\sharp  m_{a,\Z}} \circ T^{-1} = \Phi^* \Gamma_{\psi,\Z}^{*,\dagger}\binom{1_N}{0_{2N}},
$$
where $\psi = \varphi\circ T$.
On the other hand, $\eta_V^{\Phi,\Phi m_{a,\Z}} = \Phi^* \Gamma^{*,\dagger}_{\phi, z}\binom{1_N}{0_{2N}}$.
Therefore,
$$
\eta_{V}^{\Phi, \Phi m_{a,T\Z}} = \eta_V^{ \Phi, \Phi m_{a,\Z}} + \Oo(\norm{\Id-T}).
$$
\end{proof}

\section{Proof of Correlation Function~\eqref{eq-cor-neuro}}
\label{sec-proof-cor-neuro}

Let $\Xx \subset \RR^2$ denote the open unit disc. 
Then, for $x=(x_1,x_2), x'=(x_1',x_2') \in \Xx$, we let $p=M_1 + \imath M_2$, $p_1 = A_1+\imath A_2$ and $z=e^{it}$. Interpreting $X$ as the unit disc on the complex plane, one has
\begin{align*}
\Corr(x,x') &= \dotp{\phi(x)}{\phi(x')}
 = \int_{\partial X} \frac{\mathrm{d}z}{\imath z \abs{z-p}^2 \abs{z-p_1}^2} \\
& = \int_{\partial X} \frac{-\imath z}{(z-p)(1-z\overline p)(z-p_1)(1-z\overline p_1)}.
\end{align*}
When $p \neq p_1$, there are 2 poles inside $\Xx$: $z = p,p_1$, so by the Cauchy residue theorem,
\begin{align*}
\Corr(x,x') &= 2\pi  \left(  \frac{p}{(1-\abs{ p}^2)(p-p_1)(1-p\overline p_1)} + \frac{p_1}{(p_1-p)(1-p_1\overline p)(1-\abs{ p_1}^2)}\right)
\end{align*}
When $p=p_1$, there is 1 pole inside $X$: $z=p$, so,
\begin{align*}
\Corr(x,x') &= 2\pi \imath \left(\frac{\mathrm{d}}{\mathrm{d}z} \left( \frac{\imath z}{(1-z\overline p)^2}\right)\big\vert_{z=p}\right)
	= 2\pi \left( \frac{1}{(1-\abs{p}^2)^2} + \frac{2\abs{p}^2}{(1-\abs{p}^2)^3} \right)
\end{align*}
One can then check that both expressions simplify to
\eq{
	\Corr(x,x') = 2\pi  \frac{ 1 - |p|^2|p_1|^2 }  { (1-|p|^2)(1-|p_1|^2)|1-p\bar p_1|^2 }.
}

\bibliographystyle{plain}
\bibliography{biblio}

\begin{thebibliography}{10}

\bibitem{andersson2017espirit}
Fredrik Andersson and Marcus Carlsson.
\newblock Espirit for multidimensional general grids.
\newblock {\em arXiv preprint arXiv:1705.07892}, 2017.

\bibitem{azais-spike2014}
Jean-Marc Azais, Yohann De~Castro, and Fabrice Gamboa.
\newblock Spike detection from inaccurate samplings.
\newblock {\em Applied and Computational Harmonic Analysis}, 38(2):177--195,
  2015.

\bibitem{baillet2001electromagnetic}
Sylvain Baillet, John~C Mosher, and Richard~M Leahy.
\newblock Electromagnetic brain mapping.
\newblock {\em IEEE Signal processing magazine}, 18(6):14--30, 2001.

\bibitem{bendory2017robust}
Tamir Bendory.
\newblock Robust recovery of positive stream of pulses.
\newblock {\em IEEE Transactions on Signal Processing}, 65(8):2114--2122, 2017.

\bibitem{betzig2006imaging}
Eric Betzig, George~H Patterson, Rachid Sougrat, O~Wolf Lindwasser, Scott
  Olenych, Juan~S Bonifacino, Michael~W Davidson, Jennifer Lippincott-Schwartz,
  and Harald~F Hess.
\newblock Imaging intracellular fluorescent proteins at nanometer resolution.
\newblock {\em Science}, 313(5793):1642--1645, 2006.

\bibitem{bhaskar2013atomic}
Badri~Narayan Bhaskar, Gongguo Tang, and Benjamin Recht.
\newblock Atomic norm denoising with applications to line spectral estimation.
\newblock {\em IEEE Transactions on Signal Processing}, 61(23):5987--5999,
  2013.

\bibitem{vetterli-sparse2008}
Thierry Blu, Pier-Luigi Dragotti, Martin Vetterli, Pina Marziliano, and Lionel
  Coulot.
\newblock Sparse sampling of signal innovations: Theory, algorithms and
  performance bounds.
\newblock {\em IEEE Signal Processing Magazine}, 25(2):31--40, 2008.

\bibitem{bonnans2013perturbation}
J~Fr{\'e}d{\'e}ric Bonnans and Alexander Shapiro.
\newblock {\em Perturbation analysis of optimization problems}.
\newblock Springer Science \& Business Media, 2013.

\bibitem{boyd2017alternating}
Nicholas Boyd, Geoffrey Schiebinger, and Benjamin Recht.
\newblock The alternating descent conditional gradient method for sparse
  inverse problems.
\newblock {\em SIAM Journal on Optimization}, 27(2):616--639, 2017.

\bibitem{bredies-inverse2013}
Kristian Bredies and Hanna~Katriina Pikkarainen.
\newblock Inverse problems in spaces of measures.
\newblock {\em ESAIM: Control, Optimisation and Calculus of Variations},
  19(1):190--218, 2013.

\bibitem{cadzow1988signal}
James~A Cadzow.
\newblock Signal enhancement-a composite property mapping algorithm.
\newblock {\em IEEE Transactions on Acoustics, Speech, and Signal Processing},
  36(1):49--62, 1988.

\bibitem{candes-superresolution2013}
Emmanuel~J. Cand{\`e}s and Carlos Fernandez-Granda.
\newblock Super-resolution from noisy data.
\newblock {\em Journal of Fourier Analysis and Applications}, 19(6):1229--1254,
  2013.

\bibitem{candes-towards2013}
Emmanuel~J. Cand{\`e}s and Carlos Fernandez-Granda.
\newblock Towards a mathematical theory of super-resolution.
\newblock {\em Communications on Pure and Applied Mathematics}, 67(6):906--956,
  2014.

\bibitem{candes2006robust}
Emmanuel~J Cand{\`e}s, Justin Romberg, and Terence Tao.
\newblock Robust uncertainty principles: Exact signal reconstruction from
  highly incomplete frequency information.
\newblock {\em IEEE Transactions on information theory}, 52(2):489--509, 2006.

\bibitem{chen-atomic1998}
Scott~Shaobing Chen, David~L Donoho, and Michael~A Saunders.
\newblock Atomic decomposition by basis pursuit.
\newblock {\em SIAM review}, 43(1):129--159, 2001.

\bibitem{claerbout-robust1973}
Jon~F Claerbout and Francis Muir.
\newblock Robust modeling with erratic data.
\newblock {\em Geophysics}, 38(5):826--844, 1973.

\bibitem{clark1994two}
Michael~P Clark and Louis~L Scharf.
\newblock Two-dimensional modal analysis based on maximum likelihood.
\newblock {\em IEEE Transactions on Signal Processing}, 42(6):1443--1452, 1994.

\bibitem{clergeot1989performance}
H~Clergeot, Sara Tressens, and A~Ouamri.
\newblock Performance of high resolution frequencies estimation methods
  compared to the cramer-rao bounds.
\newblock {\em IEEE Transactions on Acoustics, Speech, and Signal Processing},
  37(11):1703--1720, 1989.

\bibitem{condat2015cadzow}
Laurent Condat and Akira Hirabayashi.
\newblock Cadzow denoising upgraded: A new projection method for the recovery
  of dirac pulses from noisy linear measurements.
\newblock {\em Sampling Theory in Signal and Image Processing}, 14(1):p--17,
  2015.

\bibitem{de1992computational}
Carl De~Boor and Amos Ron.
\newblock Computational aspects of polynomial interpolation in several
  variables.
\newblock {\em Mathematics of Computation}, 58(198):705--727, 1992.

\bibitem{de1992least}
Carl De~Boor and Amos Ron.
\newblock The least solution for the polynomial interpolation problem.
\newblock {\em Mathematische Zeitschrift}, 210(1):347--378, 1992.

\bibitem{deCastro-exact2012}
Yohann De~Castro and Fabrice Gamboa.
\newblock Exact reconstruction using {Beurling} minimal extrapolation.
\newblock {\em Journal of Mathematical Analysis and applications},
  395(1):336--354, 2012.

\bibitem{de2017exact}
Yohann De~Castro, Fabrice Gamboa, Didier Henrion, and J-B Lasserre.
\newblock Exact solutions to super resolution on semi-algebraic domains in
  higher dimensions.
\newblock {\em IEEE Transactions on Information Theory}, 63(1):621--630, 2017.

\bibitem{demanet-recoverability2014}
Laurent Demanet and Nam Nguyen.
\newblock The recoverability limit for superresolution via sparsity.
\newblock {\em arXiv preprint arXiv:1502.01385}, 2015.

\bibitem{2017-denoyelle-jafa}
Quentin Denoyelle, Vincent Duval, and Gabriel Peyr{\'e}.
\newblock Support recovery for sparse super-resolution of positive measures.
\newblock {\em to appear in Journal of Fourier Analysis and Applications},
  2017.

\bibitem{donoho1992superresolution}
David~L Donoho.
\newblock Superresolution via sparsity constraints.
\newblock {\em SIAM journal on mathematical analysis}, 23(5):1309--1331, 1992.

\bibitem{donoho2006compressed}
David~L Donoho.
\newblock Compressed sensing.
\newblock {\em IEEE Transactions on information theory}, 52(4):1289--1306,
  2006.

\bibitem{donoho1992maximum}
David~L Donoho, Iain~M Johnstone, Jeffrey~C Hoch, and Alan~S Stern.
\newblock Maximum entropy and the nearly black object.
\newblock {\em Journal of the Royal Statistical Society. Series B
  (Methodological)}, pages 41--81, 1992.

\bibitem{duval2015exact}
Vincent Duval and Gabriel Peyr{\'e}.
\newblock Exact support recovery for sparse spikes deconvolution.
\newblock {\em Foundations of Computational Mathematics}, 15(5):1315--1355,
  2015.

\bibitem{2017-Duval-IP-lasso}
Vincent Duval and Gabriel Peyr{\'e}.
\newblock Sparse spikes super-resolution on thin grids {I}: the {LASSO}.
\newblock {\em Inverse Problems}, 33(5):055008, 2017.

\bibitem{fuchs2005sparsity}
Jean-Jacques Fuchs.
\newblock Sparsity and uniqueness for some specific under-determined linear
  systems.
\newblock In {\em Acoustics, Speech, and Signal Processing, 2005.
  Proceedings.(ICASSP'05). IEEE International Conference on}, volume~5, pages
  v--729. IEEE, 2005.

\bibitem{gramfort2013time}
Alexandre Gramfort, Daniel Strohmeier, Jens Haueisen, Matti~S
  H{\"a}m{\"a}l{\"a}inen, and Matthieu Kowalski.
\newblock Time-frequency mixed-norm estimates: Sparse {M/EEG} imaging with
  non-stationary source activations.
\newblock {\em NeuroImage}, 70:410--422, 2013.

\bibitem{gribonval2017compressive}
R{\'e}mi Gribonval, Gilles Blanchard, Nicolas Keriven, and Yann Traonmilin.
\newblock Compressive statistical learning with random feature moments.
\newblock {\em arXiv preprint arXiv:1706.07180}, 2017.

\bibitem{hua1990matrix}
Yingbo Hua and Tapan~K Sarkar.
\newblock Matrix pencil method for estimating parameters of exponentially
  damped/undamped sinusoids in noise.
\newblock {\em IEEE Transactions on Acoustics, Speech, and Signal Processing},
  38(5):814--824, 1990.

\bibitem{jacques2008geometrical}
Laurent Jacques and Christophe De~Vleeschouwer.
\newblock A geometrical study of matching pursuit parametrization.
\newblock {\em IEEE Transactions on Signal Processing}, 56(7):2835--2848, 2008.

\bibitem{jaggi2013revisiting}
Martin Jaggi.
\newblock Revisiting frank-wolfe: Projection-free sparse convex optimization.
\newblock In {\em ICML (1)}, pages 427--435, 2013.

\bibitem{jiang2001almost}
Tao Jiang, Nicholas~D Sidiropoulos, and Jos~MF ten Berge.
\newblock Almost-sure identifiability of multidimensional harmonic retrieval.
\newblock {\em IEEE Transactions on Signal Processing}, 49(9):1849--1859, 2001.

\bibitem{krim-two1996}
Hamid Krim and Mats Viberg.
\newblock Two decades of array signal processing research: the parametric
  approach.
\newblock {\em IEEE signal processing magazine}, 13(4):67--94, 1996.

\bibitem{kunis2016multivariate}
Stefan Kunis, Thomas Peter, Tim R{\"o}mer, and Ulrich von~der Ohe.
\newblock A multivariate generalization of {Prony's} method.
\newblock {\em Linear Algebra and its Applications}, 490:31--47, 2016.

\bibitem{levy-reconstruction1981}
Shlomo Levy and Peter~K Fullagar.
\newblock Reconstruction of a sparse spike train from a portion of its spectrum
  and application to high-resolution deconvolution.
\newblock {\em Geophysics}, 46(9):1235--1243, 1981.

\bibitem{liao-music2014}
Wenjing Liao and Albert Fannjiang.
\newblock Music for single-snapshot spectral estimation: Stability and
  super-resolution.
\newblock {\em Applied and Computational Harmonic Analysis}, 40(1):33--67,
  2016.

\bibitem{lorentz2000multivariate}
Rudolph~A Lorentz.
\newblock Multivariate hermite interpolation by algebraic polynomials: a
  survey.
\newblock {\em Journal of computational and applied mathematics},
  122(1):167--201, 2000.

\bibitem{mallat1993matching}
St{\'e}phane~G Mallat and Zhifeng Zhang.
\newblock Matching pursuits with time-frequency dictionaries.
\newblock {\em IEEE Transactions on signal processing}, 41(12):3397--3415,
  1993.

\bibitem{moitra2014threshold}
Ankur Moitra.
\newblock The threshold for super-resolution via extremal functions.
\newblock {\em arXiv preprint arXiv:1408.1681}, 2, 2014.

\bibitem{morgenshtern2016super}
Veniamin~I Morgenshtern and Emmanuel~J Candes.
\newblock Super-resolution of positive sources: The discrete setup.
\newblock {\em SIAM Journal on Imaging Sciences}, 9(1):412--444, 2016.

\bibitem{peterreconstruction}
Thomas Peter, Gerlind Plonka, and Robert Schaback.
\newblock Reconstruction of multivariate signals via {Prony's} method.
\newblock {\em Proc. Appl. Math. Mech., to appear}, 2017.

\bibitem{prony1795essai}
{Gaspard de} Prony.
\newblock Essai exp\'erimental et analytique: sur les lois de la dilatabilit\'e
  de fluides \'elastique et sur celles de la force expansive de la vapeur de
  l'alkool, \`a diff\'erentes temp\'eratures.
\newblock {\em J. de l'Ecole Polytechnique}, 1(22):24?--76, 1795.

\bibitem{roy1989esprit}
Richard Roy and Thomas Kailath.
\newblock {ESPRIT}-estimation of signal parameters via rotational invariance
  techniques.
\newblock {\em IEEE Transactions on acoustics, speech, and signal processing},
  37(7):984--995, 1989.

\bibitem{rust2006sub}
Michael~J Rust, Mark Bates, and Xiaowei Zhuang.
\newblock Sub-diffraction-limit imaging by stochastic optical reconstruction
  microscopy ({STORM}).
\newblock {\em Nature methods}, 3(10):793--795, 2006.

\bibitem{sacchini1993two}
Joseph~J Sacchini, William~M Steedly, and Randolph~L Moses.
\newblock Two-dimensional prony modeling and parameter estimation.
\newblock {\em IEEE Transactions on signal processing}, 41(11):3127--3137,
  1993.

\bibitem{santosa-linear1986}
Fadil Santosa and William~W Symes.
\newblock Linear inversion of band-limited reflection seismograms.
\newblock {\em SIAM Journal on Scientific and Statistical Computing},
  7(4):1307--1330, 1986.

\bibitem{sauer2017prony}
Tomas Sauer.
\newblock Prony's method in several variables.
\newblock {\em Numerische Mathematik}, 136(2):411--438, 2017.

\bibitem{schiebinger2015superresolution}
Geoffrey Schiebinger, Elina Robeva, and Benjamin Recht.
\newblock Superresolution without separation.
\newblock {\em arXiv preprint arXiv:1506.03144}, 2015.

\bibitem{schmidt-multiple1986}
Ralph Schmidt.
\newblock Multiple emitter location and signal parameter estimation.
\newblock {\em IEEE transactions on antennas and propagation}, 34(3):276--280,
  1986.

\bibitem{shahram2006statistical}
Morteza Shahram and Peyman Milanfar.
\newblock Statistical and information-theoretic analysis of resolution in
  imaging.
\newblock {\em IEEE Transactions on Information Theory}, 52(8):3411--3437,
  2006.

\bibitem{stoica2005spectral}
Petre Stoica, Randolph~L Moses, et~al.
\newblock {\em Spectral analysis of signals}, volume 452.
\newblock Pearson Prentice Hall Upper Saddle River, NJ, 2005.

\bibitem{tang2013atomic}
Gongguo Tang and Benjamin Recht.
\newblock Atomic decomposition of mixtures of translation-invariant signals.
\newblock {\em IEEE CAMSAP}, 2013.

\bibitem{tibshirani1996regression}
Robert Tibshirani.
\newblock Regression shrinkage and selection via the {Lasso}.
\newblock {\em Journal of the Royal Statistical Society. Series B
  (Methodological)}, pages 267--288, 1996.

\bibitem{trefethen1997numerical}
Lloyd~N Trefethen and David Bau~III.
\newblock {\em Numerical linear algebra}, volume~50.
\newblock Siam, 1997.

\end{thebibliography}

\end{document}